\documentclass[12pt,a4paper]{amsart}
\usepackage{microtype}
\usepackage[margin=2.54cm]{geometry}
\usepackage[utf8]{inputenc}
\usepackage[T1]{fontenc}
\usepackage{lmodern}
\usepackage{amsfonts}
\usepackage{epsfig}
\usepackage{graphicx}
\usepackage{breqn}
\usepackage{amsmath}
\usepackage{amssymb}
\usepackage{latexsym}
\usepackage{rotating}
\usepackage{pstricks, pst-node, pst-text, pst-3d}
\usepackage{amsbsy}
\usepackage{subfig}
\usepackage{enumerate}
\usepackage[shortlabels]{enumitem}
\usepackage{amsthm}
\usepackage{mathrsfs}
\usepackage{hyperref}
\usepackage{textcomp}
\usepackage{color}
\usepackage{verbatim}
\usepackage{subfig}
\usepackage{afterpage}
\usepackage{pdflscape}
\usepackage{rotating}
\usepackage[super]{nth}
\usepackage{nag}
\usepackage{enumitem}
\usepackage{rotating}

\allowdisplaybreaks[1]

\newtheorem{theorem}{Theorem}
\newtheorem{proposition}{Proposition}

\newtheorem{lemma}{Lemma}

\newtheorem{definition}{Definition}
\newtheorem{remark}{Remark}

\theoremstyle{remark}
\newtheorem{example}{Example}
\numberwithin{equation}{section}

\renewcommand{\Re}{\,\rm{Re}\,}
\renewcommand{\Im}{\,\rm{Im}\,}
\renewcommand{\th}{\,\rm{th}}
\newcommand{\sh}{\,\rm{sh}}
\newcommand{\ch}{\,\rm{ch}}
\renewcommand{\th}{\,\rm{th}}
\newcommand{\cth}{\,\rm{cth}}
\newcommand{\tg}{\,\rm{tg}}

\DeclareMathOperator{\id}{id}

\DeclareMathOperator{\de}{\partial}
\DeclareMathOperator{\dS}{d^{\rm{S}}}
\DeclareMathOperator{\dI}{d^{\rm{It\^{o}}}}
\DeclareMathOperator{\Lc}{\mathcal{L}}
\DeclareMathOperator{\Ac}{\mathcal{A}}
\DeclareMathOperator{\Dc}{\mathcal{D}}

\DeclareMathOperator{\D}{\mathbb{D}}
\DeclareMathOperator{\C}{\mathbb{C}}
\DeclareMathOperator{\HH}{\mathbb{H}}

\DeclareMathOperator{\SSS}{\mathbb{S}}
\DeclareMathOperator{\LL}{\mathbb{L}}
\DeclareMathOperator{\Hc}{\mathcal{H}}

\DeclareMathOperator{\K}{\mathcal{K}}

\DeclareMathOperator{\map}{\rightarrow}
\DeclareMathOperator{\then}{~\Rightarrow~}

\DeclareMathOperator{\supp}{\rm{supp }}
\newcommand{\Ev}[1]{\mathbb{E}\left[{#1}\right]}
\newcommand{\Evv}[2]{\mathbb{E}_{#1}\left[{#2}\right]}

\begin{document}
\title{Coupling of Gaussian Free field with general slit SLE}
\author{Alexey Tochin}
\author{Alexander Vasil'ev}

\address{Department of Mathematics,
University of Bergen, P.O.~Box~7803, Bergen, N-5020, Norway}
\email{alexey.tochin@gmail.com}\email{alexander.vasiliev@math.uib.no}

\thanks{The authors were  supported by EU FP7 IRSES program STREVCOMS, grant no. PIRSES-GA-2013-612669, and by the grants of the Norwegian Research Council \#239033/F20}



\begin{abstract}
We consider a coupling of the Gaussian free field with  slit holomorphic stochastic flows,  called
($\delta,\sigma$)-SLE, which contains known SLE processes (chordal, radial, and dipolar) as
particular cases. In physical terms, we study a free boundary conformal field theory with one scalar bosonic field, where Green's function is assumed to have some general regular harmonic part. We establish which of these models allow coupling with ($\delta,\sigma$)-SLE, or equivalently, when the correlation functions induce  local ($\delta,\sigma$)-SLE martingales (martingale observables).
\end{abstract}

\maketitle

\tableofcontents

\section{Introduction}

\subsection{Simple example of coupling}

The paper is focused on the coupling of the Gaussian free field (GFF) with the ($\delta,\sigma$)-SLE, and studies in detail some special cases. Let us explain briefly  the case of the chordal $\text{SLE}_4$ and its coupling with the Gaussian free field subject to the Dirichlet boundary conditions, see  \cite{Kang2011} and \cite{Schramm2005} for details. We review some known generalizations, and then, explain  the main results of the paper. 

Let 
$G_t:\HH\setminus K_t\map \HH$
be a conformal map from a subset $\HH\setminus K_t$ of the upper half-plane 
$\HH:=\{z\in \mathbb{C}: \Im{z}>0\}$ 
onto $\HH$ defined by the  initial value problem for the stochastic differential equation
\begin{equation}
 \dI G_t(z) = \frac{2}{G_t(z)}dt - \sqrt{\kappa} \dI B_t,\quad
 G_0(z) = z,\quad
 t\geq 0,\quad \kappa>0,
 \label{Formula: Chordal SLE}
\end{equation} 
where $\dI$ denotes the It\^{o} differential and $B_t$ stands for the standard Brownian motion (we avoid writing $\circ$ for the Stratonovich integral because
of possible confuse with superposition used at the same time, so we explicitly write $\dS$ or $\dI$).  
This initial value problem is known as the chordal SLE whose solution $G_t$ is a random conformal map which defines a strictly growing random set 
$K_t:=\HH\setminus G_t^{-1}(\HH)$, called the SLE hull,
($K_t\subset K_s,~s>t$)
 bounded almost surely. In particular, it is a random simple curve for $\kappa\leq 4$ that generically tends to infinity as 
$t\map \infty$. 
It turns out that the random law defined this way is related to various problems of mathematical physics.

It is straightforward to check that the following random processes are local martingales for 
$\kappa=4$
\begin{equation*}\begin{split}
 {M_1}_t(z) :=& \arg G_t(z),\\
 {M_2}_t(z,w) : =& -\log\frac{|G_t(z)-G_t(w)|}{|G_t(z)-\overline{G_t(w)}|} + \arg G_t(z) \arg G_t(w),\\
 {M_3}_t(z,w,u) : =& 
 -\log\frac{|G_t(z)-G_t(w)|}{|G_t(z)-\overline{G_t(w)}|}\arg G_t(u) 
 + \arg G_t(z) \arg G_t(w)\arg G_t(u) 
 +\\+&
  \{z \leftrightarrow w \leftrightarrow u \},\\
 \dotso &  
\end{split}\end{equation*}
for $z,w,u\in\HH$ until a stopping time  $t_{z}$, which is geometrically defined by the fact that $K_t$ touches $z$ for the first time, i.e. 
$G_t(z)$ is no longer defined for $t>t_z$. The Wick pairing structure can be recognized  in this collection of martingales. It appears in higher moments of multivariant Gaussian distributions as well as in the correlation functions in the free quantum field theory.
We consider  
the \emph{Gaussian free field} (GFF)
$\Phi$ which 
is a random real-valued functional over smooth functions in $\HH$. 
Let us use the heuristic notation `$\Phi(z)$' ($z\in\HH$) until 
a rigorous definition is given in Section 
\ref{Section: Gaussian free field}.

For some special choice of GFF the moments are  
\begin{equation*}\begin{split}
 S_1(z):=\Ev{\Phi(z)} =& \arg{z},\\
 S_2(z,w):=\Ev{\Phi(z)\Phi(w)} =& 
 -\log\frac{|z-w|}{|z-\overline{w}|} + \arg z \arg w ,\\
 S_3(z,w,u) : =\Ev{\Phi(z)\Phi(w)\Phi(u)}= & 
 -\log\frac{|z-w|}{|z-\overline{w}|}\arg u 
 + \arg z \arg w \arg u + \{z \leftrightarrow w \leftrightarrow u \},\\
 \dotso & 
\end{split}\end{equation*}
Roughly speaking, $\Phi(z)$ is a Gaussian random variable for each $z$ with the expectation 
$S_1(z)$ 
proportional to  
$\arg{z}$, and with
the covariance 
$-\log\frac{|z-w|}{|z-\overline{w}|}$. 
It can be shown that the stochastic processes  
\begin{equation}
 {M_n}_t:=S_n(G_t(z_1),G_t(z_2),\dotso G_t(z_n))
 \label{Formula: M_n = S_n(G,G...)}
\end{equation}
are martingales for all $n=1,2,\dotso$ .

This fact is closely related to the observation that the random variable 
$\Phi(G_t(z))$
agrees in law with 
$\Phi(z)$, 
where 
$\Phi$ and $G_t$ are sampled independently. It was also obtained, see \cite{Miller2012}, that the properly defined zero-level line of the random distribution $\Phi(z)$ starting at the origin in the vertical direction agrees in law with $\text{SLE}_4$.

This connection, i.e.,  coupling, can be generalized
for an arbitrary $\kappa>0$, see \cite{Sheffield2010}.
In order to explain its geometric meaning let us associate another field $J(z)$ 
of unit vectors $|J(z)|= 1$ with  $\Phi(z)$ by 
\begin{equation*}
 J(z): = e^{i \Phi(z)/\chi}
\end{equation*}
for some $\chi>0$.
A unit vector field $J$ transforms from a domain $D$ to  $D\setminus K$ according to the rule
\begin{equation*}
 J(z)\map \tilde J(\tilde z) = \sqrt{ 
  \frac{G'(\tilde z)} {\overline{ G'(\tilde z)} } 
 } J(G(\tilde z)),
\end{equation*}  
where $G$ is the conformal map $G:D\setminus K \map D$ (the coordinate transformation).
Thereby, $\Phi(z)$ transforms as
\begin{equation}
 \Phi(z) \map \tilde \Phi(\tilde z) = \Phi(G(\tilde z)) 
  - \chi \, \rm{arg}\, {G}'(\tilde z).
 \label{Formula: Phi ->Phi - chi arg}
\end{equation}
The coupling  for $\kappa>0$ is the agreement in law of   
$\Phi(G_t(z))-\chi \arg G'(z)$
with
$\Phi(z)$, 
where
\begin{equation}
 \chi = 
  \frac{2}{\sqrt{\kappa}}-\frac{\sqrt{\kappa}}{2}.
 \label{Formula: chi = 2/k - k/2}
\end{equation}
Besides,  the flow line of 
$J(z)$ 
starting at the origin agrees in law with the SLE${}_{\kappa}$ curve.

The  statement above can be extended from 
the chordal equation 
(\ref{Formula: Chordal SLE})
to the radial equation, see \cite{Bauer2004,Kang2012a}, or
 to the dipolar one, see \cite{Bauer2004a,Kang,Kanga}, see also \cite{Izyurov2010} for the SLE-($\kappa,\rho$) case.

In \cite{Ivanov2014, IvanovKangVas}, we  considered slit holomorphic stochastic flows, called in this paper $(\delta,\sigma)$-SLE, 
which contain all the above mentioned SLEs  as special cases except SLE-($\kappa,\rho$). 
In the present paper, we consider a general case of the coupling of  $(\delta,\sigma)$-SLE with GFF. In particular, we study the known cases of the couplings in a systematic way as well as consider some new ones. 

There is also another type of coupling  described in \cite{Sheffield2010} for the chordal case. The usual forward L\"owner equation is replaced by its reverse version with the opposite sign at the drift term proportional to $dt$ in
(\ref{Formula: Chordal SLE}).
The corresponding solution is a conformal map
$G_t: \HH\map\HH\setminus K_t$,
and 
$\Phi(z)$ agrees in law with
$\Phi(G_t(z))+\gamma \log|G_t'(z)|$
for some real $\kappa$-dependent constant $\gamma$.

In addition to the coupling, there are also interesting connections to other aspects of conformal field theory such as the highest weight representations of the Virasoro algebra \cite{Bauer2004b}, and the vertex algebra \cite{Kang2011}.
On the other hand, the crossing probabilities, such as touching the real line by an SLE curve, are connected with the CFT stress tensor correlation functions, see \cite{Bauer2004b,Gruzberg2006}. Both of these directions, as well as  the reverse coupling, exceed the scope of this paper.

\subsection{($\delta,\sigma$)-SLE overview}

The ($\delta,\sigma$)-SLE is a unification of the well-known chordal, radial, and dipolar stochastic L\"owner equations described first in \cite{Ivanov2014}. We will give a rigorous definition in Section \ref{Section: Preliminaries} and will use a simplified version in Introduction. Let 
$D\subset\C$ is a simply connected hyperbolic domain. A ($\delta,\sigma$)-SLE or a \emph{slit holomorphic stochastic flow} is the solution to the stochastic differential equation
\begin{equation}
 \dS G_t(z) = \delta(G_t(z))dt + \sigma(G_t(z)) \dS B_t,\quad
 G_0(z)=z, \quad z\in D, \quad t\geq 0,
 \label{Formula: SLE in S form}
\end{equation}
where $B_t$ is the standard Brownian motion,
$\sigma\colon D\map \C$ and $\delta \colon D\map \C$ are some fixed holomorphic vector fields defined in such a way that
the solution $G_t$ is always a conformal map $G_t\colon D\setminus K_t\map D$ for some random curve-generated subset $K_t$.  The differential $\dS$ is the Stratonovich differential, which is more convenient in our setup than the more frequently used It\^{o} differential $\dI$, see for example \cite{Gardiner1982}. The same equation in the It\^{o} form is
\begin{equation}\begin{split}
 &\dI G_t(z) = \left(\delta(G_t(z))+\frac12 \sigma'(G_t(z))\sigma(G_t(z))\right)dt + \sigma(G_t(z)) \dI B_t,\\
 &G_0(z)=z,
 \quad t\geq 0,
 \label{Formula: SLE in It\^{o} form}
\end{split}\end{equation}
see Appendix 
\ref{Appendix: Some relations from stochastic calculus}
for details. Equation (\ref{Formula: SLE in It\^{o} form}) can be easily reformulated for any other domain $\tilde D\subset\C$ using the transition function 
$\tau\colon \tilde D\map D$. Define 
$\tilde G_t := \tau^{-1} \circ G_t \circ \tau$. The equation for $\tilde G_t$ is of the same form as 
(\ref{Formula: SLE in S form}), 
but with new fields $ \tilde\delta$ and $ \tilde\sigma$ defined as
\begin{equation*}\begin{split}
 \tilde \delta(\tilde z) = \frac{1}{\tau'(\tilde z)} \delta(\tau(\tilde z)),\quad
 \tilde \sigma(\tilde z) = \frac{1}{\tau'(\tilde z)} \sigma(\tau(\tilde z)).
\end{split}\end{equation*} 
The general form of $\delta(z)$ and $\sigma(z)$ for $D=\HH$ (the upper half-plane) is
\begin{equation}\begin{split}
 &\delta^{\HH}(z) = \frac{\delta_{-2}}{z} + \delta_{-1} + \delta_{0} z + \delta_{1} z^2, \\
 &\sigma^{\HH}(z) =  \sigma_{-1} + \sigma_{0} z + \sigma_{1} z^2, \\
 &\delta_{-1},~\delta_0,~\delta_1,~\sigma_{-1},~\sigma_0,~\sigma_1\in\mathbb{R},~\quad
 \sigma_{-1}\neq 0,\quad
 \delta_{-2}>0.
 \label{Formula: allowed delta and sigma}
\end{split}\end{equation}
The values of $\delta$ and $\sigma$ for the classical SLEs (in the upper half-plane) are summarized in the following table.
\bigskip

\begin{center}
\begin{tabular}{|c|c|c|}
 \hline 
 SLE type & $\delta^{\HH}(z)$ & $\sigma^{\HH}(z)$ \\ 
 \hline 
 Chordal & $2/z$ & $-\sqrt{\kappa}$ \\ 
 \hline 
 Radial & $1/(2z)+z/2 $ & $-\sqrt{\kappa}(1+z^2)/2$ \\ 
 \hline 
 Dipoloar & $1/(2z)-z/2 $ & $-\sqrt{\kappa}(1-z^2)/2$ \\ 
 \hline 
 ABP SLE, see \cite{Ivanov2012a}  & $1/(2z)$ & $-\sqrt{\kappa}(1+z^2)/2$ \\ 
 \hline 
\end{tabular} 
\end{center}
\bigskip

We use the half-plane formulation in this table just for simplicity. In order to obtain the commonly used form of the radial equation  in the unit disk
one  applies  the transition map
\begin{equation}
 \tau_{\HH,\D}(z):=i \frac{1-z}{1+z},\quad \tau_{\HH,\D} :\D \map\HH,
 \label{Formula: tau_H D}
\end{equation}
for the unit disk
$\D:=\{z\in \C\colon |z|<1\}$. 
For more details, see Example 
\ref{Example: Radial Leowner equation} 
and Section 
\ref{Section: Radial SLE with drift}.
The same procedure with the transition map
\begin{equation}
 \tau_{\HH,\mathbb{S}}(z)=\rm{th}\frac{z}{2},\quad \tau_{\HH,\SSS}:\mathbb{S} \map\HH,
 \label{Formula: tau_H S}
\end{equation}
for the strip $\SSS:=\{z\in\C~:~0< \Im z < \pi\}$,
gives the common form for the dipolar SLE, see also Section \ref{Section: Dipolar SLE with drift}.

It was shown in \cite{Ivanov2014}, that the choice 
of $\delta$ and $\sigma$ given by
(\ref{Formula: allowed delta and sigma}) guarantees that
the random set $K_t$ is curve-generated
similarly to the standard known $\rm{SLE}_{\kappa}$. Moreover, $K_t$ has the same local behaviour (the fractal dimension, dilute phases, and etc.). 

The main difference between the classical SLEs and the case of the general ($\delta,\sigma$)-SLE is the absence of  fixed normalization points, in general, e.g., in the classical cases the solution $G_t$ to the radial SLE equation always fixes an interior point ($0$ in the unit disk formulation), 
the dipolar SLE preserves two boundary points  
($-\infty$ and $+\infty$ in the strip coordinates $\SSS$), 
the chordal equation is normalized at the boundary point at infinity in the 
half-plane formulation.
The chordal case can be considered as the limiting case of the dipolar one as  both fixed boundary points collide, or the limiting case of the radial equation when the fixed interior point approaches  the boundary.
In the general case of ($\delta,\sigma$)-SLE, there is no such a simple normalization. Numerical experiments show that the curve (or the curve-generated hull for $\kappa\geq 4$) $K_t$ tends to a random point inside the disk as $t\to\infty$. A version of the domain Markov property still holds due to an autonomous form of equation 
(\ref{Formula: SLE in It\^{o} form}).

\subsection{Overview and purpose of this paper}

We consider the general form of 2-dimensional GFF $\Phi$ with some expectation
\begin{equation*}
	\eta(z):=\Ev{\Phi(z)},
\end{equation*}
and the covariance
\begin{equation*}
	\Gamma(z,w):=\Ev{\Phi(z)\Phi(w)}-\Ev{\Phi(z)}\Ev{\Phi(w)}.
\end{equation*}
We postulate that $\Gamma(z,w)$ is symmetric fundamental solution to the Laplace equation,
\begin{equation*}
	\triangle_z \Gamma(z,w) = 2\pi \delta(z,w),\quad \Gamma(z,w)=\Gamma(w,z), 
\end{equation*}
but  imposing no boundary conditions. In other words, $\Gamma(z,w)$ has the from
\begin{equation*}
	\Gamma(z,w)= -\frac12 \log|z-w|~+~\text{symmetric harmonic function of $z$ and $w$}.
\end{equation*}
We address the following question: which of so defined GFF can be coupled with some ($\delta,\sigma$)-SLE in the sense that $\Phi(z)$ and $\Phi(G_t(z))$ agree in law? In particular, we show 
(Theorem \ref{Theorem: The coupling theorem}) 
that this is possible if and only if the following system of partial differential equation is satisfied 
\begin{equation}
\begin{array}{c}
	\left( \Lc_{\delta}  + \frac12\Lc_{\sigma}^2 \right) \eta(z) = 0, \\
	\Lc_{\delta} \Gamma(z,w) + \Lc_{\sigma} \eta(z) \Lc_{\sigma} \eta(w) = 0, \\
	\Lc_{\sigma} \Gamma(z,w) = 0,
\end{array}
	\label{Formula: The main system}
\end{equation}
where,
following \cite{Kang2011},
 we understand under $\Lc$ a generalized version of the Lie derivatives defined by
\begin{equation*}
\begin{array}{c}
	\Lc_v \eta(z) := 
	v(z)\de_{z} + \overline{v(z)} \de_{\bar z} + \chi\, \rm{Im}\, v'(z),\\
	\Lc_v \Gamma(z,w) := 
	v(z)\de_{z} + \overline{v(z)} \de_{\bar z} 
	+ v(w)\de_{w} + \overline{v(w)} \de_{\bar w}.
\end{array}
\end{equation*}
Here, $v=\delta$ or $v=\sigma$ and $\chi$ refers to the SLE parameter $\kappa$ by
\eqref{Formula: chi = 2/k - k/2}. These equations generalize the corresponding relation from 
\cite{Sheffield2010}, see
the table on page 45, and they are versions of
(M-cond') and (C-cond') 
from \cite{Izyurov2010}.
The second equation 
in \eqref{Formula: The main system}
is known as a version of Hadamard's variation formula, and the third 
states that the covariance 
$\Gamma$ must be invariant with respect to the M\"obius automorphisms generated by the vector field $\sigma$.

The paper is a continuation and extension of \cite{IvanovKangVas} and is organized as follows.
The rest of Introduction provides some remarks about the relations between equations
\eqref{Formula: The main system}
and the BPZ equation considered first in 
\cite{Belavin1984}. 
In Section 
\ref{Section: Preliminaries},
we give all necessary  definitions  for ($\delta,\sigma$)-SLE as well as a very basic definitions and properties of 
GFF that we will need in what follows.
Further, in Section \ref{Section: Coupling between SLE and GFF}, we prove the coupling Theorem 
\ref{Theorem: The coupling theorem}.
It states that for any $(\delta,\sigma$)-SLE
a proper pushforward of $S_n$ denoted by
${G_t^{-1}}_* S_n$ 
(a generalisation of (\ref{Formula: M_n = S_n(G,G...)}) with an additional $\chi$-term similar to 
(\ref{Formula: Phi ->Phi - chi arg})) 
is a local martingale if and only if the 
system 
\eqref{Formula: The main system}
is satisfied for given $\delta$, $\sigma$, $\Gamma$, and $\eta$.
The same theorem also states that  both: the local martingale property and the system of equations are equivalent to a \emph{local coupling} which is a weaker version of the coupling from \cite{Sheffield2010} discussed above. We expect, but do not consider in this paper, that the local coupling leads to the same property of the flow lines of 
$e^{i\Phi(z)/\chi}$ to agree in law with the 
($\delta,\sigma$)-SLE curves.

The general solution to the system 
\eqref{Formula: The main system}
gives all possible ways to couple $(\delta,\sigma)$-SLEs with the GFF at least in the frameworks of our assumptions of the pre-pre-Schwarzian behaviour of $\eta$ and of the scalar behaviour of $\Gamma$. 

In Section 
\ref{Section: The coupling in case of Dirichlet boundary condition of the covariance function},
we assume the simplest choice of the  Dirichlet boundary conditions for $\Gamma$ and study all 
($\delta,\sigma$)-SLEs that can be coupled. It turns out that only the classical SLEs with drift are allowed plus some exceptionals cases, 
see Sections
\ref{Section: Chordal SLE with fixed time reparametrization}
and
\ref{Section: SLE with one fixed boundary point},
that define the same measure as for the chordal SLE up to  time reparametrization 
Observe however, that the coupling with the classical SLEs with drift has not been considered in the literature so far.

Further, in Sections 
\ref{Section: Coupling to GFF with Dirichlet-Neumann boundary condition}
and
\ref{Section: Coupling to twisted GFF},
we assume less trivial choices of $\Gamma$ and obtain that among all ($\delta,\sigma$)-SLEs only the dipolar and radial SLEs with $\kappa=4$ can be coupled. The first case was considered in \cite{Kanga}. The second one requires a construction of a specific ramified GFF 
$\Phi_{\text{twisted}}$,
that changes its sign to the opposite while being turned once around some interior point in $D$. 
This construction was considered before and it is called the `twisted CFT' as we were informed by Nam-Gyu Kang, see \cite{Kang_preparation}.

\subsection{Nature of  coupling}

One of the approaches to conformal field theory (CFT) boils down to consideration of a probability measure in a space of function in $D$. The simplest choice is GFF. The chordal SLE/CFT correspondence is revealed in, for example, 
\cite{Friedrich2003}
and
\cite{Kang2011}. Here, we extend this treatments to a $(\delta,\sigma)$-SLE/CFT correspondence. 
This section is less important for the general objective of the paper and is dedicated mostly to a reader with physical background. 

We consider a boundary CFT (BCFT) defined oin a domain $D\subset\C$ with one free scalar bosonic field 
$\Phi$. 
One of the standard approach to define a quantum field theory is the heuristic functional integral formulation with the classical action $I[\Phi]$, see for instance, \cite{Faddeev1991}. 
The following triple definition of the Schwinger functions 
$S(z_1,z_2,\dotso z_n)$
manifests the relation between the probabilistic notations, functional integral formulation, and the operator approach.
\begin{equation}\begin{split}
 S_n(z_1,z_2,\dotso z_n):=& 
 \Ev{\Phi(z_1)\Phi(z_2)\dotso \Phi(z_n)} 
 =\\=&
 c^{-1}\int
 \Phi(z_1)\Phi(z_2)\dotso \Phi(z_n) e^{-I[\Phi]} \mathfrak{D}\Phi
 =\\=&
 \langle| 
 \mathbb{T} 
 \left[ \hat \Psi(a) \hat \Phi(z_1) \hat \Phi(z_2) \dotso \hat \Phi(z_n) \right] 
 |\rangle,\\
 & z_i\in D, \quad i=1,2,\dotso n,\quad n=1,2,\dotso
\end{split}\label{heuristic}\end{equation}
Here in the second term,  `$|\rangle$' is the vacuum state, $\hat \Phi(z)$ the primary operator field, 
$\hat\Psi$ a certain operator field taken at a boundary point $a\in D$,  
and $\mathbb{T}[\dotso]$ is the time-ordering, which is often dropped in the physical literature, we refer to \cite{Schottenloher2008} for details.
The second string in \eqref{heuristic} contains a heuristic integral over some space of functions $\Phi(z)$ on $D$, which corresponds to the operator $\hat \Phi(z)$. 

The first string in \eqref{heuristic} is a mathematically precise formulation of the second one. The expectation $\Ev{\dotso}$ can be understood as the Lebesgue integral over the space $\mathcal{D}'(D)$ of linear functionals over smooth functions in $D$ with compact support.
The expression $e^{-I[\Phi]} \mathfrak{D}\Phi$ can  be in its turn understood as the 
differential w.r.t. the measure.

We emphasize here that the correlation functions are not completely defined by the action $I[\Phi]$ because one has to specify also the space of functions the integral is taken over and the measure. For instance, the Euclidean free field action 
\begin{equation*}
	I[\Phi]=\frac12 \int_{D} \de \Phi(z) \bar \de \Phi(z) d^2 z
\end{equation*}
defines only the singular part of the 2-point correlation function, which is $-(2\pi)^{-1}\log |z|$. 
To illustrate the statement above let us consider the following example. Let the expectation of $\Phi$ be identically zero,
\begin{equation*}
	S_1(z) = 
	\int\limits
	\Phi(z) e^{-I[\Phi]} \mathfrak{D}\Phi  
	= 0, \quad z\in D.
\end{equation*}
Depending on the choice of the space of functions $\Hc_s'\ni\Phi$ and of the measure on it, the 2-point correlation function 
\begin{equation*}
	\Gamma(z_1,z_2) =
	\int\limits
	\Phi(z_1)\Phi(z_2) e^{-I[\Phi]} \mathfrak{D} \Phi 
\end{equation*}
varies. For example in \cite{Kang2011,Kang2012a,Kang,Miller2012},
$\Gamma$ was assumed to possess the Dirichlet boundary conditions,
see (\ref{Formula: Gamma_D = Log...}). But in 
\cite{Kanga}, $\Gamma$ vanishes only on a part of the boundary and on the other part the boundary conditions are Neumann's. In \cite{Izyurov2010}, even more general boundary conditions were considered.

In fact, under CFT one can understand a probability measure on the space of functions 
$\Hc_s'$ in $D$, which is just a version of $\mathcal{D}'(D)$.
In the case of the free field, the Schwinger functions (equivalently, correlation functions or moments) $S_n$ are of the form 
\begin{equation*}\begin{split}
 S_1[z_1]=&\eta[z_1],\\
 S_2[z_1,f_2]=&\Gamma[f_1,f_2] + \eta[f_1]\eta[f_2],\\
 S_3[z_1,f_2,f_3]=&\Gamma[f_1,f_2]\eta[f_3] + 
  \Gamma[z_3,z_1] \eta[f_2] + 
  \Gamma[z_2,z_3] \eta[f_1]+
  \eta[f_1] \eta[z_2] \eta[f_3],\\
 S_4[f_1,f_2,f_3,f_4]=&
 \Gamma[z_1,z_2]\Gamma[z_3,z_4]+\Gamma[f_1,f_3]\Gamma[f_2,f_4]+\Gamma[f_1,f_4]\Gamma[z_2,z_3]
 +\\+&
 \Gamma[z_1,z_2]\eta(z_3)\eta[z_4]+\Gamma[z_1,z_3]\eta[z_2]\eta[z_4]+\Gamma[z_1,z_4]\eta[f_2]\eta[z_3]
 +\\+&
 \eta[z_1]\eta[z_2]\eta[z_3]\eta[z_4],\\
 \dotso&
\end{split}\end{equation*}
In this case, the measure 
is characterized by the space $\Hc_s'$ equipped with certain topology,
by  the field expectation
$\eta(z)=S_1(z)$, and
by Green's function (covariance) 
$\Gamma(z,w)=S_2(z,w)-S_1(z)S_1(w)$. 

Assume now that 
\begin{equation}
	\Lc S_n(z_1,z_2,\dotso,z_n) = 0,\quad n=1,2,\dotso
	\label{Formula: L S_n = 0}
\end{equation}
for some first order differential operator $\Lc$. Then, in the language of quantum field theory, the functions $S_n$ are called `correlation functions' and the relations 
\eqref{Formula: L S_n = 0}
can be called the Ward identities. This identities usually correspond to invariance of $S_n$ with respect to some Lie group.
For example, if 
\begin{equation}
	\Lc:= \sum\limits_{i=1}^{\infty}
	\sigma(z_i)\de_{z_i} + \overline{\sigma(z_i)} \de_{\bar z_i}
	\label{Formual: Lc := sigma de + ...}
\end{equation}
for a vector field $\sigma$ of the form
\eqref{Formula: allowed delta and sigma}, 
then \eqref{Formula: L S_n = 0} is equivalent to
\begin{equation*}
	S_n(z_1,z_2,\dotso,z_n) = S_n(H_s[\sigma](z_1),H_s[\sigma](z_2),\dotso,H_s[\sigma](z_n))
	,\quad s\in\mathbb{R},
\end{equation*}
where $\{H_t[\sigma]\}_{t\in\mathbb{R}}$ is a one parametric Lie group of M\"obious automorphisms 
$H_t[\sigma]:D\map D$
induced by $\sigma$:
\begin{equation}
	d H_t[\sigma](z) = \sigma(H_t[\sigma](z))dt,\quad z\in D,\quad t\in\mathbb{R},\quad 
	H_0[\sigma](z)=z,
	\label{Formula: dG = sigma(G) dt}
\end{equation}
because
\begin{equation}
	\left.
		d S_n(H_s[\sigma](z_1),H_s[\sigma](z_2),\dotso,H_s[\sigma](z_n)) 
	\right|_{s=0}	
	= 
	\Lc S_n(z_1,z_2,\dotso,z_n) ds.	
	\label{Formula: dS(H) = L S ds}
\end{equation}

The relations 
\eqref{Formula: L S_n = 0}
with
\eqref{Formual: Lc := sigma de + ...}
are satisfied if and only if
\begin{equation}\begin{split}
	&\left(
		\sigma(z)\de_{z} + \overline{\sigma(z)} \de_{\bar z}
	\right)
	\eta(z) = 0,\\
	&\left(
		\sigma(z)\de_{z} + \overline{\sigma(z)} \de_{\bar z} 
		+ \sigma(w)\de_{w} + \overline{\sigma(w)} \de_{\bar w} 
	\right) 
	\Gamma(z,w)=0.
	\label{Formula: formula 1_4}
\end{split}\end{equation}

Consider now the vector field $\delta$ in  place of $\sigma$. It
is still holomorphic but generates a semigroup 
$\{H_t[\sigma]\}_{t\in(-\infty,0]}$ 
of endomorphisms
$G_t:D\map D\setminus \gamma_t$ ($\gamma_t$ is some growing curve in $D$)
instead of the group of automorphisms of $D$. 
Then the second identity in
\eqref{Formula: formula 1_4}
does not hold because
\begin{equation*}\begin{split}
	&\left( 
		\delta(z) \de_z + \overline{\delta(z)} \de_{\bar z} 
		+ \delta(w) \de_w + \overline{\delta(w)} \de_{\bar w} 
	\right) \Gamma(z,w) \neq 0.
\end{split}\end{equation*}
Geometrically this can be explained by the fact that $H_t[\delta]$ is not just a change of coordinates but also it necessarily shrinks the domain $D$. 

Let now $G_t$ satisfy 
\eqref{Formula: SLE in S form},
which is just a stochastic version of
\eqref{Formula: dG = sigma(G) dt},
where $v(z)dt$ is replaced by $\delta(z)dt+\sigma(z) \dS B_t$.
The vector field $\delta$ induces endomorphisms, $\sigma$ induces automorphisms, and the multiple of 
$\dS B_t$ can be understood as the white noise. 
Such a stochastic change of coordinates $G_t$ in the infinitesimal form, according to the It\^{o} formula, leads to the second order differential operator 
$\delta(z)\de_z + \frac12 (\sigma(z)\de_z)^2$ 
instead of the first order Lie derivative $v(z)\de_z$.

The first two Ward identities become
\begin{equation}\begin{split}
	&\left( 
		\delta(z)\de_z + \frac12 (\sigma(z)\de_z)^2 
		+ \text{complex conjugate}
	\right) 
		\eta (z) = 0,\\
	&\left(
		\delta(z)\de_z + \delta(w)\de_w +
		\frac12 \left( \sigma(z)\de_z + \sigma(w)\de_w \right)^2
		+ \text{complex conjugate}
	\right)
	\times\\ \times&
	\left(\Gamma(z,w) + \eta(z)\eta(w)\right) = 0,
	\label{Formula: delta + 1/2 sigma^2 eta = 0, (delta + 1/2 sigma^2 eta)(Gamma + eta eta) = 0}
\end{split}\end{equation}
which is equivalent to 
\eqref{Formula: The main system}
if $\chi=0$.
Due to a version of the It\^o formula  we have
\begin{equation*}\begin{split}
	&\left.
		\dI S_n(G_s(z_1),G_s(z_2),\dotso,G_s(z_n)) 
	\right|_{s=0}	
	=\\=& 
	\left( \Lc_{\sigma} + \frac12 \Lc_{\sigma}^2 \right) S_n(z_1,z_2,\dotso,z_n) dt +
	\Lc_{\sigma} S_n(z_1,z_2,\dotso,z_n) \dI B_t,	
\end{split}\end{equation*}
which is an analog to the relation 
\eqref{Formula: dS(H) = L S ds}.
In other words, the system
\eqref{Formula: The main system} represents
the local martingale conditions.

For the case 
\begin{equation}
	\delta(z)=\frac{2}{z}, \quad 
	\sigma(z)=-\sqrt{\kappa},\quad \kappa>0,
	\quad z\in D=\HH,
	\label{Formula: delta=2/z, sigma=-kappa^1/2}
\end{equation}
the identities 
\eqref{Formula: delta + 1/2 sigma^2 eta = 0, (delta + 1/2 sigma^2 eta)(Gamma + eta eta) = 0}
is an analog to the BPZ equation (5.17) in \cite{Belavin1984}. 
The choice of
\eqref{Formula: delta=2/z, sigma=-kappa^1/2}
corresponds to the chordal SLE
(see Example \ref{Example: Chordal Loewner equation})
and it was considered first in
\cite{Friedrich2003},
and later in \cite{Kang2011}.
In this paper, we assume that the field $\delta$
is holomorphic, has a simple pole of positive residue at a boundary point $a\in\de D$, and tangent at the rest of the boundary.

\section{Preliminaries}
\label{Section: Preliminaries}

Each version of the L\"{o}wner equations and holomorphic flows are usually associated with a certain canonical domain 
$D\subset\mathbb{C}$
in the complex plane specifying fixed interior or boundary points, for example, in the case of the upper half-plane, the unit disk, etc. It is always possible to map these domains one to another if necessary.
In this paper, we focus on conformally invariant properties, i.e., those which are not related to a specific choice of the canonical domain.
This can be achieved by considering a generic hyperbolic simply connected domain and conformally invariant structures from the very beginning. 
It could be also natural to go further and work with a simply connected hyperbolic Riemann surface $\Dc$ (with the boundary $\de\Dc$). In what follows, $\Dc$ is thought of as a generic domain with a well-defined boundary or a Riemann surface.

We  use mostly global chart maps 
$\psi\colon\Dc\map D^{\psi}\subset \mathbb{C}$ from 
\index{chart map $\psi$}
$\Dc$ to a domain of the complex plane, 
writing $\psi$ for a chart 
$(D,\psi)$ for simplicity.
The charts $\psi_{\HH}:\Dc\map\HH$ corresponding to the upper half-plane or 
$\psi_{\D}:\Dc\map\D$ to the unit disk are related by
\eqref{Formula: tau_H D}. 
Another example is a multivalued map 
$\psi_{\mathbb{L}}:\Dc\map\HH$
onto the upper half-plane related to $\psi_{\D}$ by
\eqref{Formula: tau_D LL = ...}.
A single-valued branch of $\psi_{\mathbb{L}}:\Dc\map\HH$ is not a global chart map.

\subsection{Vector fields and coordinate transform}
\label{Section: Vector fields and coordinate transform}

Consider a holomorphic vector field $v$ on $\Dc$, which can be defined as a holomorphic section of the tangent bundle. 
We also can define it
as a map 
$\psi\mapsto v^{\psi}$ 
from the set of all possible global chart maps 
$\psi:\Dc\map D^{\psi}$ 
to the set of holomorphic functions 
$v^{\psi}:D^{\psi}\map\C$
defined in the image of these maps $D^{\psi}:=\psi(\Dc)$.
For  the vector fields,  the following coordinate change holds. For any 
$\tilde \psi:\Dc\map D^{\tilde \psi}$,
\begin{equation}
 v^{\tilde \psi}(\tilde z) = 
 \frac{1}  { \tau'(\tilde z)} 
 v^{\psi}(\tau(\tilde z)),\quad
 \tau :=\psi\circ\tilde\psi^{-1},\quad
 \tilde z \in \tilde D.
 \label{Formula: tilde v = 1/dtau v(tau)}
\end{equation}

We consider  a conformal homeomorphism 
$G:\Dc\setminus\K \map \Dc$ (inverse endomorphism) that
in a given chart $\psi$  is given as 
\begin{equation*}
	G^{\psi}:= \psi \circ G \circ \psi^{-1}:
	D^{\psi}\setminus K^{\psi} \map D^{\psi}. 
\end{equation*}
A vector field in other words is a $(-1,0)$-differential and the \emph{pushforward of a vector field} 
$G_*^{-1}:v^{\psi} \mapsto \tilde v^{\tilde \psi}$ 
is defined
by the rule
\begin{equation}\begin{split}
 & G_*^{-1} v^{\psi}(z) := 
 v^{\psi \circ G}(z) = 
 v^{\psi \circ G \circ \psi^{-1} \circ \psi}(z) =
 v^{G^{\psi} \circ \psi}(z) 
 =\\=&
 \frac{1}{\left(G^{\psi}\right)'(z)} 
 v^{\psi}\left(G^{\psi}(z)\right),\quad 
 z\in D^{\psi},
 \label{Formula: G v = 1/G' v(G)}
\end{split}\end{equation} 
thereby $G^{\psi}$ plays the role of $\tau$.

It is easy to see that for any two given maps $G$ and $\tilde G$ as above
\begin{equation*}
 \tilde G_*^{-1} G_*^{-1} = \left( {\tilde G}^{-1} \circ G^{-1} \right)_*, 
\end{equation*}
which motivates the upper index $-1$ in the definition of $G_*$, because in this case we are working with a left module.

The pushforward map $G_*^{-1}$ also can  be obtained as a map from the tangent space to $\Dc$ induced by $G$. 
We follow the way above because it can be generalized then to sections of tangent and cotangent spaces and their tensor products, see Section
\ref{Section: Pre-pre-Schwarzian}.

Let $v_t$ and $\tilde v_t$ be two holomorphic vector fields depending on time
continuously such as the following differential equations has continuously
differentiable solutions $F_t$ and $\tilde F_t$ in some time interval
\begin{equation*}\begin{split}
	&\dot F_t = v_t \circ F_t,\\
	&\dot {\tilde F}_t = \tilde v_t \circ \tilde F_t.
\end{split}\end{equation*}
Then, we can conclude that
\begin{equation*}
	\frac{\de}{\de t}(F_t \circ {{\tilde F}_t}) = 
	\left(
		v_t + {F_t}_* \tilde F_t{}_* \tilde v_t
	\right) \circ F_t \circ \tilde F_t
\end{equation*}
\begin{equation*}
	\dot F_t^{-1}  = 
	- \left( {F_t^{-1}}_* v_t \right) \circ F_t^{-1}. 
\end{equation*}
in the same $t$-interval and in the region of $\Dc$ where $F_t$ and $\tilde F_t$ are
defined.
The latter relation can be reformulated in a fixed chart $\psi$ as
\begin{equation*}
	\dot {F_t^{-1}}^{\psi}(z)  = 
	- \left(\left( {F_t^{-1}}\right)^{\psi}\right)'(z)~  v^{\psi}_t(z). 
\end{equation*}

\subsection{($\delta,\sigma$)-SLE basics.}
Here we  repeat briefly some necessary material about the slit holomorphic stochastic flows
SLE($\delta,\sigma$) considered first in the paper \cite{Ivanov2014} that we advice to follow for more details.

A holomorphic  vector field $\sigma$ on $\Dc$ is called \emph{complete} if the solution 
$H_s[\sigma]^{\psi}(z)$ 
of the initial valued problem 
\begin{equation}
 \dot H_s[\sigma]^{\psi}(z) = \sigma^{\psi}(H_s[\sigma]^{\psi}(z)),\quad 
 H_0[\sigma]^{\psi}(z) = z,\quad 
 z\in D^{\psi}
 \label{Formula: d G = sigma G ds in a chart}
\end{equation}
is defined for
$s\in(-\infty,\infty)$. The solution $H_s[\sigma]^{\psi}:D^{\psi}\map D^{\psi}$ 
 is 
 a conformal automorphism of 
$D^{\psi}=\psi(\Dc)$. 
Here and further on, we denote the partial derivative with respect to $s$ as
$\dot H_s:=\frac{\de}{\de s} H_s$. 
It is straightforward to see that the differential equation is of the same form in any chart $\psi$. So it is  reasonable to drop the index $\psi$ and variable $z$ as follows
\begin{equation}
 \dot H_s[\sigma] = \sigma \circ H_s[\sigma] ,\quad 
 H_0[\sigma] = \id.
 \label{Formula: d H = sigma H ds}
\end{equation}
Let $\psi_{\HH}$ be a chart map onto the upper half-plane
$D^{\psi}=\HH$. 
We denote $X^{\HH}:=X^{\psi_{\HH}}$ if $X$ is a vector field, a conformal map, or a pre-pre-Schwarzian form (defined below) on $\Dc$. 

A complete vector field $\sigma$ in the half-plane chart admits the form
\begin{equation}
 \sigma^{\HH}(z) = \sigma_{-1} + \sigma_0 z + \sigma_1 z^2,\quad z\in \HH,\quad
 \sigma_{-1}, \sigma_0, \sigma_1 \in \mathbb{R}.
 \label{Formula: sigma with 1 slit}
\end{equation}

A vector field $\delta$ is called \emph{antisemicomplicate } if the initial valued problem
\begin{equation*}
 \dot H_s[\delta] = \delta \circ H_s[\delta] ,\quad 
 H_0[\delta] = \id.
\end{equation*}
has a solution $H_t[\sigma]$, which is a conformal map 
$H_t[\sigma]:\Dc\setminus \mathcal{K}_t\map\Dc$ for all 
$t\in[0,+\infty)$ for some family of subsets $\mathcal{K}_t\subset \Dc$.

The linear space of all antisemicomplete fields is essentially bigger and infinite-dimensional, but we restrict ourselves to the from
\begin{equation}
 \delta^{\HH}(z) = \frac{\delta_{-2}}{z} + \delta_{-1} + \delta_0 z + \delta_1 z^2,\quad z\in \HH,\quad
 \delta_{-1}, \delta_0, \delta_1 \in \mathbb{R},\quad \delta_{-2}>0,
 \label{Formula: delta with 1 slit}
\end{equation}
which guarantees that the set
$\mathcal{K}_t$
is curve-generated, see \cite{Ivanov2014}.
The first term gives just a simple pole in the boundary at the origin with a positive residue. The sum of the last three terms is just a complete field.  

Let us consider first a continuously differentiable \emph{driving function} function
$u_t:t\mapsto \mathbb{R}$ for  motivation. 
Then the solution $G_t$ of the initial value problem 
\begin{equation}
 \dot G_t = \delta \circ G_t + \dot u_t\, \sigma \circ G_t ,\quad 
 G_0 = \id,\quad t\geq 0,
 \label{Formula: d G = delta G dt + sigma G du}
\end{equation}
is a family of conformal maps $G_t:\Dc\setminus \mathcal{K}_t\map\Dc$, 
where the family of subsets $\{\mathcal{K}_t\}_{t\geq 0}$ depends on the driving function $u$. To avoid the requirement of continuous differentiability we use the following method.

Define  a conformal map 
\begin{equation*}
 g_t := H_{u_t}[\sigma]^{-1} \circ G_t,\quad t\geq 0.
\end{equation*}
It satisfies the equation 
\begin{equation}
 \dot g_t = (H_{u_t}[\sigma]^{-1}_* \delta) \circ g_t,\quad g_0=\id,
 \label{Formula: d g = h delta g dt}
\end{equation}
where $H_{u_t}[\sigma]^{-1}_* \delta$ is defined in 
\eqref{Formula: G v = 1/G' v(G)}
for $G=H_{u_t}[\sigma]$.
Reciprocally,  (\ref{Formula: d G = delta G dt + sigma G du}) can be obtained from 
(\ref{Formula: d g = h delta g dt}), although (\ref{Formula: d g = h delta g dt}) is defined for a continuous function $u_t$, not necessary continuously differentiable. This motivates the following definition.

\begin{definition}
 Let  $\sigma$ and $\delta$ be a complete and a semicomplete vector fields as  in 
 (\ref{Formula: sigma with 1 slit})
and 
 (\ref{Formula: delta with 1 slit}), 
 and let $u_t$ be a continuous function $u:[0,\infty)\map \mathbb{R}$. 
 Then the solution $g_t$ to the initial value problem (\ref{Formula: d g = h delta g dt})
 is called the \emph{(forward) general slit L\"owner chain}. Respectively, the map $G_t$ is defined by
 \begin{equation*}
 	G_t:=H_{u_t}[\sigma]\circ g_t,\quad t\geq 0.
 \end{equation*}
\end{definition}

The stochastic version of equation \eqref{Formula: d g = h delta g dt} can be set up by introducing the Brownian measure on the set of driving functions, or equivalently, as follows.
\begin{definition}
 Let $\sigma$ and $\delta$ be as  in 
 (\ref{Formula: sigma with 1 slit})
and 
 (\ref{Formula: delta with 1 slit})), and let $B_t$ be the standard Brownian motion (the Wiener process).
 Then the solution to the stochastic differential equation 
 \begin{equation}
  \dS G_t = 
  \delta \circ G_t dt + \sigma \circ G_t \dS B_t,\quad G_0=\id,\quad t\geq 0
  \label{Formula: Slit hol stoch flow Strat}
 \end{equation}
 (where $d^{\rm{S}}$ is the Stratanowich differential)
 is called a \emph{slit holomorphic stochastic flow} or ($\delta,\sigma$)-SLE.
\end{definition}

In order to formulate 
(\ref{Formula: Slit hol stoch flow Strat})
in the It\^{o} form we have to chose some chart $\psi$: 
\begin{equation}
 \dI G_t^{\psi}(z) =
 \left(\delta^{\psi} + \frac12\sigma^{\psi} \left(\sigma^{\psi}\right)' \right) 
  \circ G_t^{\psi}(z) dt 
 + \sigma^{\psi} \circ G_t^{\psi}(z) \dI B_t,\quad G^{\psi}_0(z)=z,
 \label{Formula: Slit hol stoch flow It\^{o}}
\end{equation}
see, for example, \cite[Section 4.3]{Gardiner1982}  for the definition of the Stratanovich and It\^{o}'s differentials and for the relations between them. A disadvantage of  It\^{o}'s form is that the coefficient at $dt$ transforms from  chart to  chart in a complicated way.

\begin{example}
\label{Example: Chordal Loewner equation} 
\textbf{Chordal L\"owner equation.}

Let us show how the construction 
(\ref{Formula: d G = delta G dt + sigma G du}--\ref{Formula: Slit hol stoch flow It\^{o}})
works in the case of the chordal L\"{o}wner equation.
Define 
\begin{equation*}\begin{split}
 \delta^{\HH}(z) = \frac{2}{z},\quad
 \sigma^{\HH}(z) = -1,
\end{split}\end{equation*}
in the half-plane chart.
Then $H[\sigma]_s$ can be found from  equation (\ref{Formula: d H = sigma H ds}) as
\begin{equation*}\begin{split}
 \dot H_s[\sigma]^{\HH}(z) = -1 ~,\quad
 H_0[\sigma]^{\HH} (z) = z,\quad z\in \HH,
\end{split}\end{equation*}
\begin{equation*}\begin{split}
 H_s[\sigma]^{\HH}(z) = z - s\quad z\in \HH.
\end{split}\end{equation*}
The equation for $g_t$ becomes
\begin{equation*}\begin{split}
 \dot g_t^{\HH} = \frac{2}{g_t^{\HH}-u_t},
\end{split}\end{equation*}
and it is known as the chordal L\"owner equation.

For a differentiable $u_t$ we can write
\begin{equation*}\begin{split}
 \dot G_t^{\HH}(z) = \frac{2}{G_t^{\HH}(z)} - \dot u_t~
\end{split}\end{equation*}
The stochastic version in the Stratonovich form can be obtained by  substituting 
$u_t=\sqrt{\kappa}B_t$:
\begin{equation}\begin{split}
 d^{\rm{S}} G_t^{\HH}(z) = \frac{2}{G_t^{\HH}(z)}dt - \sqrt{\kappa} d^{\rm{S}} B_t,
 \label{Formula: chordal SLE in H in Strat}
\end{split}\end{equation}
and it is of the same form in the It\^{o} case in the half-plane chart because ${\sigma^{\HH}}'(z)\equiv 0$,
\begin{equation*}\begin{split}
 d^{\rm{It\^{o}}} G_t^{\HH}(z) = \frac{2}{G_t^{\HH}(z)}dt - \sqrt{\kappa} d^{\rm{It\^{o}}} B_t.
\end{split}\end{equation*}
In other charts 
$\left(\sigma^{\chi}\right)'(z)\not \equiv 0$, and the Stratonovich and It\^{o} forms differ.
The space 
$\mathcal{G}[\delta,\sigma]$ 
consists of endomorphisms that in the half-plane chart are normalized as
\begin{equation*}
	G_t^{\HH}(z) = z + O\left(\frac{1}{z}\right).
\end{equation*}
\end{example}

\begin{example} 
\label{Example: Radial Leowner equation} 
\textbf{Radial Leowner equation.}

Let $\psi_{\D}:\Dc\map\D$ be the chart map defined by
(\ref{Formula: tau_H D}),
and let
$\tau_{\HH,\D}=\psi_{\HH}\circ\psi_{\D}^{-1}$.
The transition function $\tau_{\HH,\D}$ maps the point $1$ in the unit disk chart to the point $0$ in the half-plane chart, the point $-1$ to a point at infinity, and the point $0$ to $+i$. Similarly to what we have done in the half-plane case, we define 
$X^{\D}:=X^{\psi_{\D}}$ and call it \emph{the unit disk chart}.

Let
\begin{equation}\begin{split}
 \delta^{\D}(z) = -z\frac{z+1}{z-1},\quad
 \sigma^{\D}(z) = - i z,
 \label{Formula: delta and sigma for radial SLE in D}
\end{split}\end{equation}
that corresponds to 
\begin{equation}\begin{split}
 \delta^{\HH}(z) = \frac12 \left( \frac{1}{z} + z \right),\quad
 \sigma^{\HH}(z) = - \frac12 (1+z^2)
 \label{Formula: delta and sigma for radial SLE in H}
\end{split}\end{equation}
in the half-plane chart.

The equation for $H_s[\sigma]$ in the unit disk chart becomes
\begin{equation*}\begin{split}
 &\dot H_s[\sigma]^{\D}(z) = -i H_s[\sigma]^{\D}(z),\quad
  H_0[\sigma]^{\D}(z) = z,\quad
 z\in \D.
\end{split}\end{equation*}
The solution takes the form
\begin{equation*}\begin{split}
 H_s[\sigma]^{\D}(z) = e^{-i s}z,
\end{split}\end{equation*}
and
\begin{equation*}
 \dot G_t^{\D}(z) = 
 -G_t^{\D}(z) \frac{ G_t^{\D}(z)+1}{G_t^{\D}(z)-1}   
 -i G_t^{\D}(z) \dot u_t.
\end{equation*}
Thus, the equation for $g_t$ becomes
\begin{equation*}\begin{split}
 \dot g_t^{\D}(z) =& 
 \left( 
  \frac{1}{{H_{u_t}[\sigma]^{\D}}'} 
   \left( - H_{u_t}[\sigma]^{\D} 
   \frac{H_{u_t}[\sigma]^{\D}+1}{H_{u_t}[\sigma]^{\D}-1} \right) 
 \right)   
 \circ g_t^{\D}(z) 
 =\\=&
 \frac{1}{ e^{-i u_t} } 
 \left( - e^{-i u_t} g_t^{\D}(z) 
  \frac{ e^{-i u_t} g_t^{\D}(z) + 1 }{ e^{-i u_t} g_t^{\D}(z) - 1 } 
 \right)  
\end{split}\end{equation*}
or
\begin{equation*}\begin{split}
 \dot g_t^{\D}(z) = 
 g_t^{\D}(z) \frac{e^{i u_t} +g_t^{\D}(z)}{ e^{iu_t} - g_t^{\D}(z)} .
\end{split}\end{equation*}

Its stochastic version in the Stratanovich form is
\begin{equation}
 d^{\rm{S}} G_t^{\D}(z) = 
 G_t^{\D}(z) \frac{ G_t^{\D}(z)+1}{G_t^{\D}(z)-1} dt   
 -i \sqrt{\kappa} G_t^{\D}(z) d^{\rm{S}} B_t.
 \label{Formula: G radial SLE in D}
\end{equation}
or in the half-plane chart
\begin{equation*}
 d^{\rm{S}} G_t^{\HH}(z) = \frac12 \left(
  \frac{1}{G_t^{\HH}(z)} + G_t^{\HH}(z)
 \right) dt -
 \frac{\sqrt{\kappa}}{2} \left( 1 + \left(G_t^{\HH}(z)\right)^2 \right) d^{\rm{S}} B_t.
\end{equation*}
The collection 
$\mathcal{G}[\delta,\sigma]$
consists of maps that in the unit disk chart are normalized by
\begin{equation*}
	G^{\D}(0) = 0.
\end{equation*}
\end{example}

\medskip
\noindent
It turns out that some different combinations of $\delta$ and $\sigma$ induce measures that can be transformed one to another in a simple way. For example, let $m:\Dc\map\Dc$ be a M\"obius automorphism fixing a point $a\in\de \Dc$ where $\delta$ has a pole. 
Then 
\begin{equation}
 \delta\map m_* \delta,\quad \sigma\map m_*\sigma
 \label{Formula: M - transfrom}
\end{equation}
are also vector fields of the form  
(\ref{Formula: sigma with 1 slit})
and 
(\ref{Formula: delta with 1 slit}).
For instance, let
$\delta$ and $\sigma$ be as  in Example 
\ref{Example: Radial Leowner equation}, and let
$m^{\D}:\D\map \D$ 
map the center point $0$ of the unit disk to another point inside it. Then 
we come to an equation defined by $m_*\delta$ and  $m_*\sigma$, which is still in fact,
the radial equation  written in different coordinates. 

Another example of such a transformation preserving the form 
(\ref{Formula: sigma with 1 slit})
and 
(\ref{Formula: delta with 1 slit})
can be constructed as follows.
\begin{equation}
 \delta^{\HH}(z) \map c^2 \delta^{\HH}(z),\quad  
 \sigma^{\HH}(z) \map c \sigma^{\HH}(z),\quad 
 t \map c^{-2} t,\quad 
 c>0.
 \label{Formula: S - transform}
\end{equation}
The solution is not changed as a random law, because
$c B_{t/c^2}$ and $B_t$ agree in law.

It is important to know which of the equations defined by the parameters $\delta_{-2}$, $\delta_{-1}$, $\delta_{0}$, $\delta_{1}$, $\sigma_{-1}$, $\sigma_{0}$, and $\sigma_{0}$ are `essentially different'. A systematic analysis of this question 
was presented in \cite{Ivanov2014}.  Without lost of generality we can restrict ourselves to the from
\begin{equation}\begin{split}
 \delta^{\HH}(z) = \frac{2}{z} + \delta_{-1} + \delta_{0} z + \delta_{1} z^2,\quad
 \sigma^{\HH}(z) = - \sqrt{\kappa}(1 + \sigma_{0} z + \sigma_{1} z^2),\quad \kappa>0.
 \label{Formula: delta and sigma normalized form}
\end{split}\end{equation}
Transformations 
(\ref{Formula: S - transform})
and
\begin{equation}
 \delta\map {r_c}_* \delta,\quad \sigma\map {r_c}_*\sigma,\quad 
 r^{\HH}_c(z):=\frac{z}{1-cz},~c\in\mathbb{R}
 \label{Formula: R - transform}
\end{equation}
preserve 
(\ref{Formula: delta and sigma normalized form}) and keep $\kappa$ unchanged. 
Thus, we can fix some 2 of  6 parameters in
(\ref{Formula: delta and sigma normalized form}). 
Besides, the transform
\begin{equation*}
 \delta\map \delta + \frac{\nu}{\sqrt{\kappa}} \sigma,\quad \nu\in\mathbb{R},
\end{equation*}
can be interpreted as an insertion of a drift to the Brownian measure.
Thus, all 6 parameters 
can be fixed  by $\kappa$ (responsible for the fractal dimension of the slit), 
$\nu$ 
(the drift), and some 2 
parameters that set the equation type (for example, chordal, radial, and dipolar).

Due to the autonomous form of the equation 
(\ref{Formula: d G = delta G dt + sigma G du})
the solution $G_t$ possesses the following property.
\begin{proposition}[\cite{Ivanov2014}]
	Let $\tilde u_{t-s}:=u_t - u_s$ for some fixed 
	$s:t\geq s \geq 0$, and let 
	  $\tilde G_{\tilde t}$ be defined by 
	\eqref{Formula: d G = delta G dt + sigma G du} 
	with the driving function $\tilde u_{\tilde t}$. Then
	\begin{equation*}
		G_t = \tilde G_{t-s} \circ G_s.
	\end{equation*}
\end{proposition}

In the the stochastic case 
the process $\{G_t\}_{t\geq 0}$,
taking values in the space of inverse endomorphisms of $\Dc$,
is a continuous homogeneous Markov process.
In particular,
\begin{proposition}[\cite{Ivanov2014}]
If $G_t$ and $\tilde G_s$ are two independently sampled 
($\delta,\sigma$)-SLE maps, then $G_t \circ \tilde G_s$ has the same law as 
$G_{t+s}$. 
\end{proposition}

\subsection{Pre-pre-Schwarzian}
\label{Section: Pre-pre-Schwarzian}

A collection of maps $\eta^{\psi}\colon \psi(\Dc)\map \mathbb{C}$, each of which is given in a global chart map $\psi:\Dc\map \psi(\Dc)\subset \mathbb{C}$, is called 
a \emph{pre-pre-Schwarzian} form of order $\mu,\mu^* \in\mathbb{C}$ if for any chart map $\tilde \psi$
\begin{equation}
 \eta^{\tilde \psi}(\tilde z) = 
 \eta^{\psi}( \tau (\tilde z) ) 
 + \mu \log \tau'(\tilde z)
 + \mu^* \log \bar \tau'(\tilde z) ,\quad
 \tau = \psi \circ \tilde\psi^{-1},\quad
 \tilde z\in\tilde D,\quad
 \forall \psi,\tilde\psi,
 \label{Formula: tilde phi = phi - chi arg}
\end{equation}
for any chart map $\tilde \psi$.
If $\eta$ is defined for one chart map, then it is automatically defined for all  chart maps. We borrowed the term `pre-pre-Schwarzian' from 
\cite{Kang2011}. In \cite{Sheffield2010}, an analogous object is called `AC surface'.

Analogously to vector fields in Section
\ref{Section: Vector fields and coordinate transform}
we define the \emph{pushforward of a pre-pre-Schwarzian} by
\begin{equation}
 G_*^{-1} \eta^{\psi}(z) : = 
 \eta^{\psi \circ G}(z) =
 \eta^{\psi}( {G}^{\psi}(z) ) 
 + \mu \log \left({G}^{\psi}\right)'(z)
 + \mu^* \log \overline{ \left({G}^{\psi}\right)'(z)}
 \label{Formula: G eta(z) = eta(G(z)) + mu log G'(z) + ...}
\end{equation}

We are interested in two special cases. The first one corresponds to $\mu=\mu^*=\gamma/2\in\mathbb{R}$, and
\begin{equation}
 G_*^{-1} \eta^{\psi}(z) = 
 \eta^{\psi}( {G}^{\psi}(z) ) 
 + \gamma \log \left| {G}^{\psi}{}'(z) \right|.
 \label{Formula: tilde eta = eta + gamma log}
\end{equation}
The second one is $\mu=i\chi/2,~\mu^*=-i\chi/2,~\chi\in\mathbb{R}$, and
\begin{equation}
 G_*^{-1} \eta^{\psi}(z) = 
 \eta^{\psi}( {G}^{\psi}(z) ) 
 - \chi \arg {G}^{\psi}{}'(z).
 \label{Formula: tilde eta = eta - chi arg}
\end{equation}
In both cases $\eta$ can be chosen real in all charts. Moreover, if the pre-pre-Schwarzian is represented by a real-valued function  it is one of two above forms in all charts.
 
A ($\mu,\mu^*$)-pre-pre-Schwarzian can be obtained from a vector field $v$ by the relation
\begin{equation}
 \eta^{\psi}(z) \equiv -\mu \log v^{\psi}(z) - \mu^* \log \overline{v^{\psi}(z)}. 
 \label{Formula: eta = -mu log v - mu log v}
\end{equation}
For  two special cases above we have
\begin{equation*}
 \eta = -\gamma \log |v|
\end{equation*} 
and
\begin{equation*}
 \eta = \chi \arg v,
\end{equation*} 
where we drop the upper index $\psi$.
 
In Section 
\ref{Section: Linear functionals and change of coordinates},
we obtain the transformation rules 
(\ref{Formula: tilde eta = eta + gamma log})
by taking logarithm of a $(1,1)$-differential.
The second type of the real pre-pre-Schwarzian is connected to a sort of an imaginary analog of the metric. 

We can define the \emph{Lie derivative} of $X$ as
\begin{equation*}
 \Lc_{v} X^{\psi} (z):=
 \left. \frac{\de}{\de \alpha} {H_t^{-1}[v]}_* X^{\psi} (z) \right|_{t=0},
\end{equation*}
where $X$ can be a pre-pre-Schwarzian, a vector field, or even an object with a more general transformation rule, see (\ref{Formula: d G = sigma G ds in a chart})
for any two holomorphic vector fields $v$.
If $X$ is a holomorphic vector field $w$, then 
\begin{equation}
	\Lc_v w^{\psi} = 
	[v,w]^{\psi}(z) := 
	v^{\psi}(z) w^{\psi}{}'(z) - v^{\psi}{}'(z) w^{\psi}(z),
 \label{Formula: L_v w = [v,w] := v w' - v' w}
\end{equation}
see 
(\ref{Formula: G v = 1/G' v(G)}). 

If $X$ is a pre-pre-Schwarzian, then
\begin{equation}
 \Lc_{v} \eta^{\psi} (z) = 
 v^{\psi}(z) \de_z \eta^{\psi}(z) + 
 \overline{v^{\psi}(z)} \de_{\bar z} \eta^{\psi}(z) 
 + \mu \,{v^{\psi}}'(z) + \mu^* \overline{{v^{\psi}}'(z)}. 
 \label{Formula: L eta = v d eta + chi dv}
\end{equation}
Here and further on, we use notations
\begin{equation*}\begin{split}
 \de_z := \frac12 \left( \frac{\de}{\de x} - i \frac{\de}{\de y} \right),\quad
 \de_{\bar z} := \frac12 \left( \frac{\de}{\de x} + i \frac{\de}{\de y} \right),
\end{split}\end{equation*}
and $f'(z):=\de_z f(z)$ for a holomorphic function $f$.

If $\mu = \mu^* = 0$, then $\eta$ is called a \emph{scalar}. It is remarkable that if $\eta$ is a pre-pre-Schwarzian, then $\Lc_v \eta$ is a scalar anyway, which is stated in the following lemma.

\begin{lemma} 
Let $\eta$ be a pre-pre-Schwarzian of order $\mu,\mu^*$, let $v$ be a holomorphic vector field, and let $G$ be a conformal self-map. Then 
\begin{equation*}
 G^{-1}_* (\Lc_v \eta)^{\psi}(z) = (\Lc_v \eta)^{\psi} \circ G^{\psi}(z)
\end{equation*}
or in the infinitesimal form,
\begin{equation}
 \Lc_w (\Lc_v \eta)^{\psi}(z) = 
 w^{\psi}(z) \de_z (\Lc_v \eta)^{\psi}(z) +  
 \overline{w^{\psi}(z)} \de_{\bar z} (\Lc_v \eta)^{\psi}(z).
 \label{Formula: L L phi = w de L phi}
\end{equation}
\end{lemma}

\begin{proof}
The straightforward calculations imply
\begin{equation*}\begin{split}
 &G^{-1}_* (\Lc_v \eta)^{\psi}(z) = G^{-1}_* \left(
  v^{\psi}(z) \de_z \eta^{\psi}(z) + \overline{v^{\psi}(z)} \de_{\bar z} \eta^{\psi}(z) 
  + \mu {v^{\psi}}'(z) + \mu^* \overline{{v^{\psi}}'(z)} 
 \right)
 =\\=&
 \frac{v^{\psi}\circ G^{\psi}(z) }{{G^{\psi}}'(z)} \de_z 
  \left( \eta^{\psi}\circ G^{\psi}(z) 
  + \mu \log {G^{\psi}}'(z) + \mu^* \log \overline{{G^{\psi}}'(z)} \right) 
 +\\+&
 \overline{\frac{v^{\psi}\circ G^{\psi}(z) }{{G^{\psi}}'(z)}}
  \de_{\bar z} \left( \eta^{\psi}\circ G^{\psi}(z)
  + \mu \log {G^{\psi}}'(z) + \mu^* \log \overline{{G^{\psi}}'(z)} \right) 
 +\\+&
 \mu \de_z \frac{v^{\psi}\circ G^{\psi}(z)}{{G^{\psi}}'(z)} 
 + \mu^* \de_{\bar z} \frac{\overline{v^{\psi}\circ G^{\psi}(z)}}{\overline{{G^{\psi}}'(z)}} 
 =\\=&
 v^{\psi}\circ G^{\psi}(z) (\de \eta^{\psi})\circ G^{\psi}(z) 
 + \mu \frac{v^{\psi}\circ G^{\psi}(z) }{{G^{\psi}}'(z)}
 \de_z \log {G^{\psi}}'(z)
 + 0 
 +\\+&
 \overline{v^{\psi}\circ G^{\psi}(z)} (\bar \de \eta^{\psi})\circ G^{\psi}(z) 
 + 0 + \mu^* \overline{ \frac{v^{\psi}\circ G^{\psi}(z) }{{G^{\psi}}'(z)} }
 \de_{\bar z} \log \overline{{G^{\psi}}'(z)} 
 +\\+&
 \mu {v^{\psi}}' \circ G^{\psi}(z)  
 - \mu \frac{v^{\psi}\circ G^{\psi}(z) {G^{\psi}}''(z) }{{G^{\psi}}'(z)^2} 
 + \mu^* \overline{{v^{\psi}}' \circ G^{\psi}(z)}  
 - \mu^* \frac{\overline{v^{\psi}\circ G^{\psi}(z)} \overline{{G^{\psi}}''(z)} }{\overline{{G^{\psi}}'(z)^2}} 
 =\\=&
 \left( v^{\psi} \de \eta^{\psi} \right) \circ G^{\psi}(z) 
 + \mu {v^{\psi}}' \circ G^{\psi}(z)
 +\\+&
 \left( \overline{v^{\psi}} \bar \de \eta^{\psi} \right) \circ G^{\psi}(z)
 + \mu^* \overline{{v^{\psi}}' \circ G^{\psi}(z)}  
 =\\=&
 \left( 
  v^{\psi} \de \eta^{\psi} + \overline{ v^{\psi}} \bar \de \eta^{\psi}
  + \mu {v^{\psi}}' + \mu^* \overline{{v^{\psi}}'} 
 \right) \circ G^{\psi}(z)
 =\\=&
 (\Lc_v \eta)^{\psi} \circ G^{\psi}(z) 
\end{split}\end{equation*}
\end{proof}

\subsection{Test functions}
\label{Section: Test function}

We will define the Schwinger functions $S_n$ and the Gaussian free field $\Phi$
in terms of linear functionals over some space of smooth test functions defined
in what follows. 

Let $\Hc_s^{\psi}$ 
\index{$\Hc_s$}
be a linear space of real-valued smooth
functions $f\colon D^{\psi} \map \mathbb{R}$ 
in the domain 
$D^{\psi}:=\psi(\Dc)\subset\C$ 
with compact support  equipped with the topology of homogeneous
convergence of all derivatives on the corresponding compact, namely, the
topology is generated by following collection of neighborhoods of the zero
function
\begin{equation*}\begin{split}
	&U^{\psi}_{K} := 
	\bigcap\limits_{n,m=0,1,2,\dotso}  
	\{f(z)\in C^{\infty}(D) \colon \supp f \subseteq K \wedge \left| \de^n \bar
	\de^m f^{\psi}(z)\right| <\varepsilon_{n,m},\quad z\in K \},\\
	&\varepsilon_{n,m}>0,\quad n,m = 0,1,2\dotso,
\end{split}\end{equation*}
where $K\subset D^{\psi}$ is any compact subset of $D^{\psi}$.

We call $f^{\psi}\in\Hc_s^{\psi}$ the \emph{test functions}
\index{test function} 
and assume that they are $(1,1)$-differentials 
\begin{equation}
 f^{\tilde \psi}
 (\tilde z) = \tau'(\tilde z) \overline{\tau'(\tilde z)} f
 ^{\psi}(\tau(\tilde z)),\quad
 \tau:=\psi \circ \tilde \psi^{-1},
 \label{Formula: f^tilde psi = tau^2 f^psi(tau)}
\end{equation} 
It is straightforward to check that any transition map 
$\tau$ induces a homeomorphism between $\Hc_s^{\psi}$ and 
$\Hc_s^{\tilde \psi}$.
Thereby, we will drop the index $\psi$ at $\Hc_s$, and  consider 
the space $\Hc_s$ as a topological space of smooth ($1,1$)-differentials with
compact support.

The space $\Hc_s$ does not match  all cases of coupling. For the couplings with radial SLE we use spaces $\Hc_{s,b}$ and
$\Hc_{s,b}^{\pm}$
defined in corresponding Sections
\ref{Section: Radial SLE with drift}
and
\ref{Section: Coupling to twisted GFF}.
Henceforth, we denote by $\Hc$
any of those nuclear spaces 
$\Hc_s$, $\Hc_{s,b}$, or
$\Hc_{s,b}^{\pm}$ for shortness.
An important property of $\Hc$ is  \emph{nuclearity},
see
\cite{Gelfand1964,Hida2008,Pietsch1972} 
which is necessary and sufficient to admit the uniform Gaussian measure
on the dual space $\Hc'$ (the GFF).

Constructing such a uniform Gaussian measure on a finite dimentional linear
space is a trivial problem, however, it is not possible on an
infinite-dimentional Hilbert space. On the other hand, if a space $\Hc$ is
nuclear as $\Hc_s$, then the dual space $\Hc'$ admits a uniform
Gaussian measure. A general recipy holds not only for Gaussian measures and is
given by the following theorem.

\begin{proposition} 
(\textbf{Bochner-Minols} \cite{Gelfand1964,Hida2008,SergioAlbeverio}) \\
 \label{Theorem: Bochner-Minols}
 Let $\Hc$ be a nuclear space, and let  
  $\hat \mu\colon \Hc\map \C $ be a functional (non-linear). 
 Then the following 3 conditions  
 \begin{enumerate} [1.]
  \item $\hat \mu$ is positive definite
  \begin{equation*}
   \forall \{z_1,z_2,\dotso z_n\}\in\C^n,~ 
   \forall \{f_1,f_2,\dotso f_n\}\in \Hc^n~\then
   \sum\limits_{1\leq k,l\leq n} z_k \bar z_l \hat \mu [f_k - f_l]\geq 0;	 	
  \end{equation*}
  \item $\hat \mu(0)=1$;
  \item $\hat \mu$ is continuous
 \end{enumerate}  
 are satisfied if and only if there exists a unique probability measure 
 $P_{\Phi}$ 
 on
 $(\Omega_{\Phi},\mathcal{F}_{\Phi},P_{\Phi})$ 
  for 
 $\Omega_{\Phi}=\Hc'$, 
 with $\hat \mu$ as a characteristic function
 \begin{equation}
  \hat \mu [f] := 
  \int\limits_{\Phi\in \Hc'} e^{(i\Phi[f])} P_{\Phi}(d\Phi), \quad \forall
  f \in \Hc.
  \label{Formula: hat mu = int exp iPhi dPhi 2}
 \end{equation}
 The corresponding $\sigma$-algebra $\mathcal{F}_{\Phi}$ is generated by the
 cylinder sets
 \begin{equation*}
  \{ F \in \Hc'\colon \quad F[f]\in B \},\quad 
  \forall f\in \Hc,\quad
  \forall \text{ Borel sets } B \text{ of } \mathbb{R}~. 
 \end{equation*} 
\end{proposition}

The random law on $\Hc'$ is called uniform with respect to a bilinear functional
$B\colon \Hc\times\Hc\map \mathbb{R}$ 
if the characteristic function
$\hat \mu$ 
is of the form
\begin{equation*}
	\hat \mu[f] = e^{-\frac12 B[f,f]},\quad f\in \Hc.
\end{equation*}  
We consider the class of bilinear functionals we work with in Section
\ref{Section: Fundamental solution to the Laplace-Beltrami equation}.
First, we study the linear and bilinear functionals over $\Hc_s$ and
their transformation properties.

\subsection{Linear functionals and change of coordinates}
\label{Section: Linear functionals and change of coordinates}

In this section, we consider linear functionals over $\Hc_s$ and $\Hc$ that
transform as pre-pre-Schwarzians.

Let 
$\eta^{\psi}\in\Hc_s^{\psi}{}'$
be a linear functional over 
$\Hc_s^{\psi}$
for a given chart $\psi$.
The functional is called regular if there exists a locally integrable function 
$\eta^{\psi}(z)$ 
such that 
\begin{equation*}
 \eta^{\psi}[f] := 
 \int\limits_{D^{\psi}} \eta^{\psi}(z) f^{\psi}(z) l(dz),
\end{equation*}
where $l$ is the Lebesgue measure on $\C$. We use the brackets $[\cdot]$ for
functionals and the parentheses $(\cdot)$ for corresponding functions
(kernels).

We assume that $f$ transforms according to
(\ref{Formula: f^tilde psi = tau^2 f^psi(tau)}).
If $\eta^{\psi}(z)$ is a scalar, then the number $\eta^{\psi}[f]\in\mathbb{R}$
does not depend on the choice of the chart $\psi$. Indeed, for any choice of
another chart $\tilde \psi$, we have
\begin{equation*}\begin{split}
 \eta^{\tilde \psi}[f] :=& 
 \int\limits_{D^{\tilde \psi}} \eta^{\tilde \psi}(\tilde z) f^{\tilde \psi}(\tilde z) l(d\tilde z)=
 \int\limits_{D^{\tilde \psi}} \eta^{\psi}(\tau(\tilde z)) f^{\psi}(\tau(\tilde z))
 |\tau'(\tilde z)|^2  l(d\tilde z)
 =\\=&
 \int\limits_{D^{\psi}} \eta^{\psi}(z) f^{ \psi}(z) l(dz)
 =\eta^{\psi}[f] ,
\end{split}\end{equation*}

If $\eta^{\psi}(z)$ is a pre-pre-Schwarzian, then
\begin{equation}\begin{split}
 \eta^{\tilde \psi}[f] =& 
 \int\limits_{D^{\tilde \psi}} \eta^{\tilde \psi}(\tilde z) f^{\tilde \psi}(\tilde z) l(d\tilde z)
 =\\=&
 \int\limits_{D^{\tilde \psi}} \left(
  \eta^{\psi}(\tau(\tilde z)) 
  + \mu \log \tau'(z) + \mu^* \overline{\log \tau'(z)}
 \right)  
 f^{\psi}(\tau(\tilde z)) |\tau'(\tilde z)|^2  l(d\tilde z)
 =\\=&
 \int\limits_{D^{\psi}} \left(
  \eta^{\psi}(z) 
  - \mu \log \tau^{-1}{}'(z) - \mu^* \overline{\log \tau^{-1}{}'(z)}
 \right)  
 f^{\psi}(z)  l(d z)
 =\\=&
 \eta^{\psi}[f] 
 - \int\limits_{D^{\psi}} \left(
  \mu \log \tau^{-1}{}'(z) + \mu^* \overline{\log \tau^{-1}{}'(z)}
 \right)  
 f^{\psi}(z) l(d z)
 \label{Formula: eta^tilde psi[f] = eta psi[f] - mu int ...}
\end{split}\end{equation}
according to 
\eqref{Formula: G eta(z) = eta(G(z)) + mu log G'(z) + ...}.

If $\eta^{\psi}$ is not a regular pre-pre-Schwarzian but just a functional from 
$\Hc_s'$ we can consider the last line of
\eqref{Formula: eta^tilde psi[f] = eta psi[f] - mu int ...}
as a definition of the transformation rule for $\eta[f]$ from a chart $\psi$
to a chart $\tilde \psi$.

Let us denote by $\Hc_s'$ the linear space of pre-pre-Schwarzians as above.
Consider now the pushforward operation $G_*^{-1}$
on $(1,1)$-differentials $f$ defined by
\begin{equation*}
	G_*^{-1} f^{\psi}(z) := 
	\left| G^{\psi} {}'(z)\right|^2 
	f^{\psi} \left(G^{\psi} (z) \right).
\end{equation*}
The right-hand side is well-defined only for
$\psi(\Dc\setminus\K)$.
Here we define $F_*$ only on a subset of $\Hc_s$ of test functions that
are supported in $\Dc\setminus\K=F^{-1}(\Dc)$.

Define the pushforward operation by
\begin{equation}\begin{split}
	&G_*^{-1} \eta^{\psi} [f] 
	= \eta^{\psi\circ G} [f] 
	=\\=&
	\eta^{\psi}[G_* f] 
	+	\int\limits_{\supp f^{\psi}} 
	\left( 
		\mu \log G^{\psi}{}'(z) +
		\mu^{*}	\overline{\log G^{\psi}{}'(z)}
	\right)
	f^{\psi}(z) l(dz),\\
	&	f\in\Hc_s\colon\quad \supp f\subset G^{-1}(\Dc).
	\label{Formula: F eta f= eta F-1 f}
\end{split}\end{equation}
It can be understood as a pushforward $F_*:\Hc_s'\map\Hc_s'$ 
in the dual space.

Functionals over the space $\Hc_s$ are differentiable infinitely many times. 
According to 
\eqref{Formula: L eta = v d eta + chi dv}
the Lie derivative is defined by
\begin{equation*}\begin{split}
	\Lc_v \eta[f] =& 
	\left. \frac{\de}{\de s} {H_s^{-1}[v]}_* \eta^{\psi} [f] \right|_{s=0} 
	=\\=&
	-\eta^{\psi}[\Lc_{v} f] 
	+ \int\limits_{\supp f^{\psi}} 
	\left(
		\mu v^{\psi}{}'(z) + \mu^* \overline{v^{\psi}{}'(z)} 
	\right)
	f^{\psi}(z) l(dz),
\end{split}\end{equation*}
where
\begin{equation*}\begin{split}
	\Lc_v f^{\psi}(z) =& 
	\left. \frac{\de}{\de s} {H_s^{-1}[v]}_* f^{\psi} \right|_{s=0} 
	=\\=&
	v^{\psi}(z) \de_z f^{\psi}(z) + \overline{v^{\psi}(z)} \de_{\bar z} f^{\psi}(z)
	+ v^{\psi}{}'(z) f^{\psi}(z) + \overline{v^{\psi}{}'(z)} f^{\psi}(z).
\end{split}\end{equation*}

\subsection{Fundamental solution to the Laplace-Beltrami equation}
\label{Section: Fundamental solution to the Laplace-Beltrami equation}

In this section, we consider linear continuous functionals with respect to
each argument in $\Hc_s$. An important example is the Dirac functional
\begin{equation}
	\delta_{\lambda}[f,g] := 
	\int\limits_{\psi(\Dc)} f^{\psi}(z) g^{\psi}(z) 
	\frac{1}{\lambda^{\psi}(z)} l(dz),\quad 
	f,g\in\Hc_s,
	\label{Formual: delta_lambda[f,g] := ...}
\end{equation}
where $\lambda(z) l(dz)$ is  some measure on $D^{\psi}$, which is
absolutely continuous with respect to the Lebesgue measure $l(dz)$. 
The Radon-Nikodym
 derivative $\lambda^{\psi}(z)$ transforms as a
$(1,1)$-differential:
\begin{equation*}
	\lambda^{\tilde \psi}
 (\tilde z) 
 = \tau'(\tilde z) \overline{\tau'(\tilde z)} \lambda^{\psi}
 (\tau(\tilde z))
 ,\quad
 \tau:=\psi \circ \tilde \psi^{-1}.
\end{equation*}
It is easy to see that the right-hand side of
\eqref{Formual: delta_lambda[f,g] := ...}
does not depend on the choice
of $\psi$.
 
We call the functional regular if there exists a function 
$B^{\psi}(z,w)$ on $\psi^{\Dc}\times\psi^{\Dc}$ such that
\begin{equation}
	B^{\psi}[f,g] := 
	\int\limits_{\psi(\Dc)} \int\limits_{\psi(\Dc)} B^{\psi}(z,w )
	f^{\psi}(z) g^{\psi}(w)
 	l(dz) l(dw),\quad 
	f,g\in\Hc_s. 
	\label{Formula: H = int H}
\end{equation}
Let us use the same convention about the brackets and parentheses as for the linear
functionals. We consider only scalar regular bilinear functionals and
require the transformation rules
\begin{equation*}
 B^{\tilde \psi}(\tilde z,\tilde w) = 
 B^{\psi} (\tau(\tilde z),\tau(\tilde w)),\quad 
 \tau=\psi\circ \tilde \psi^{-1},\quad 
 z,w\in\tilde\psi(\Dc).
\end{equation*}
Thus, the right-hand side of 
(\ref{Formula: H = int H})
does not depend on the choice 
of the chart $\psi$ and we can drop the index $\psi$ in the left-hand side.

The pushforward is defined by 
\begin{equation}
	F_* B^{\psi} (z,w) 
	= B^{\psi\circ F} (z,w)
	:= B^{\psi}( \left(F^{\psi} \right)^{-1}(z),\left(F^{\psi} \right)^{-1}(w))
	,\quad z,w\in \Im(F^{\psi}).
	\label{Formula: G B(z,w) = B(G(z),G(w))},
\end{equation}
which becomes
\begin{equation*}
	F_* B^{\psi} [f,g] 
	= B^{\psi\circ F} [f,g]
	:= B^{\psi}[ F^{-1}_* f,F^{-1}_* g ]
	,\quad f,g\in\Hc_s\colon\supp f\subset \Im(F),
\end{equation*}
for an arbitrary functional $F$,
The same remarks  remain true in this case as in the previous section for $\eta$.

Define now the Lie derivative in the same way as before
\begin{equation}\begin{split}
	&\Lc_{v} B^{\psi}(z,w)
	:=\left. \frac{\de}{\de s} 
	H_s[v]^{-1}_* B^{\psi} (z,w) \right|_{s=0} 
	=\\=&
	v^{\psi}(z)\de_z B^{\psi}(z,w) + 
	\overline {v^{\psi}(z)}\de_{\bar z} B^{\psi}(z,w) + 
	v^{\psi}(w)\de_w B^{\psi}(z,w) + 
	\overline {v^{\psi}(w)}\de_{\bar w} B^{\psi}(z,w). 
	\label{Formula: L Gamma = ...}
\end{split}\end{equation}
We remark that $\Lc_v B$ is also scalar in two variables.
Functionals $\delta_{\lambda}$ and $B$ are both scalar and continious with
respect to each variable.

Define the Laplace-Beltrami operator $\Delta_{\lambda}$ as
\begin{equation*}
	{\Delta_{\lambda}}_1 B^{\psi}(z,w)
	:= -\frac{4}{\lambda^{\psi}(z)} \de_z \de_{\bar z} B^{\psi}(z,w),
\end{equation*}   
where the lower index `$1$' means that the operator acts only with respect to
the first argument.

Let a regular bilinear functional $\Gamma_{\lambda}$ be a solution to the
equation
\begin{equation}
	{\Delta_{\lambda}}_1 \Gamma_{\lambda}[f,g] 
	= 2 \pi~ \delta_{\lambda}[f,g],\quad \Gamma_{\lambda}[f,g]
	=\Gamma_{\lambda}[g,f],
	\quad f,g\in\Hc_s.
	\label{Formula: Delta Gamma = delta}
\end{equation}
The boundary conditions will be fixed later.
This equation is conformally invariant in the sense that if 
$\Gamma_{\lambda}^{\psi}(z,w)$ is a solution on a chart 
$\psi$, then 
\begin{equation*}
 \Gamma_{\lambda}^{\psi}(\tau(\tilde z),\tau(\tilde w)) = 
 \Gamma^{\tau^{-1}\circ \psi}_{\lambda}(z,w)
\end{equation*}
is a solution in the chart 
$\tau^{-1} \circ \psi$.

The solution 
$\Gamma_{\lambda}^{\psi}(z,w)$ 
is a collection of smooth and harmonic functions on 
$\psi(\Dc)\times \psi(\Dc)\setminus\{z\times w\colon z=w\}$ 
of  general form
\begin{equation}
	\Gamma_{\lambda}^{\psi}(z,w)=-\frac12\log (z-w)(\bar z-\bar w)+H^{\psi}(z,w)~,
	\label{Formula: Gamma = Log + H}
\end{equation}
where $H^{\psi}(z,w)$ is an arbitrary symmetric harmonic function with respect
to each variable that is defined by the boundary conditions and will be
specified in what follows.

It is straightforward to verify that the function $\Gamma_{\lambda}^{\psi}(z,w)$
does not depend on the choice of $\lambda$ because the identity
\eqref{Formula: Delta Gamma = delta} 
in the integral form becomes
\begin{equation*}\begin{split}
	&\int\limits_{\psi(\Dc)} \int\limits_{\psi(\Dc)}
		-\frac{4}{\lambda^{\psi}(z)} \de_z \de_{\bar z} 
	\Gamma^{\psi}(z,w)
	f^{\psi}(z) g^{\psi}(w) l(dz)  l(dw) 
	=\\=& 
	\int\limits_{\psi(\Dc)} f^{\psi}(z) g^{\psi}(z) 
		\frac{1}{\lambda^{\psi}(z)} l(dz).
\end{split}\end{equation*}
The change $\lambda\map \tilde \lambda$ is equivalent to the change 
$f^{\psi}(z)\map \frac{\lambda^{\psi}(z)}{\tilde\lambda^{\psi}(z)} f^{\psi}(z)$. 
We will drop the lower index $\lambda$ in $\Gamma_{\lambda}$ in what follows.
The fundamental solutions to the Laplace equation are also known as 
\emph{Green's functions} (for the free field).
\index{Green's function}

\begin{example} \textbf{Dirichlet boundary conditions.}
\label{Example: Dirichlet boundary conditions Gamma}
Let us denote by  $\Gamma_D$  the solution $\Gamma$ to
(\ref{Formula: Delta Gamma = delta}) 
satisfying the zero boundary conditions, namely,
\begin{equation*}
  \left. \Gamma_D^{\HH}(z,w) \right|_{z\in\mathbb{R}}=0,\quad
  \lim_{z\rightarrow \infty}\Gamma_D^{\HH}(z,w) = 0,\quad w\in \HH.
\end{equation*}
Then, $\Gamma_D$ admits the form
\begin{equation}
  \Gamma_D^{\HH}(z,w):=-\frac12\log
  \frac{(z-w)(\bar z-\bar w)}{(z-\bar w)(\bar z-w)},
  \label{Formula: Gamma_D = Log...} 
\end{equation}
and possesses the property of symmetry with respect to all M\"obious automorphisms 
$H:\Dc\map\Dc$,
\begin{equation*}
	H_* \Gamma_D = \Gamma_D
\end{equation*}
or
\begin{equation}
  \Lc_{\sigma} \Gamma_D(z,w) = 0~,\quad 
  \forall \text{ complete vector field } \sigma~.
  \label{Formula: Lv Gamma_D = 0}
\end{equation}
\end{example}

\begin{example} \textbf{Combined Dirichlet-Neumann boundary conditions.}
\label{Example: Gamma: Combined Dirichlet-Neumann boundary conditions}
  Let $\Gamma_{DN}$ denote the solution to
 (\ref{Formula: Delta Gamma = delta}) 
 satisfying the following boundary conditions in the strip chart 
 \begin{equation*}\begin{split}
	&\left. \Gamma_{DN}^{\SSS}(z,w) \right|_{z\in\mathbb{R}}=0, \quad
  \left. \de_y\Gamma_{DN}^{\SSS}(x+i y,w) \right|_{y=\pi }=0, \quad
  x\in\mathbb{R},\\
  &\lim_{z\rightarrow \infty\wedge \Re z>0}\Gamma_{DN}^{\SSS}(z,w) = 0,\quad
  \lim_{z\rightarrow \infty\wedge \Re z<0}\Gamma_{DN}^{\SSS}(z,w) = 0,\quad
  w\in \HH.
\end{split}\end{equation*}
We consider this case in Section
\ref{Section: Coupling to GFF with Dirichlet-Neumann boundary condition}
and the exact form of $\Gamma_{DN}$ is given by
\eqref{Formula: G_DN^S = ...}.
It is not invariant with respect to 
all M\"obious automorphisms but it is invariant if the automorphism
preserves the points of change of the boundary conditions, which are
$\pm \infty$ in the strip chart.
\end{example}

We will consider another example ($\Gamma_{\text{tw},b}$) in Section
\ref{Section: Coupling to twisted GFF}

\subsection{Gaussian free field}
\label{Section: Gaussian free field}

\begin{definition}
For some nuclear space of smooth functions $\Hc$, let   
the linear functional $\eta$ and some Green's functional
$\Gamma$ be given.
Assume in Theorem
\ref{Theorem: Bochner-Minols}
\begin{equation}
  \hat \mu[f] := \exp{\left(-\frac12 \Gamma[f,f]+ i\eta[f] \right)},
  \quad f\in\Hc.
  \label{Formula: GFF chracteristic function 2}
\end{equation}
Then the $\Hc'$-valued random variable $\Phi$ is called the
\emph{Gaussian free field (GFF)}.
\index{Gaussian free field (GFF)}
We will denote it by 
$\Phi(\Hc,\Gamma,\eta)$.
\index{$\Phi(\Hc,\Gamma,\eta)$}
\end{definition}

For convenience, we change the definition of the characteristic function from
(\ref{Formula: hat mu = int exp iPhi dPhi 2}) to
\begin{equation*}
 \hat\phi[f] := 
 \int\limits_{\Phi\in \Hc_s'} e^{\Phi[f]} P^{\Phi}(d\Phi)
 ,\quad \forall f \in \Hc,
\end{equation*}
and (\ref{Formula: GFF chracteristic function 2}) changes to
\begin{equation}
	\hat\phi[f] = e^{\left(\frac12 \Gamma[f,f]+\eta[f] \right)},
	\label{Formula: GFF chracteristic function}
\end{equation}
which is possible for the Gaussian measures.

The \emph{expectation} of a random variable $X[\Phi]$ ($X:\Hc' \map \C$) is
defined as
\begin{equation*}
 \Ev{X}:=\int\limits_{\Phi\in \Hc_s'} X[\Phi] P^{\Phi}(d\Phi).
\end{equation*}
 
An alternative and equivalent (see, for example \cite{Hida2008}) definition of
GFF can be formulated as follows:

\begin{definition}	
The \emph{Gaussian free field} 
$\Phi$ 
is a 
$\Hc'$-valued 
random variable, that is a map
$\Phi\colon \Hc	\times \Omega \map \mathbb{R}$ 
(measurable on 
$\Omega$ 
and continuous linear on the nuclear space
$\Hc$),
or a measurable map 
$\Phi\colon \Omega \map \Hc'$,
such that 
$\rm{Law}[\Phi[f]]=N\left(\eta[f],\Gamma[f,f]^{\frac12} \right),
~ f\in \Hc$,
i.e.,
it possesses the properties
\begin{equation*}
  \Ev{\Phi[f]}=\eta[f], \quad \forall~f\in \Hc,
 \end{equation*}
 \begin{equation*}
  \Ev{\Phi[f]\Phi[f]}=\Gamma[f,f]+\eta[f]\eta[f], \quad \forall~f\in \Hc~
\end{equation*}
for Green's bilinear positively defined functional $\Gamma$, and for a linear
functional $\eta$.
\end{definition}

The random variable $\Phi$ introduced this way transforms from one chart to
another according to the pre-pre-Schwarzian rule 
\begin{equation}
 \Phi^{\tilde \psi}[f] = \Phi^{\psi}[f] - 
 \int\limits_{\psi(\Dc)} \left( \mu \log \tau^{-1}{}'(z) + \mu^* \overline{ \log \tau^{-1}{}'(z)} \right) 
 f^{\psi}(z) l(dz),\quad \tau:=\psi \circ \tilde \psi^{-1},
 \label{Formula: Phi^tilde psi[f] = Phi^psi [f] - mu in log tau f ldz}
\end{equation}
due to the corresponding property
\eqref{Formula: eta^tilde psi[f] = eta psi[f] - mu int ...}
of $\eta$.

The pushforward can also be defined by
\begin{equation*}\begin{split}
	&G_*^{-1} \Phi^{\psi} [f] = 
	\Phi^{\psi}[G_* f] 
	+	\int\limits_{\supp f^{\psi}} 
	\left( 
		\mu \log G^{\psi}{}'(z) +
		\mu^{*}	\overline{\log G^{\psi}{}'(z)}
	\right)
	f^{\psi}(z) l(dz),\\
	& f\in\Hc_s\colon\quad \supp f\subset G^{-1}(\Dc)
\end{split}\end{equation*}
as well as the Lie derivative becomes
\begin{equation*}\begin{split}
	\Lc_v \Phi[f] =& 
	\left. \frac{\de}{\de s} {H_s^{-1}[v]}_* \Phi^{\psi} [f] \right|_{s=0} 
	=\\=&
	-\Phi^{\psi}[\Lc_{v} f] 
	+ \int\limits_{\supp f^{\psi}} 
	\left(
		\mu v^{\psi}{}'(z) + \mu^* \overline{v^{\psi}{}'(z)} 
	\right)
	f^{\psi}(z) l(dz),
\end{split}\end{equation*}

\begin{example}
Let $\Hc:=\Hc_s$,
$\Gamma := \Gamma_D$ (as in Example \ref{Formula: Gamma_D = Log...}), and let
$\eta^{\psi}(z) := 0$ in all charts $\psi$ ($\mu=\mu^*=0$). 
Then we call $\Phi(\Hc_s,\Gamma_D,0)$ the Gaussian free field with \emph{zero
boundary condition}.
\end{example}

\begin{example} 
Relax the previous example. Let $\eta^{\psi}$ be a harmonic function in
$D^{\psi}$ continuously extendable to the boundary $\de D^{\psi}$ if the
chart map $\psi$ can be extended to $\de\Dc$.
Then 
we call $\Phi$ the Gaussian free field with the \emph{Dirichlet boundary condition}.  
\end{example}

We can define the Laplace-Beltrami operator $\Delta_{\lambda}$ over $\Phi$ 
as well as the Lie derivative by
\begin{equation*}
	(\Delta_{\lambda} \Phi)[g]:=\Phi[\Delta_{\lambda} g]
	,\quad g\in\Hc,
\end{equation*}
where $\Delta_{\lambda}$ on a ($1,1$)-differential is defined by
\begin{equation*}
	\Delta_{\lambda} g^{\psi}(z):=-4\de_z \de_{\bar z}
	\frac{g^{\psi}(z)}{\lambda^{\psi}(z)}
\end{equation*}
in any chart $\psi$. 
If $\eta$ is harmonic the identity 
\begin{equation}\begin{split}
	&\Ev{(\Delta_{\lambda} \Phi)[g] \Phi[f_1]\Phi[f_2]\dotso\Phi[f_n]}
	=\\=&
	\sum\limits_{i=1,2,\dotso,n}
	\delta_{\lambda}[g,f_i] \,
	\Ev{\Phi[f_1]\Phi[f_2]\dotso \Phi[f_{i-1}]\Phi[f_{i+1}]...\dotso\Phi[f_n]}
	\label{Formula: E[DD Phi ] = delta}
\end{split}\end{equation}
is satisfied. Thereby, one can write heuristically
\begin{equation*}
	\Delta_{\lambda} \Phi(z) = 0
	,\quad z\not\in\supp f_1 \cup \supp f_2 \cup\dotso \cup\supp f_n.   
\end{equation*}

It turns out that the characteristic functional $\hat\phi$ is also a derivation functional for the  correlation functions.
Define the variational derivative over some functional $\nu$ as a map
$\frac{\delta}{\delta f}\colon  \nu\mapsto \frac{\delta}{\delta f} \nu$ to the set of functionals by 
\begin{equation*}
 \left(\frac{\delta}{\delta f} \nu \right) [g]:= 
  \left.\frac{\de}{\de \alpha} \nu [g+\alpha f] \right|_{\alpha=0},\quad \forall f,g\in \Hc.
\end{equation*}

If $\nu$ is such that $\nu[g+\alpha f]$ is an analytic function with respect to $\alpha$ for each $f$ and $g$, like $\hat \phi$, 
it is straightforward to see that for each $g,f_1,f_2,\dotso\in\Hc$,
\begin{equation*}
 \nu[g] = \nu[0],\quad g\in \Hc  
 \quad \Leftrightarrow \quad
 \left( \frac{\delta }{\delta f_1} \frac{\delta }{\delta f_2} \dotso \frac{\delta }{\delta f_n}
  \nu\right)[0] = 0,\quad n=1,2,\dotso.
\end{equation*}
Define the Schwinger functionals as
\begin{equation}
  S_n[f_1,f_2,\dotso ,f_n]:=
  \Ev{\Phi[f_1]\Phi[f_2]\dotso\Phi[f_n]}=
  \left( \frac{\delta }{\delta f_1} \frac{\delta }{\delta f_2} \dotso \frac{\delta }{\delta f_n}
  \hat\phi\right)[0]
 \label{Formula: S[f f...f] = d d...d mu}
\end{equation}
where $\hat\phi\colon \Hc\map \mathbb{R}$ is defined in 
(\ref{Formula: GFF chracteristic function}).

The identity 
\eqref{Formula: E[DD Phi ] = delta}
can be reformulated as 
\begin{equation*}
	\Ev{(\Delta_{\lambda}\Phi)[g] e^{\Phi[f]}} =
	\left(
		\delta_{\lambda}[g,f] + \eta [\Delta_{\lambda} f]
	\right) \hat \phi[f]
	,\quad f,g\in\Hc.
\end{equation*}

\subsection{The Schwinger functionals}

In this section, we consider the Schwinger functionals defined by
(\ref{Formula: S[f f...f] = d d...d mu}) 
and their derivation functional 
$\hat \phi$ in detail.

For any finite collection $\{f_1,f_2,\dotso,f_n\}$ of functions from $\Hc_s$ or
$H_{\Gamma}$, the collection of random variables 
$\{ \Phi[f_1],\Phi[f_2],\dotso,\Phi[f_n] \}$ 
has the multivariate normal distribution. Thus, we have
\begin{equation*}
  \Ev{\Phi[f_1]\Phi[f_2]\dotso\Phi[f_n]}=\sum\limits_{\text{partitions}} 
  \prod \limits_k \Gamma[f_{i_k},f_{j_k}]~,
\end{equation*}
for 
$\eta(z)\equiv 0$, 
where the sum is taken over all partitions of the set 
$\{1,2\dotso,n\}$ 
into disjoint pairs 
$\{i_k,j_k\}$.
In particular, the expectation of the product of an odd number of  fields is
identically zero.
For the general case ($\eta\not\equiv 0$) the Schwinger functionals are   
\begin{equation*}
 S[f_1,f_2,\dotso,f_n]:=
 \Ev{\Phi[f_1]\Phi[f_2]\dotso\Phi[f_n]}=
 \sum\limits_{\text{partitions}} 
   \prod \limits_k \Gamma[f_{i_k},f_{j_k}] \prod \limits_l \eta[f_{i_l}],
\end{equation*}
where the sum is taken over all partitions of the set 
$\{1,2\dotso,n\}$ 
into
disjoint non-ordered pairs 
$\{i_k,j_k\}$, 
and non-ordered single elements
$\{i_l\}$.
In particular,
\begin{equation*}\begin{split}
 S_1[f_1]=&\eta[f_1],\\
 S_2[f_1,f_2]=&\Gamma[f_1,f_2] + \eta[f_1]\eta[f_2],\\
 S_3[f_1,f_2,f_3]=&\Gamma[f_1,f_2]\eta[f_3] + 
  \Gamma[f_3,f_1] \eta[f_2] + 
  \Gamma[f_2,f_3] \eta[f_1]+
  \eta[f_1] \eta[z_2] \eta[f_3],\\
 S_4[f_1,f_2,f_3,f_4]=&
 \Gamma[f_1,f_2]\Gamma[f_3,f_4]+\Gamma[f_1,f_3]\Gamma[f_2,f_4]+\Gamma[f_1,f_4]\Gamma[f_2,f_3]
 +\\+&
 \Gamma[f_1,f_2]\eta[f_3]\eta[f_4]+\Gamma[f_1,f_3]\eta[f_2]\eta[f_4]+\Gamma[f_1,f_4]\eta[f_2]\eta[f_3]
 +\\+&
 \eta[f_1]\eta[f_2]\eta[f_3]\eta[f_4].
\end{split}\end{equation*}
Such correlation functionals 
are called the \emph{Schwinger functionals}. Their kernels 
\begin{equation*}
	S_n(z_1,z_2,\dotso,s_n)
\end{equation*}
are known as Schwinger functions or $n$-point
functions.
For regular functionals $\Gamma$ and $\eta$, the Schwinger functions are also
regular but it is still reasonable to understand $S_n$ as a functional because
the derivatives are not regular. For example,
\begin{equation*}
	{\Delta_{\lambda}}_1 S_2^{\psi}(z,w) = 2\pi \delta_{\lambda}(z-w).
\end{equation*}

The transformation rules for $S_n$ (the behaviour under the action of $G_*$) 
are quite complex. We present here only the infinitesimal ones
\begin{equation*}\begin{split}
 &\Lc_v S_n^{\psi}[f_1,f_2,\dotso] =
 -\sum\limits_{1\leq k\leq n} S_n^{\psi}[f_1,f_2,\dotso \Lc_v f_k,\dotso, f_n] 
 -\\-&
 \sum\limits_{1\leq k\leq n}
 S_{n-1}^{\psi}[f_1,f_2,\dotso f_{k-1},f_{k+1},\dotso f_{n-1}] 
 \int\limits_{\psi(\Dc)} \left( \mu v^{\psi}{}'(z) + \mu^* \overline{ v^{\psi}{}'(z)} \right) 
 f_k^{\psi}(z) l(dz) 
\end{split}\end{equation*}

We prefer to work with the characteristic functional $\hat\phi$, rather
than with $S_n$. For instance, for any inverse endomorphism 
$G\colon \Dc\setminus\K \map \Dc$, 
we can define the pushforward 
$G^{-1}_*\colon \hat\phi (\Gamma, \eta) \mapsto \hat\phi (F_*\Gamma,F_*\eta)$ 
that maps the functionals 
on $\Dc$ to functionals on $\Dc\setminus\K$. Equivalently,
\begin{equation}
	\left( G^{-1}_* \hat\phi (\Gamma, \eta) \right) [\tilde f] :=
	\hat\phi(G^{-1}_*\Gamma, G^{-1}_* \eta)[\tilde f],
	\quad \tilde f\in \Hc_s[\tilde \Dc]
 \label{Formula: F phi(Gamma, eta) = phi(F^-1 Gamma, F^-1 eta)}
\end{equation}
(we need to mark the dependence on the functionals $\Gamma$ and on $\eta$ here).

The Lie derivative $\Lc_v$ over an arbitrary nonlinear functional 
$\rho\colon \Hc_s\map \mathbb{C}$ 
can be also  defined as
\begin{equation*}
 \Lc_v \rho[f]:=(\Lc_v \rho)[f]=
 \frac{\de}{\de t}\left. (H_{\alpha}[v]_*^{-1} \rho)[f] \right|_{t=0}
\end{equation*} 
(if the partial derivative w.r.t. $\alpha$ 
is well-defined).

For example,
\begin{equation*}
 \Lc_v \exp{\left( \rho[f] \right)} = 
 (\Lc_v \rho[f]) \exp{\left( \rho[f] \right)},
\end{equation*}
\begin{equation*}
 \Lc_v^2 \exp{\left(\rho[f] \right)} = 
 \left(\Lc_v^2 \rho[f] + (\Lc_v \rho[f])^2 \right) \exp{\left( \rho[f] \right)}.
\end{equation*}
In our case $\rho[f] = \hat \phi[f]= \exp(\frac12 \Gamma[f,f] + \eta[f])$.
We remind that the Lie derivative of $\eta$ and $\Gamma$ are defined in 
(\ref{Formula: L eta = v d eta + chi dv})
and
(\ref{Formula: L Gamma = ...})
respectively. 

The operations $G_*^{-1}$ and $\frac{\delta}{\delta f}$ or 
$\Lc$ and $\frac{\delta}{\delta f}$ 
commute. Thus, for example, we have
\begin{equation*}
 \Lc_v S_n[f_1,f_2,\dotso,f_n] = 
 \left( \frac{\delta }{\delta f_1} \frac{\delta }{\delta f_2} \dotso \frac{\delta }{\delta f_n}
  \Lc_v \hat \phi \right)[0].
\end{equation*}
We use this to deduce the martingale properties of $G_t^{-1}{}_*S_n$ and of all
their variational derivatives from the martingale property of
$G_t^{-1}{}_*\hat\phi$, which will be discussed in the next section.

\section{Coupling between SLE and GFF}
\label{Section: Coupling between SLE and GFF}

Let 
$(\Omega^{\Phi},\mathcal{F}^{\Phi},P^{\Phi})$ 
be the probability space
for GFF 
$\Phi$
and let
$(\Omega^{B},\mathcal{F}^{B},P^{B})$ 
be the independent probability space for
the Brownian motion 
$\{B_t\}_{t\in[0,+\infty)}$, 
which governs some
$(\delta,\sigma)$-SLE 
$\{G_t\}_{t\in[0,+\infty)}$. 
In this section, we consider a coupling between
these random laws. 

The pushforward 
$G_t^{-1}{}_*\Phi[f]$
of the GFF
$\Phi[f]$
is well-defined if
$\supp f \in {\rm image}[G^{-1}_t]$. 
In order to handle this,  we introduce a stopping time 
$T[f]$, for which the hull 
$\K_t$ of 
($\delta,\sigma$)-SLE touches 
some small neighborhood $U(\supp f)$ of the support of $f$ for the first time:
\index{$T[f]$}
\begin{equation}
	T[f]:=\sup \{t>0\colon \K_t\cap U(\supp f) 
	= \emptyset\},\quad f\in\Hc.
	\label{Formula: T[f] = ...}
\end{equation}
The neighborhood $U(\supp f)$ can be defined, for example, as the set of points  from $\supp f$
with the Poincare distance less than some $\varepsilon>0$. Thus,
$T[f]>0$ a.s.
We consider a stopped process 
$\{G_{t\wedge T[f]}\}_{t\in[0,+\infty)}$. This approach was also used in 
\cite{Izyurov2010}.
The most important  property of the process
$\{{G_{t\wedge T[f]}^{-1}}_*\Phi[f]\}_{t\in[0,+\infty)}$ 
is that it is a local martingale.
A stopped local martingale is also a local martingale. That is why a
stopping of 
$\{G_t\}_{t\in[0,+\infty)}$
does not change our results. However, we lose some information,
which makes the proposition of coupling less substantial than one possibly
expects.

We present here two definitions of the coupling.
The first one is similar to \cite{Sheffield2010,Izyurov2010}.
The second one is a weaker statement that we shall use in this paper.

\begin{definition}
A GFF 
$\Phi(\Hc,\Gamma,\eta)$
is called \emph{coupled}
to the forward or reverse
$(\delta,\sigma)$-SLE, driven by $\{B_t\}_{t\in[0,+\infty)}$, if
the random variable ${G^{-1}_{t\wedge T[f]}}_*\Phi^{\psi}[f]$
obtained by independent sampling of $\Phi$ and $G_t$ has the same law as
$\Phi^{\psi}[f]$ for any test function $f\in\Hc$, chart map $\psi$, and
$t \in[0,+\infty)$.
\end{definition}

If the  coupling holds for a fixed chart map $\psi$ and for any $f\in\Hc$,
then it also holds for any chart map $\tilde \psi$,
due to
(\ref{Formula: Phi^tilde psi[f] = Phi^psi [f] - mu in log tau f ldz}).
We also give a weaker version of the coupling statement that we plan to use here. 
To this end, we have
to consider a stopped versions of the stochastic process 
$\{G_{t\wedge T[f]}\}_{t\in[0,+\infty)}$.

A collection of stopping
times $\{T_n\}_{n=1,2,\dotso}$ is called a
\emph{fundamental sequence}
if $0\leq T_n\leq T_{n+1}\leq \infty$, $n=1,2,\dotso$ a.s., and
$\lim\limits_{n\map \infty} T_n=\infty$ a.s.

A stochastic process $\{x_t\}_{t\in[0,+\infty)}$ is called a 
\emph{local martingale}
if there exists a fundamental sequence of stopping times
$\{T_n\}_{n=1,2,\dotso}$, such that the stopped process
$\{x_{t\wedge T_n}\}_{t\in[0,+\infty)}$ 
is a martingale for each
$n=1,2,\dotso$.

Let now the statement of coupling above be valid only for the  process
$\{G_{t\wedge {T[f]} \wedge T_n}\}_{t\in[0,+\infty)}$ stopped by $T_n$
for each $n=1,2,\dotso$.
Namely,
${G^{-1}_{t\wedge {T[f]} \wedge T_n}}_*\Phi^{\psi}[f]$
has the same law as
$\Phi^{\psi}[f]$ for each $n=1,2,\dotso$.

We are ready now to define the local coupling.

\begin{definition}
A GFF 
$\Phi(\Hc,\Gamma,\eta)$
is called \emph{locally coupled}
\index{local coupling}
to 
$(\delta,\sigma)$-SLE, driven by $\{B_t\}_{t\in[0,+\infty)}$, if
there exists a fundamental sequence 
$\{T_n[f,\psi]\}_{n=1,2,\dotso}$, such that
the random variable ${G^{-1}_{t\wedge T[f]}}_*\Phi^{\psi}[f]$
obtained by independent sampling of $\Phi$ and $G_t$ has the same law as
$\Phi^{\psi}[f]$ until the
stopping time $T_n[f,\psi]$ for each $n=1,2,\dotso$, for any test function
$f\in\Hc$, and a chart map $\psi$.
\end{definition}

\begin{remark}
If 
$\mathbb{T}[f,\psi]=+\infty$ a.s. 
for each
$f\in\Hc$,
then the coupling is not local.
\end{remark}

The following theorem generalizes the result of 
\cite{Sheffield2010}.

\begin{theorem}
\label{Theorem: The coupling theorem}
The following three statements are equivalent:
\begin{enumerate} [1.]
\item
GFF
$\Phi(\Hc,\Gamma,\eta)$
is locally coupled to
$(\delta,\sigma)$-SLE;
\item
${G_{t \wedge T[f]}^{-1}}_* \hat\phi^{\psi}[f]$ is a local martingale for
$f\in\Hc$
in any chart $\psi$;
\item
The system of the equations
 \begin{equation}\begin{split}
  \Lc_{\delta} \eta[f] + \frac12\Lc_{\sigma}^2 \eta[f] = 0
  ,\quad f\in\Hc,
  \label{Formula: L eta + 1/2 L^2 eta}
 \end{split}\end{equation}
 \begin{equation}
  \Lc_{\delta} \Gamma[f,g] + \Lc_{\sigma} \eta[f] \Lc_{\sigma} \eta[g] = 0
  ,\quad f,g\in\Hc,
  \label{Formula: Hadamard's formula}
 \end{equation}
 and
 \begin{equation}
  \Lc_{\sigma} \Gamma[f,g] = 0
  ,\quad f,g\in\Hc.
  \label{Formula: L_sigma Gamma = 0}
 \end{equation}
is satisfied.
 \end{enumerate}
\end{theorem}

We start the proof after some remarks.
Just for clarity (but not for applications) we reformulate the system
(\ref{Formula: L eta + 1/2 L^2 eta}--\ref{Formula: L_sigma Gamma = 0})
directly in terms of partial derivatives using
(\ref{Formula: L eta = v d eta + chi dv}),
(\ref{Formula: L L phi = w de L phi}),
(\ref{Formula: L Gamma = ...}),
and
(\ref{Formula: A = L + 1/2 L^2})
as
\begin{equation*}\begin{split}
 &\delta(z) \de_z \eta(z) + \overline{\delta(z)} \de_{\bar z} \eta(z)
 + \mu \delta'(z) + \mu^* \overline{\delta'(z)}
 +\\+&
 \frac12 \sigma^2(z) \de_z^2 \eta(z) 
 + \frac12 \overline{\sigma^2(z)} \de_{\bar z}^2 \eta(z)
 + \sigma(z) \overline{\sigma(z)} \de_z \de_{\bar z} \eta
 +\\+&
 \frac12 \sigma(z) \sigma'(z) \de_z \eta(z)
 + \frac12 \overline{\sigma(z)} \overline{\sigma'(z)} \de_{\bar z} \eta(z)
 + \mu \sigma(z) \sigma''(z) + \mu^* \overline{ \sigma(z) \sigma''(z) } = 0
 ;\\
 &\delta(z) \de_z \Gamma(z,w) + \delta(w) \de_w \Gamma(z,w)
 +\overline{\delta(z)} \de_{\bar z} \Gamma(z,w) 
 + \overline{\delta(w)} \de_{\bar w} \Gamma(z,w)
 +\\+&
 \left(
  \sigma(z) \de_z \eta(z) + \overline{\sigma(z)} \de_{\bar z} \eta(z)
  + \mu \sigma'(z) + \mu^* \overline{ \sigma'(z) }
 \right)
 \times \\ \times &
 \left(
  \sigma(w) \de_w \eta(w) + \overline{\sigma(w)} \de_{\bar w} \eta(w)
  + \mu \sigma'(w) + \mu^* \overline{ \sigma'(w) }
 \right) = 0;\\
 &\sigma(z) \de_z \Gamma(z,w) + \sigma(w) \de_w \Gamma(z,w)
 +\overline{\sigma(z)} \de_{\bar z} \Gamma(z,w) 
 + \overline{\sigma(w)} \de_{\bar w} \Gamma(z,w) 
 = 0,
\end{split}\end{equation*}
where we drop  the upper index $\psi$ for shortness.

The first equation
(\ref{Formula: L eta + 1/2 L^2 eta})
is just a local martingale condition for $\eta$. The second one 
\eqref{Formula: Hadamard's formula} 
is a special case of Hadamard's variation formula, where the variation is concentrated at one point at the boundary. 
The third equationmeans that $\Gamma$ should be invariant under the one-parametric family of  M\"obius automorphisms generated by $\sigma$.

\medskip
\noindent
{\it Proof of  Theorem \ref{Theorem: The coupling theorem}}.
Let us start with showing how the statement 1 about the coupling implies the
statement 2 about the local martingality.

\medskip
\noindent
\textbf{1.$\Leftrightarrow$2.}
Let 
$G_{t\wedge T_f \wedge T_n[f,\psi]}$ 
be a stopped process $G_{t\wedge T_f}$ by the stopping times
$T_n[f,\psi]$ 
forming some fundamental sequence. 
The coupling statement can be reformulated as an equality of characteristic
functions for the random variables 
${G^{-1}_{t\wedge \tilde T_n[f,\psi]}}_*\Phi^{\psi}[f]$
and
$\Phi^{\psi}[f]$
for all test functions $f$.
Namely, the following expectations must be equal
\begin{equation*}
 \Evv{B}{ \Evv{\Phi} 
 {e^{ {G^{-1}_{t\wedge T[f]\wedge T_n[f,\psi]}}_*\Phi[f] } } }
 = \Evv{\Phi}{e^{\Phi[f]}}
 ,\quad f\in \Hc,\quad t\in[0,+\infty),\quad n=1,2,\dotso~,
\end{equation*}
which in particular, means the integrability of
$e^{{G^{-1}_{t\wedge \tilde T_n[f,\psi]}}{}_*\Phi[f]}$ 
with respect to 
$\Omega^{B}$ 
and 
$\Omega^{\Phi}$.
We used 
$\Evv{B}{\cdot}$ 
for the expectation with respect to the random law of
$\{B_t\}_{t\in[0,+\infty)}$ 
(or $\{G_t\}_{t\in[0,+\infty)}$) 
and
$\Evv{\Phi}{\cdot}$ 
for the expectation with respect to $\Phi$.
Let us use 
(\ref{Formula: GFF chracteristic function})
and
(\ref{Formula: F phi(Gamma, eta) = phi(F^-1 Gamma, F^-1 eta)})
to simplify this identity to
\begin{equation*}
	\Evv{B}{{G^{-1}_{t\wedge T[f]\wedge T_n[f,\psi]}}_* \hat\phi^{\psi}[f]} =
	\hat\phi^{\psi}[f],\quad f\in \Hc
	,\quad t\in[0,+\infty),\quad n=1,2,\dotso~.
\end{equation*}
After substituting 
$f\map \tilde G_{s \wedge T[f]}{}_*f$ 
for some independently
sampled 
$\tilde G_s$ 
and
$s\in[0,+\infty)$, 
we obtain
\begin{equation*}\begin{split}
	&\Evv{B}{G^{-1}_{t\wedge T[\tilde G_{s \wedge T[f]}{}_* f]
		\wedge T_n[\tilde
	G_{s \wedge T[f]}{}_*f,\psi]}{}_* 
	\hat\phi^{\psi}[\tilde G_{s \wedge T[f]}{}_*f]} 
	= \hat\phi^{\psi}[\tilde G_{s \wedge T[f]}{}_*f],\\
	&f\in \Hc
	,\quad t\in[0,+\infty),\quad n=1,2,\dotso.
\end{split}\end{equation*}
Multiplying both sides by 
\begin{equation*}
\supp
	e^{\int\limits_{~\supp f^{\psi}} 
	\left( 
		\mu \log \left(\tilde G_{s \wedge T[f]}^{\psi}\right){'}(z) +
		\mu^{*}	\overline{\log \left(\tilde G_{s \wedge T[f]}^{\psi}\right){'}(z)}
	\right)
	f(z) l(dz) }
\end{equation*}
and by making use of 
\eqref{Formula: F eta f= eta F-1 f}
and
\eqref{Formula: GFF chracteristic function},
we conclude that
\begin{equation}\begin{split}
	&\Evv{B}{
	\tilde G_{s \wedge T[f]}^{-1}{}_*
	G^{-1}_{t\wedge T[\tilde G_{s \wedge T[f]}{}_* f]
		\wedge T_n[\tilde
	G_{s \wedge T[f]}^{-1}{}_*f,\psi]}{}_* 
	\hat\phi^{\psi}[f]} 
	= \tilde G_{s \wedge T[f]}{}_* \hat\phi^{\psi}[f],\\
	&f\in \Hc
	,\quad t\in[0,+\infty),\quad n=1,2,\dotso.
	\label{Formula: 5}
\end{split}\end{equation}
Defined now the process
\begin{equation*}
	\tilde {\tilde G}_{t+s} 
	:= G_{t} \circ \tilde G_{s}
	,\quad s,t\in[0,+\infty),
\end{equation*}
which has the law of ($\delta,\sigma$)-SLE.
Its stopped version possesses the identity
\begin{equation*}
	G_{t+s\wedge T[f]} 
	= G_{t\wedge T[\tilde G_{s \wedge T[f]}{}_* f]} \circ \tilde G_{s\wedge T[f]}
	,\quad s,t\in[0,+\infty),\quad f\in\Hc.
\end{equation*}
The left-hand side of
\eqref{Formula: 5}
is equal to
\begin{equation*}\begin{split}
	&\Evv{B}{
	\left(
		{G_{t\wedge T[\tilde G_{s \wedge T[f]}{}_* f]
		\wedge T_n[\tilde
		G_{s \wedge T[f]}^{-1}{}_*f,\psi]}} 
		\circ \tilde G_{s \wedge T[f]}
	\right)_*^{-1} 
	\hat\phi^{\psi}[f]}
	=\\=&
	\Evv{B}{
	\left(
		\tilde {\tilde G}_{t+s\wedge T[f]
		\wedge T_n[\tilde	G_{s \wedge T[f]}^{-1}{}_*f,\psi]+s } 
	\right)_*^{-1} 
	\hat\phi^{\psi}[f]~|~
	\mathcal{F}^B_{s\wedge T[f]}}.
\end{split}\end{equation*}
We use now the Markov property of ($\delta,\sigma$)-SLE and conclude that
$T_n'[f,\psi]:= T_n[\tilde	G_{s \wedge T[f]}^{-1}{}_*f,\psi]+s$
is a fundamental sequence for the pair of $f$ and 
$\psi$. 
Thus, 
\eqref{Formula: 5}
simplifies to
\begin{equation*}
	\Evv{B}{
		G_{t+s\wedge T[f]
		\wedge T_n'[f,\psi] }^{-1} {}_* \hat\phi^{\psi}[f]~|~
	\mathcal{F}^B_{s\wedge T[f]\wedge T_n'[f,\psi]}}
	=G_{t+s\wedge T[f]\wedge T_n'[f,\psi] }^{-1}{}_* \hat \phi[f],
\end{equation*} 
hence, 
$\{G_{t\wedge T[f] }^{-1}{}_* \hat \phi[f]\}_{t\in[0,+\infty)}$ 
is a local martingale.

The inverse statement can be obtained by the same method in the reverse
order.

\medskip
\noindent
\textbf{2.$\Leftrightarrow$3.}
According to Proposition~\ref{Proposition: G chi[f] is an It\^{o} process}, Appendix~A,
the drift term, i.e., the coefficient at $dt$, vanishes identically when
\begin{equation}
 \Ac W[f] +  \frac12 \left( \Lc_{\sigma} W[f] \right)^2 = 0,\quad f\in\Hc.
 \label{Formula: A W + 1/2 L W L W = 0}
\end{equation}
The left-hand side is a functional polynomial of degree 4.
We use the fact that a regular symmetric functional 
$P[f] := \sum\limits_{k=1,2,\dotso n}
p_k[f,f,\dotso,f]$
of degree $n$ over such spaces as $\Hc_s$,
$\Hc_s^*$, $\Hc_{s,b}$, 
or
$\Hc^{\pm}_{s,b}{}^{*}$
is identically zero if and only if 
\begin{equation*}
	p_k[f_1,f_2,\dotso,f_n] = 0,\quad k=1,2,\dotso n, \quad f\in \Hc.
\end{equation*}
Thus, each of the following functions must be identically zero:
\begin{equation}\begin{split}
	&\Ac \eta[f] = 0,\quad
	\frac12 \Ac \Gamma[f,g]
	+ \frac12 \Lc_{\sigma} \eta[f] \Lc_{\sigma} \eta[g] = 0,\\
	& \Lc_{\sigma} \eta[f] \Lc_{\sigma} \Gamma[g,h] + \text{symmetric terms} = 0,\\
	&\Lc_{\sigma} \Gamma[f,g] \Lc_{\sigma} \Gamma[h,l]  + \text{symmetric terms} =0,\\
	& f,g,h,l\in \Hc.
	\label{Formula: 6}
\end{split}\end{equation}
We can conclude that $\Lc_{\sigma} \Gamma[f,g]= 0$,
$\Ac \Gamma[f,g] = \Lc_{\delta}\Gamma[f,g]$
for any $f,g\in\Hc$,
and this system is equivalent to the system
(\ref{Formula: L eta + 1/2 L^2 eta}--\ref{Formula: L_sigma Gamma = 0}).
For the case $\Hc=\Hc_s$ we can write \eqref{Formula: 6} in terms of functions
on $\psi(\Dc)$:
\begin{equation*}\begin{split}
	&\Ac \eta(z) = 0,\quad
	\frac12 \Ac \Gamma(z,w)
	+ \frac12 \Lc_{\sigma} \eta(z) \Lc_{\sigma} \eta(w) = 0,\\
	& \Lc_{\sigma} \eta(z) \Lc_{\sigma} \Gamma(w,u) 
	+ \text{symmetric terms} = 0,\\
	& \Lc_{\sigma} \Gamma(z,w) \Lc_{\sigma} \Gamma(u,v) 
	+ \text{symmetric terms} =0,\\
	& z,w,u,v\in \psi(\Dc)
	,\quad z \neq w,~u \neq v, \dotso~.
\end{split}\end{equation*}

\begin{remark}
Fix a chart $\psi$.
The coupling and the martingales are not local if in addition to the
proposition 3 in Theorem
\ref{Theorem: The coupling theorem}
the relation
\begin{equation*}
 \Evv{B}{ \left| \int\limits_0^t \exp
 	\left( {G_{\tau \wedge T[f]}^{-1}}_* W^{\psi}[f] \right) 
 	{G_{\tau \wedge T[f]}^{-1}}_* \Lc_{\sigma}
 	W^{\psi}[f] \dI B_{\tau} \right| } <\infty,\quad
 t\geq 0,
\end{equation*}
holds. This is the condition that the diffusion term at
$\dI B_t$ in
(\ref{Formula: G chi = chi + int ...})
is in $L_1(\Omega^{B})$.
However, this may not be true, in general, in another chart $\tilde \psi$.
Meanwhile, if the local martingale property of
${G_{t \wedge T[f]}^{-1}}_* \hat\phi^{\psi}[f]$
is satisfied in one chart $\psi$ for any $f\in\Hc$, then it is also true in
any  chart due to the invariance of the condition 
\eqref{Formula: A W + 1/2 L W L W = 0} in the proof.
\end{remark}

The study of the general solution to 
(\ref{Formula: L eta + 1/2 L^2 eta}-\ref{Formula: L_sigma Gamma = 0})
is an interesting and complicated problem. Take the Lie derivative
$\Lc_{\sigma}$ 
in the second equation, the Lie derivative 
$\Lc_{\delta}$
in the third equation, and consider the difference of the resulting equations.
It is an algebraically independent equation
\begin{equation*}
	\Lc_{[\delta,\sigma]} \Gamma[f,g] = 
	- \Lc_{\sigma}^2 \eta[f] \Lc_{\sigma} \eta[g]
	\Lc_{\sigma}^2 \eta[f] \Lc_{\sigma}^2 \eta[g].
\end{equation*}
Continuing by induction we obtain an infinite system of a priori algebraically
independent equations because the Lie algebra induced by the vector fields
$\delta$ and $\sigma$ is infinite-dimensional. Thereby, the existence of the solution to the system
(\ref{Formula: L eta + 1/2 L^2 eta}-\ref{Formula: L_sigma Gamma = 0})
is a special event that is strongly related to the properties of this algebra. 

Before studying special solutions to the system
(\ref{Formula: L eta + 1/2 L^2 eta}--\ref{Formula: L_sigma Gamma = 0}),
let us consider some of its general properties. We also reformulate it in terms
of the analytic functions $\eta^+$, $\Gamma^{++}$ and $\Gamma^{+-}$, which is
technically more convenient.

\begin{lemma}
\label{Lemma: eta structure}
Let $\delta$, $\sigma$, $\eta$, and $\Gamma$ be such that the system
(\ref{Formula: L eta + 1/2 L^2 eta}--\ref{Formula: L_sigma Gamma = 0})
is satisfied, let
$\Gamma$ be a fundamental solution to the Laplace equation
(see (\ref{Formula: Gamma = Log + H})),
which transforms as a scalar,
see (\ref{Formula: G B(z,w) = B(G(z),G(w))}), 
and let $\eta$ be a pre-pre-Schwrazian. Then,
\begin{enumerate}[1.]
\item
$\eta$ is a $(i\chi/2,-i\chi/2)$-pre-pre-Schwarzian
\eqref{Formula: tilde eta = eta - chi arg}
given by a harmonic function in any chart with $\chi$ given by
\eqref{Formula: chi = 2/k - k/2}.
\item
The boundary value of $\eta$ undergoes a jump $2\pi/\sqrt{\kappa}$ at the
source point $a$, namely,  its local behaviour in the half-plane chart is given
by 
\eqref{Formula: eta = -2/k (arg z - pi/2) + hol} 
up to a sign;
\item
The system
\eqref{Formula: L eta + 1/2 L^2 eta}--\ref{Formula: L_sigma Gamma = 0}
is equivalent to the system
\eqref{Formula: eta = eta^+ + eta^-},
\eqref{Formula: delta/sigma j + mu [sigma,delta]/sigma + 1/2 L j = ibeta},
\eqref{Formula: Gamma = Gamma^++ + bar Gamma^++ + Gamma^+- + bar Gamma^+-},
\eqref{Formula: L Gamma^++ + L sigma^+ L sigma^+ = e + e},
and
\eqref{Formula: L Gamma^++ = 0, L Gamma^+- = 0}.
\end{enumerate}
\end{lemma}

\begin{proof}
The system
(\ref{Formula: L eta + 1/2 L^2 eta}--\ref{Formula: L_sigma Gamma = 0})
defines $\eta$ only up to an additive constant $C$ that we keep writing in
the formulas for $\eta$ below. The condition for the pre-pre-Schwarzian $\eta$
to be real leads to only two possibilities:
\begin{enumerate} [1.]
 \item $\mu=-\mu^*$ and is pure imaginary as in
 (\ref{Formula: tilde eta = eta - chi arg});
 \item $\mu=\mu^*$ and is real as in
 (\ref{Formula: tilde eta = eta + gamma log}).
\end{enumerate}

The equation
\eqref{Formula: Hadamard's formula} 
shows that the functional $\Lc_{\sigma} \eta$ has to be given by a harmonic
function as well as $\Lc_{\sigma}^2 \eta$ in any chart. On the other hand,
\eqref{Formula: L eta + 1/2 L^2 eta} 
implies that $\Lc_{\delta}\eta$ is also harmonic. The vector fields $\delta$
and $\sigma$ are transversal almost everywhere. We conclude that $\eta$ is
harmonic. We used also the fact that the additional $\mu$-terms in 
\eqref{Formula: L eta = v d eta + chi dv} 
are harmonic.

The harmonic function
$\eta^{\psi}(z)$ 
can be represented as a sum of an analytic function 
${\eta^+}^{\psi}(z)$ 
and its complex conjugate in any chart $\psi$
\begin{equation}
 \eta^{\psi}(z) = {\eta^+}^{\psi}(z) + \overline{{\eta^+}^{\psi}(z)}.
 \label{Formula: eta = eta^+ + eta^-}
\end{equation}
Below in this proof, we drop the chart index $\psi$, which can
be chosen arbitrarily.

Let us define $\eta^+$ and $\overline {\eta^+}$ to be pre-pre-Schwarzians of
orders $(\mu,0)$ and $(0,\mu^*)$ respectively by
(\ref{Formula: L eta = v d eta + chi dv}). Thus, $\eta^+$ is defined up to a
complex constant $C^+$.
We denote
\begin{equation}
  j^+ := \Lc_{\sigma} \eta^+.
  \label{Formula: j^+ = L eta^+}
\end{equation}
and
\begin{equation}
  j := \Lc_{\sigma} \eta = \Lc_{\sigma} \eta^+ + \overline{\Lc_{\sigma} \eta^+}.
  \label{Formula: j = L eta}
\end{equation}
The reciprocal formula is
 \begin{equation}\begin{split}
  \eta^+(z) 
  := \int { \frac{j^+(z) - \mu \sigma'(z)}{\sigma(z)} }dz.
	\label{Folmula: phi+ = int ...}
\end{split}\end{equation}
This integral can be a ramified function if $\sigma(z)$ has a zero inside of
$\Dc$ (the elliptic case). We consider how to handle this technical difficulty
in Section
\ref{Section: Radial SLE with drift}.

Let us reformulate now
(\ref{Formula: L eta + 1/2 L^2 eta}) in terms of $j^+$.
Using the fact that
\begin{equation*}
 \Lc_{v}^2 (\eta^+ + \overline{\eta^+}) =
 \Lc_{v}^2 \eta^+ + \Lc_{v}^2 \overline{\eta^+},
\end{equation*}
we arrive at 
\begin{equation}
 \Lc_{\delta} \eta^+ + \frac12 {\Lc_{\sigma}}^2 \eta^+ = C^+.
 \label{Formula: Lc eta^+ + 1/2 L^2 eta^+ = C}
\end{equation}
Here $C^+=i\beta$ for some $\beta\in\mathbb{R}$ for the forward case. For the
reverse case, 
$C^+=-\beta + i \beta'$
for some 
$\beta,\beta'\in\mathbb{R}$
because 
\eqref{Formula: Lc eta^+ + 1/2 L^2 eta^+ = C}
is an identity in sense of functionals over $\Hc_s^*$. 

The relation 
\eqref{Formula: Lc eta^+ + 1/2 L^2 eta^+ = C}
is equivalent to
\begin{equation*}\begin{split}
 &\frac{\delta}{\sigma} \Lc_{\sigma} \eta^+ +
 \frac{\sigma \Lc_\delta \eta^+ - \delta \Lc_\sigma \eta^+}{\sigma}
 + \frac12 {\Lc_{\sigma}}^2 \eta^+ = C^+ \quad
  \Leftrightarrow \\
 &\frac{\delta}{\sigma} j^+
 + \frac{\sigma \delta \de \eta^+ + \mu \sigma \delta'
  - \delta \sigma \de \eta^+ - \mu \delta \sigma'}{\sigma}
 + \frac12 \Lc_\sigma j^+ = C^+ \quad \Leftrightarrow
\end{split}\end{equation*}
\begin{equation}\begin{split}
 &\frac{\delta}{\sigma} j^+
 + \mu \frac{[\sigma, \delta]}{\sigma}
 + \frac12 \Lc_\sigma j^+ = C^+
 \label{Formula: delta/sigma j + mu [sigma,delta]/sigma + 1/2 L j = ibeta}
\end{split}\end{equation}
by
\eqref{Formula: L eta = v d eta + chi dv}
and
\eqref{Formula: L_v w = [v,w] := v w' - v' w}.

Consider now the function $\Gamma^{\HH}(z,w)$. It is harmonic with respect to
both variables with the only logarithmic singularity. Hence, it can be split as a sum of four terms
\begin{equation}
 \Gamma^{\HH}(z,w) :=
 {\Gamma^{++}}^{\HH}(z,w) + \overline{{\Gamma^{++}}^{\HH}(z,w)}
 - {\Gamma^{+-}}^{\HH}(z,\bar w) - \overline{{\Gamma^{+-}}^{\HH}(z,\bar w)},
 \label{Formula: Gamma = Gamma^++ + bar Gamma^++ + Gamma^+- + bar Gamma^+-}
\end{equation}
where
${\Gamma^{++}}^{\HH}(z,w)$
and
${\Gamma^{+-}}^{\HH}(z,w)$
are analytic with respect to both variables except the diagonal $z=w$ for
${\Gamma^{++}}^{\HH}(z,w)$.

So, e.g.,
$\overline{{\Gamma^{+-}}^{\HH}(z,\bar w)}$
is anti-analytic with respect to $z$ and analytic with respect to $w$.
We can assume that both
${\Gamma^{++}}(z,w)$ and ${\Gamma^{+-}}(z,w)$ transform as scalars represented
by analytic functions in all charts and symmetric with respect to $z
\leftrightarrow w $. Observe that these functions are defined at least up to 
the transform
\begin{equation*}\begin{split}
 &{\Gamma^{++}}^{\HH}(z,w) \rightarrow
  {\Gamma^{++}}^{\HH}(z,w) + \epsilon^{\HH}(z) + \epsilon^{\HH}(w),\\
 &{\Gamma^{+-}}^{\HH}(z,w) \rightarrow
  {\Gamma^{+-}}^{\HH}(z,w) + \epsilon^{\HH}(z) + \epsilon^{\HH}(w)
\end{split}\end{equation*}
for any analytic function $\epsilon^{\HH}(z)$ such that
\begin{equation*}
 \overline{\epsilon^{\HH}(z)} = \epsilon^{\HH}(\bar z).
\end{equation*}
These additional terms are cancelled due to the choice
of minus in the pairs `$\mp$' in 
\eqref{Formula: Gamma = Gamma^++ + bar Gamma^++ + Gamma^+- + bar Gamma^+-}. 
In the reverse case, the contribution of these functions is equivalent to
zero bilinear functional over $\Hc_s^*$.

Consider the  equation
(\ref{Formula: L_sigma Gamma = 0}). It leads to
\begin{equation}\begin{split}
	&\Lc_{\sigma} {\Gamma^{++}}^{\HH}(z,w) = \beta_2^{\HH}(z) +
	\beta_2^{\HH}(w),\quad \Lc_{\sigma} {\Gamma^{+-}}^{\HH}(z,w) = \beta_2^{\HH}(z)
	+ \beta_2^{\HH}(w)
	\label{Formula: L Gamma^++ = beta + beta, L Gamma^+- = beta + beta}
\end{split}\end{equation}
for any analytic function $\beta_2^{\HH}(z)$ such that
$\overline{\beta_2^{\HH}(z)} = \beta_2^{\HH}(\bar z)$.
One can fix this freedom, i.e., the function $\beta_2^{\HH}$, by
the conditions
\begin{equation}\begin{split}
 &\Lc_{\sigma} {\Gamma^{++}}^{\HH}(z,w) = 0,\quad
 \Lc_{\sigma} {\Gamma^{+-}}^{\HH}(z,w) = 0.
 \label{Formula: L Gamma^++ = 0, L Gamma^+- = 0}
\end{split}\end{equation}
Thus,
${\Gamma^{++}}^{\HH}(z,w)$ and ${\Gamma^{+-}}^{\HH}(z,w)$
are fixed up to a non-essential constant.

The second equation
(\ref{Formula: Hadamard's formula})
can be reformulated now as
\begin{equation}\begin{split}
 &\Lc_{\delta} {\Gamma^{++}}^{\HH}(z,w)
 + \Lc_{\sigma} {\eta^+}^{\HH}(z) \Lc_{\sigma} {\eta^+}^{\HH}(w)
 = \beta_1^{\HH}(z) + \beta_1^{\HH}(w),\\
 &\Lc_{\delta} {\Gamma^{+-}}^{\HH}(z,\bar w)
 + \Lc_{\sigma} {\eta^+}^{\HH}(z) \overline{\Lc_{\sigma} {\eta^+}^{\HH}(w)}
 = \beta_1^{\HH}(z) + \beta_1^{\HH}(\bar w)
 \label{Formula: L Gamma^++ + L sigma^+ L sigma^+ = e + e}
\end{split}\end{equation}
for any analytic function $\beta_1^{\HH}(z)$ such that
$\overline{\beta_1^{\HH}(z)} = \beta_1^{\HH}(\bar z)$ analogously to 
\eqref{Formula: L Gamma^++ = beta + beta, L Gamma^+- = beta + beta}.
We can conclude now that the system
(\ref{Formula: L eta + 1/2 L^2 eta}--\ref{Formula: L_sigma Gamma = 0})
is equivalent to the system
(\ref{Formula: eta = eta^+ + eta^-}),
(\ref{Formula: delta/sigma j + mu [sigma,delta]/sigma + 1/2 L j = ibeta}),
(\ref{Formula: Gamma = Gamma^++ + bar Gamma^++ + Gamma^+- + bar Gamma^+-}),
(\ref{Formula: L Gamma^++ + L sigma^+ L sigma^+ = e + e}),
and
(\ref{Formula: L Gamma^++ = 0, L Gamma^+- = 0}).

Use now the fact
\begin{equation}
	{\Gamma^{++}}^{\HH}(z,w) = -\frac12 \log(z-w) + \text{analytic terms}
	\label{Formula: Gamma^++ = - 1 log z-w + hol}
\end{equation}
to obtain a singularity of
${j^+}^{\HH}$
about the origin in the half-plane chart.
Relation
(\ref{Formula: delta and sigma normalized form})
yields
\begin{equation}
 \frac{2}{z} \de_z \left( -\frac12 \log(z-w) \right) 
 + \frac{2}{w} \de_w \left( -\frac12 \log(z-w) \right)
 = \frac{1}{zw},
 \label{Formula: 2/z de_z - 12 log(z-w) + 2/z de_w - 12 log(z-w) = ...}
\end{equation}
hence,
\begin{equation}\begin{split}
	{j^+}^{\HH}(z) &= \frac{ -i  }{ z} + \text{holomorphic part}.
	\label{Formula: j^+ = mp/z + hol}
\end{split}\end{equation}
The choice of the sign of
${j^+}^{\HH}(z)$
is irrelevant. We made the choice above to be consistent with
\cite{Sheffield2010}.
The analytic terms in
(\ref{Formula: Gamma^++ = - 1 log z-w + hol})
can give a term with the sum of simple poles at $z$ and $w$ but in the form of
the product $1/(zw)$.

From
(\ref{Folmula: phi+ = int ...})
we conclude that the singular part of
${\eta^+}^{\HH}(z)$
is proportional to the logarithm of~$z$:
\begin{equation*}\begin{split}
	&{\eta^+}^{\HH}(z) = \frac{i}{\sqrt{\kappa}} \log z + \text{holomorphic part}.
\end{split}\end{equation*}
Thus, we have
\begin{equation}\begin{split}
	\eta^{\HH}(z) =& 
	\frac{-2}{\sqrt{\kappa}} \arg z + \text{non-singular harmonic part}.
	\label{Formula: eta = -2/k (arg z - pi/2) + hol}
\end{split}\end{equation}
We can chose the additive constant such that, in the half-plane chart, we have
\begin{equation}
  \eta^{\HH}(+0)=-\eta^{\HH}(-0) = \frac{\pi}{\sqrt{\kappa}}
  \label{Formula: eta(+0) = eta(-0) = ...}
\end{equation}
in the forward case. This provides the jump $2 \pi/\sqrt{\kappa}$ of the value
of $\eta$ at the boundary near the origin, which is exactly the same behaviour
of $\eta$ needed for the flow line construction in \cite{Miller2012} and
\cite{Sheffield2010}.
However, the form
\eqref{Formula: eta(+0) = eta(-0) = ...}
is not chart independent, and only the jump
$2 \pi/\sqrt{\kappa} = \eta^{\psi}(+0) - \eta^{\psi}(-0)$
does not change its value if the boundary of $\psi(\Dc)$ is not singular in the
neighbourhood of the source $\psi(a)$.

Substitute now
\eqref{Formula: j^+ = mp/z + hol}
in 
\eqref{Formula: delta/sigma j + mu [sigma,delta]/sigma + 1/2 L j = ibeta} 
in the half-plane chart, use 
\eqref{Formula: delta and sigma normalized form},
and consider the corresponding Laurent series.
We are interested in the coefficient near the first term
$\frac{1}{z^2}$:
\begin{equation*}
	\frac{2}{z} \frac{1}{-\sqrt{\kappa}} \frac{- i }{ z} +
	\mu \frac{-2}{z^2} + \frac12 (-\sqrt{\kappa})\frac{ i }{ z^2}
	+ o\left(\frac{1}{z^2}\right) = C^+.
\end{equation*}
We can conclude that
\begin{equation}\begin{split}
	&\mu = i\frac{4 - \kappa}{4\sqrt{\kappa}}.
	\label{Formula: mu = i pm (-4 pm k)/4/k}
\end{split}\end{equation}
Thus, the pre-pre-Schwarzians 
\eqref{Formula: tilde eta = eta - chi arg}
with $\chi$ given by
(\ref{Formula: chi = 2/k - k/2}) 
is only one that can be realized.
\end{proof}

\section{Coupling of GFF with  the Dirichlet boundary conditions}
\label{Section: The coupling in case of Dirichlet boundary condition of the covariance function}

In this section, we consider a special solution to the system 
(\ref{Formula: L eta + 1/2 L^2 eta}--\ref{Formula: L_sigma Gamma = 0})
with the help of Lemma 
\ref{Lemma: eta structure}.
We assume the Dirichlet boundary condition for $\Gamma$ considered in Example
\ref{Example: Dirichlet boundary conditions Gamma},
and find the general solution in this case. In other words, we systematically study which of $(\delta,\sigma)$-SLE can be coupled with GFF if $\Gamma=\Gamma_{D}$.

Let us formulate the following general theorem, and then consider each of the allowed cases of $(\delta,\sigma)$-SLE individually.

\begin{theorem}
Let a ($\delta$,$\sigma$)-SLE be coupled to the GFF with
$\Hc=\Hc_s$, $\Gamma=\Gamma_D$, and
let $\eta$ be the pre-pre-Schwarzian
\eqref{Formula: tilde eta = eta - chi arg} 
of order $\chi$. Then  the only special combinations of
$\delta$ and $\sigma$ for  $\kappa\neq 6$ and $\nu\neq 0$ summarized in Table \ref{Table: some simple cases},
and their arbitrary combinations when $\kappa=6$ and $\nu=0$ are possible.
\label{Theorem: ppS -> simple cases}
\end{theorem}

 \afterpage{
\begin{table}
\begin{tabular}{| c | p{3cm}  | c | c | c | c |}
 \hline
  N & Name & 
    $\delta^{\HH}(z)$ & $\sigma^{\HH}(z) $ & $\alpha$ & $\beta$ \\ \hline
  1 & Chordal with drift &
   $ \frac{2}{z} - \nu $ & 
   $ -\sqrt{\kappa} $ &
   $ -\frac{\nu}{2} $ &
   $ \frac{-\nu^{2}}{2\sqrt{\kappa}} $ \\ \hline
  2 & Chordal with fixed time change & 
   $ \frac{2}{z}+2\xi z$ & 
   $ -\sqrt{\kappa} $ & 
   $ 0 $ &
   $ \frac{\xi(8-\kappa)}{2\sqrt{\kappa}} $ \\ \hline
  3 & Dipolar with drift & 
   $ 2\left(\frac{1}{z} - z \right) - \nu (1-z^2) $ & 
   $ -\sqrt{\kappa}(1-z^2) $ & 
   $ -\frac{\nu}{2} $ & 
   $ \frac{4-\nu^{2}}{2\sqrt{\kappa}} $ \\ \hline
  4 & One right fixed boundary point & 
   $\begin{aligned} 
    &\frac{2}{z} +\kappa-6 + \\
    &+2(3-\kappa+\xi)z + \\
    &+(\kappa - 2 - 2 \xi)z^2 
   \end{aligned}$ & 
   $ -\sqrt{\kappa}(1-z^2) $ & 
   $ \frac12(\kappa-6)$ & 
   $ \frac{\xi(8-\kappa)}{2\sqrt{\kappa}} $ \\ \hline
  5 & One left fixed boundary point & 
   $\begin{aligned} 
    &\frac{2}{z} -(\kappa-6) + \\
    &+2(3-\kappa+\xi)z + \\
    &-(\kappa - 2 - 2 \xi)z^2 
   \end{aligned}$ & 
   $ -\sqrt{\kappa}(1-z^2) $ & 
   $ -\frac12(\kappa-6)$ & 
   $ \frac{\xi(8-\kappa)}{2\sqrt{\kappa}} $ \\ \hline
  6 & Radial with drift & 
   $ 2\left( \frac{1}{z} + z \right) - \nu( 1+z^2 ) $ & 
   $ -\sqrt{\kappa}(1+z^2) $ & 
   $ -\frac{\nu}{2} $ &
   $ \frac{4-\nu^{2}}{2\sqrt{\kappa}} $ \\ \hline
\end{tabular}
 \caption{($\delta$,$\sigma$)-SLE types that can be coupled with 
 CFT with the Dirichlet boundary conditions ($\Gamma=\Gamma_0$). 
 Each of the pairs of $\delta$ and $\sigma$ is given in the half-plane chart for simplicity.
 For the same purpose we use the normalization (\ref{Formula: delta and sigma normalized form}). 
 The details are in Sections 
 \ref{Section: Chrodal SLE with drift} 
 and 
 \ref{Section: Radial SLE with drift}.
 See also  the comments after  Theorem \ref{Theorem: ppS -> simple cases}.}
\label{Table: some simple cases}
\end{table}
}

Table \ref{Table: some simple cases} consists of 6 cases, each of which  is a one-parameter family of ($\delta$,$\sigma$)-SLEs parametrized by the drift $\nu\in\mathbb{R}$, and by a parameter $\xi\in\mathbb{R}$. These cases may overlap for vanishing values of 
$\nu$ or $\xi$. 

In other words, different combinations of $\delta$ and $\sigma$ can correspond essentially to the same process in $\Dc$ but written in  different coordinates. We give one example of such choices in each case  of $\delta$ and $\sigma$ . 

Some particular cases of CFTs studied here were considered earlier in the literature.
The chordal SLE without drift (case $1$ from the table with $\nu=0$) was considered in \cite{Kang2011},
the radial SLE without  drift (case $3$ from the table with $\nu=0$)  in \cite{Kang2012a}, and
 the dipolar SLE without  drift (case $4$ from the table with $\nu=0$) appeared in \cite{Kang}.
The case $2$ actually corresponds to the same measure defined by the chordal SLE but stopped at the  time
$t=1/4\xi$ 
(see Section \ref{Section: Chordal SLE with fixed time reparametrization}). 
The cases $5$ and $6$ are mirror images of each other. They are discussed in Section 
\ref{Section: SLE with one fixed boundary point}.

\begin{remark}
An alternative approach to the relation between CFT and SLE based on 
the highest weight representation of the Virasoro algebra was considered in
\cite{Bauer2006}
and
\cite{Friedrich2003}.
We remark that such a representation can not be constructed
for non-zero drift 
($\nu\neq 0$). 
\end{remark}

\medskip
\noindent
{\it Proof of Theorem \ref{Theorem: ppS -> simple cases}}.
Let us use Theorem \ref{Theorem: The coupling theorem} and assume the
Dirichlet boundary conditions for $\Gamma=\Gamma_D$. 
\begin{equation*}
 {\Gamma^{++}}^{\HH}(z,w) = -\frac12 \log(z-w),\quad
 {\Gamma^{+-}}^{\HH}(z,\bar w) = -\frac12\log (z- \bar w)
\end{equation*}
in Theorem
\ref{Lemma: eta structure}.
The condition
(\ref{Formula: L Gamma^++ = beta + beta, L Gamma^+- = beta + beta})
is satisfied for any complete vector field $\sigma$ and some $\sigma$-dependent
$\beta_2$ which is irrelevant.

In order to obtain $j^+$ we remark first that due to the M\"obious invariance
\eqref{Formula: Lv Gamma_D = 0}
we can ignore the polynomial part of $\delta^{\HH}(z)$
\begin{equation*}
 \Lc_{\delta} \Gamma^{\HH}(z,w) =
 \left(
  \frac{2}{z} \de_z + \frac{2}{\bar z} \de_{\bar z} +
  \frac{2}{w} \de_w + \frac{2}{\bar w} \de_{\bar w}
 \right)
 \Gamma^{\HH}(z,w).
\end{equation*}
Using
(\ref{Formula: L Gamma^++ + L sigma^+ L sigma^+ = e + e}),
(\ref{Formula: Gamma_D = Log...}),
and
(\ref{Formula: 2/z de_z - 12 log(z-w) + 2/z de_w - 12 log(z-w) = ...})
we obtain that
\begin{equation}\begin{split}
	&{j^+}^{\HH}(z) = \frac{-i}{z} + i\alpha,\quad \alpha\in\C, 
	\label{Formula: j+ = i/z + ia}
\end{split}\end{equation}
with
\begin{equation*}
 \beta_1(z) = \frac{\alpha}{z} - \frac{{\alpha}^2}{2}.
\end{equation*}

In order to satisfy all  conditions formulated  in Lemma
\ref{Lemma: eta structure}
we need to check
\eqref{Formula: delta/sigma j + mu [sigma,delta]/sigma + 1/2 L j = ibeta}.
Substituting
\eqref{Formula: j+ = i/z + ia} to
\eqref{Formula: delta/sigma j + mu [sigma,delta]/sigma + 1/2 L j = ibeta}
gives
\begin{equation}\begin{split}
 & \frac{\delta}{\sigma} j^+ + \mu \frac{[\sigma,\delta]}{\sigma}
 + \frac12 \Lc_\sigma^+ j^+ = i \beta ~ \Leftrightarrow \\
 & \delta j^+ + \mu\, [\sigma,\delta] + \frac12 \sigma \Lc_{\sigma}^+ j^+
  - i \beta \sigma  = 0
 ~ \Leftrightarrow \\
 & \delta^{\HH}(z) \left( \frac{-i }{z} + i \alpha \right)
 + \mu\, [\sigma,\delta]^{\HH}(z)
 + \frac12 \left(\sigma^{\HH}(z)\right)^2 \de \left( \frac{-i }{z} 
 + i \alpha \right) - i \beta \sigma^{\HH}(z) = 0.  \label{Formula: A phi = ib then ...}
\end{split}\end{equation}
In what follows, we will use the half-plane chart in the proof.
With the help of
(\ref{Formula: S - transform})
and
(\ref{Formula: R - transform})
we can assume without lost of generality that $\sigma^{\HH}$ is one of three
possible forms:
\begin{enumerate} [1.]
  \item $\sigma^{\HH}(z)=-\sqrt{\kappa}$,
  \item $\sigma^{\HH}(z)=-\sqrt{\kappa}(1-z^2)$,
  \item $\sigma^{\HH}(z)=-\sqrt{\kappa}(1+z^2)$.
\end{enumerate}
Let us consider these cases turn by turn.

\medskip
\noindent
{\bf 1. $\sigma(z) =-\sqrt{\kappa}$.} \\
 Inserting
 \eqref{Formula: delta and sigma normalized form} to
  relation
 \eqref{Formula: A phi = ib then ...}
 reduces to
 \begin{equation*}\begin{split}
  &\frac{-2 + \frac{\kappa}{2} -2 i \sqrt{\kappa}\mu }{z^2} +
  \frac{ 2 \alpha - \delta_{-1} }{z} +
  \left(\beta \sqrt{\kappa} + \alpha \delta_{-1} - \delta_{0}
   + i \sqrt{\kappa} \mu \delta_{0} \right)
   +\\+&
   z \left( \alpha \delta_0 - \delta_1 + 2 i \sqrt{\kappa} \mu \delta_1 \right)
   + z^2 \alpha \delta_1 \equiv 0 \quad \Leftrightarrow \\
  &(\ref{Formula: mu = i pm (-4 pm k)/4/k}),\quad
  2 \alpha - \delta_{-1}=0,\quad
  \beta \sqrt{\kappa} + \alpha \delta_{-1} - \delta_{0}
   + i \sqrt{\kappa} \mu \delta_{0} = 0,\quad \\
  &\alpha \delta_0 - \delta_1 + 2 i \sqrt{\kappa} \mu \delta_1 = 0,\quad
  \alpha \delta_1 =0.
\end{split}\end{equation*}
There are three possible cases:
\begin{enumerate}
 	\item 
  \begin{equation*}
  	\delta_{-1} = 2\alpha,\quad
  	\delta_0 = 0,\quad
  	\delta_{1} = 0,\quad
  	\kappa>0,\quad
  	\beta = \frac{-2 \alpha^2}{\sqrt{\kappa}}.
	\end{equation*}
	It is convenient to use the drift parameter $\nu$.
		Thus,
	\begin{equation*}
		\nu=-2\alpha,
	\end{equation*}
	which is related to the drift in the chordal equation.
	This case is presented in the first line of Table
	\ref{Table: some simple cases}.	  
	\item 
	\begin{equation*}
  	\delta_{-1} = 0,\quad
  	\delta_0 = -\frac{4 \beta \sqrt{\kappa}}{\kappa - 8},\quad
   	\delta_{1} = 0,\quad
   	\kappa>0,\quad
   	\alpha = 0.
	\end{equation*}
	This case is presented in the second line of Table ($\xi\in\mathbb{R}$)
	and discussed in details in Section
	\ref{Section: Chordal SLE with fixed time reparametrization}.
	\item 
	\begin{equation*}
		\delta_{-1}=0,\quad
		\delta_0=2\sqrt{5}\beta,\quad
		\delta_1\in\mathbb{R},\quad
		\kappa=6,\quad
		\alpha=0.
	\end{equation*}	
	This is a general case of $\delta$ with $\kappa=6$ and $\nu=0$.	
\end{enumerate}

\medskip
\noindent
{\bf 2. $\sigma^{\HH}(z)=-\sqrt{\kappa}(1-z^2)$.} \\
Relation (\ref{Formula: A phi = ib then ...}) reduces to
\begin{equation*}\begin{split}
  &\frac{-2 + \frac{\kappa}{2} + 2 i \mu \sqrt{\kappa} }{z^2} +
   \frac{ 2 \alpha - \delta_{-1} }{z}
   \left(\beta \sqrt{\kappa} - \kappa + 6 i \sqrt{\kappa} \mu + \alpha
   \delta_{-1} - \delta_0 + i \sqrt{\kappa} \mu \delta_0 \right)
   +\\+&
   z \left( 2 i \sqrt{\kappa} \mu \delta_{-1} + \alpha \delta_0 - \delta_1
    + 2 i \sqrt{\kappa} \mu \delta_1 \right)
   z^2 \left( - \beta \sqrt{\kappa} + \frac{\kappa}{2} + i \sqrt{\kappa} \mu
   \delta_0 + \alpha \delta_1 \right) = 0 \quad \Leftrightarrow \\
  &(\ref{Formula: mu = i pm (-4 pm k)/4/k}),\quad
  2 \alpha - \delta_{-1} = 0,\quad
   \beta \sqrt{\kappa} - \kappa + 6 i \sqrt{\kappa} \mu + \alpha \delta_{-1} -
   \delta_0 + i \sqrt{\kappa} \mu \delta_0 = 0,\\
  & 2 i \sqrt{\kappa} \mu \delta_{-1} + \alpha \delta_0 - \delta_1
    + 2 i \sqrt{\kappa} \mu \delta_1 = 0,\quad
  - \beta \sqrt{\kappa} + \frac{\kappa}{2} + i \sqrt{\kappa} \mu \delta_0
    + \alpha \delta_1 = 0.
\end{split}\end{equation*}
There are four solutions each of which is a two-parameter family.
The first one corresponds to the dipolar SLE with the drift $\nu$,
line 3 in  Table \ref{Table: some simple cases}.
The second and the third equations are `mirror images' of each other,
as it can be seen from the  lines 4 and 5 in the table. They are parametrized
by $\xi:=\frac{2\beta\sqrt{\kappa}}{8-\kappa}$ and discussed in details in
Section 
\ref{Section: SLE with one fixed boundary point}. 
The fourth case is given by putting
\begin{equation*}
	\delta_{-1}=0,\quad
	\delta_0=2(\sqrt{6}\beta-3),\quad
	\delta_1\in\mathbb{R},\quad
	\kappa=6,\quad
	\alpha=0. 
\end{equation*}
This is a general form of $\delta$ with $\kappa=6$ and $\nu=0$.

\medskip
\noindent
{\bf 3. $\sigma^{\HH}(z)=-\sqrt{\kappa}(1+z^2)$.} \\
Relation (\ref{Formula: A phi = ib then ...}) reduces to
\begin{equation*}\begin{split}
  &\frac{-2 + \frac{\kappa}{2} -2 i \sqrt{\kappa} \mu }{z^2} +
   \frac{ 2 \alpha - \delta_{-1} }{z}
   +\\+&
   \left( \beta \sqrt{\kappa} - \kappa + 6 i \sqrt{\kappa} \mu + \alpha
   \delta_{-1} - \delta_0 + i \sqrt{\kappa} \mu \delta_0 \right)
   +\\+&
   z \left( 2 i \kappa \mu \delta_{-1} + \alpha \delta_0 - \delta_1
    + 2 i \sqrt{\kappa} \mu \delta_1 \right)
   +\\+&
   z^2 \left( -\beta \sqrt{\kappa} + \frac{\kappa}{2} + i \sqrt{\kappa} \mu
   \delta_0 + \alpha \delta_1 \right) = 0 \quad \Leftrightarrow \\
  &(\ref{Formula: mu = i pm (-4 pm k)/4/k}),\quad
  2 \alpha - \delta_{-1} = 0,\quad
   \beta \sqrt{\kappa} - \kappa + 6 i \sqrt{\kappa} \mu + \alpha
   \delta_{-1} - \delta_0 + i \sqrt{\kappa} \mu \delta_0 = 0,\\
  &2 i \kappa \mu \delta_{-1} + \alpha \delta_0 - \delta_1
    + 2 i \sqrt{\kappa} \mu \delta_1 = 0,\quad
  -\beta \sqrt{\kappa} + \frac{\kappa}{2} + i \sqrt{\kappa} \mu
  \delta_0 + \alpha \delta_1 =0.
\end{split}\end{equation*}
The first solution is presented in the line 6 of Table
\ref{Table: some simple cases},
where it is again convenient to introduce the parameter $\nu$
related to the drift in the radial equation. The second solution is
\begin{equation*}
	\delta_{-1}=0,\quad
	\delta_0=2(\sqrt{6}\beta-3),\quad
	\delta_1\in\mathbb{R},\quad
	\kappa=6,\quad
	\alpha=0. 
\end{equation*}
This is a general form of $\delta$ with $\kappa=6$ and $\nu=0$.
\quad\qed

\subsection{Chrodal SLE with drift}
\label{Section: Chrodal SLE with drift}
 It is natural to study this case in the half-plane chart, where
 \begin{equation}
  \delta_c^{\HH}(z):=\frac{2}{z} - \nu,\quad 
  \sigma_c^{\HH}(z):=-\sqrt{\kappa},\quad
  \nu \in \mathbb{R}.
  \label{Formula: delta and sigma chordal}
 \end{equation}
The form of $\eta^+$ can be found from
 (\ref{Formula: j^+ = L eta^+}) 
 by substituting   
 (\ref{Formula: j+ = i/z + ia}) as
 \begin{equation*}
  -\kappa^{\frac12} \de_z {\eta^+}^{\HH}(z) + \mu \cdot 0 = 
  \frac{-i}{z} + i \alpha.
 \end{equation*}
 Then
 \begin{equation}
  {\eta^+}^{\HH}(z) = 
  \frac{i}{\sqrt{\kappa}} \log z - \frac{i \alpha z}{\sqrt{\kappa}} + C^+ ,
  \label{Formula: phi+ - chordal with drift}
 \end{equation}
 and taking into account that $\alpha=-\frac{\nu}{2}$ we obtain 
 \begin{equation}
  {\eta}^{\HH}(z) = 
  \frac{-2}{\sqrt{\kappa}} \arg z  
  - \frac{ \nu}{\sqrt{\kappa}} \Im z + C.
  \label{Formula: eta - chordal with drift in H}
 \end{equation}

Let us present here an explicit form of the evolution of
 the one-point function $S_1(z)=\eta(z)$
 \begin{equation*}\begin{split}
  {M}_t^{\HH}(z) =& ({G_t^{-1}}_* \eta)^{\HH}(z) = 
  \frac{-2}{\sqrt{\kappa}} \arg G_t^{\HH}(z)  
  - \frac{ \nu}{\sqrt{\kappa}} \Im G_t^{\HH}(z)
  +\frac{\kappa-4}{2\sqrt{\kappa}} \arg {G_t^{\HH}}'(z) + C
 \end{split}\end{equation*}
 This expression with $\nu=0$ coincides (up to a constant) 
 with the analogous one from 
 \cite[Section 8.5]{Kang2011}.

Now we need to work with a concrete form of the space $\Hc_s$  discussed in Section 
\ref{Section: Test function}. It is convenient to define it in the half-plane chart. In other charts it can be obtained with  the rule
(\ref{Formula: f^tilde psi = tau^2 f^psi(tau)}).
We  choose the subspace ${C^{\infty}}^{\HH}_0$  of $C^{\infty}$-smooth  functions 
with compact support																											in the half-plane chart. The function $\phi$ defining metric can be, for example, zero in the half-plane chart, $\phi^{\HH}(z)\equiv 0$.
This choice guarantees that the integrals in
(\ref{Formula: dI int G^-1 eta = int dI G^-1 eta})
and
(\ref{Formula: dI int G^-1 Gamma = int dI G^-1 Gamma})
are well-defined with $\eta$ as above and $\Gamma=\Gamma_D$.

\subsection{Dipolar SLE with drift}
\label{Section: Dipolar SLE with drift} 
 
 The dipolar SLE equation is usually defined in the strip chart, see (\ref{Formula: tau_H S}), as
 \begin{equation}
  \dI G_t^{\SSS}(z) 
  = \cth \frac{G_t^{\SSS}(z)}{2} dt
  -\sqrt{\kappa} d^{\rm{It\^{o} }} B_t
  -\frac{\nu}{2} dt,  
 \end{equation}
 where we add the drift term $\frac{\nu}{2} dt $
 in the It\^{o} differentials with the same form in terms of Stratonovich 
 \begin{equation}
  \dS G_t^{\SSS}(z) 
  = \cth \frac{G_t^{\SSS}(z)}{2} dt
  -\sqrt{\kappa} d^{\rm{S}} B_t
  -\frac{\nu}{2} dt
 \end{equation}
 because $\sigma^{\SSS}(z)$ is constant. The vector fields $\delta$ and $\sigma$ in the strip chart, 
see (\ref{Formula: tau_H S}),
are
\begin{equation*}\begin{split}
 \delta^{\SSS}(z) 
 = 4\cth \frac{z}{2} - \frac{\nu}{2},\quad  
 \sigma^{\SSS}(z) 
 = - \sqrt{\kappa},\quad 
 \nu \in \mathbb{R},
\end{split}\end{equation*}
 
The from of $\delta$ and $\sigma$ in the half-plane can be obtained by
(\ref{Formula: tilde v = 1/dtau v(tau)}) as
\begin{equation*}\begin{split}
 &\delta^{\HH}(z) 
 = \frac{1}{\tau_{\SSS,\HH}(z)} \delta^{\SSS}(\tau_{\SSS,\HH}(z))
 = \frac12 \left( \frac{1}{z}-z \right) - \frac{\nu}{2} (1-z^2),\\
 &\sigma^{\HH}(z) 
 = \frac{1}{\tau_{\SSS,\HH}(z)} \sigma^{\SSS}(\tau_{\SSS,\HH}(z))
 = - \frac{\sqrt{\kappa}}{2} (1-z^2).
\end{split}\end{equation*}
It is more convenient  to use for our purposes the transform 
(\ref{Formula: S - transform}) 
with $c=2$, and to define
 \begin{equation}
  \delta_d^{\HH}(z):= 2 \left( \frac{1}{z} - z \right) - \nu(1-z^2),\quad 
  \sigma_d^{\HH}(z):=-\sqrt{\kappa}(1-z^2),\quad
  \nu \in \mathbb{R}
  \label{Formula: delta and sigma dipolar in H}
 \end{equation}
 that possess normalization (\ref{Formula: delta and sigma normalized form}) 
  used in Table \ref{Table: some simple cases}.
 
 Let us first find $\eta^+$, $\eta$ and ${M_1}_t$ in the half-plane chart. 
 The same way as in the previous subsection we calculate 
 \begin{equation*}
  -\sqrt{\kappa}(1-z^2) \de_z {\eta^+}^{\HH}(z) 
  + \mu \left( -\sqrt{\kappa} (1-z^2) \right)' = 
  \frac{-i}{z} + i \alpha.
 \end{equation*}
 Taking into account 
 (\ref{Formula: mu = i pm (-4 pm k)/4/k}) 
 and $\alpha=-\nu/2$ we obtain
 \begin{equation*}
  {\eta^+}^{\HH}(z) =
  \frac{i}{\sqrt{\kappa}} \log z  
  + \frac{i(\kappa-6)}{4 \sqrt{\kappa}} \log (1-z^2)  
  + \frac{i\nu}{2\sqrt{\kappa}} \rm{arcth} z + C^+,
 \end{equation*}
 \begin{equation*}
  {\eta}^{\HH}(z) = 
  \frac{-2}{\sqrt{\kappa}} \arg z  
  - \frac{(\kappa-6)}{2 \sqrt{\kappa}} \arg (1-z^2) 
  - \frac{\nu}{\sqrt{\kappa}} \Im \rm{arcth} z + C,
 \end{equation*}
 \begin{equation*}\begin{split}
  &{M_c}_t^{\HH}(z) = ({G_t^{-1}}_* \eta)^{\HH}(z) 
  =\\=&
  \frac{-2}{\sqrt{\kappa}} \arg G_t^{\HH}(z)  
  - \frac{(\kappa-6)}{2 \sqrt{\kappa}} \arg (1-G_t^{\HH}(z)^2)
  -\\-& 
  \frac{\nu}{\sqrt{\kappa}} \Im \rm{arcth}\, G_t^{\HH}(z)
  +\frac{\kappa-4}{2\sqrt{\kappa}} \arg {G_t^{\HH}}'(z) + C
 \end{split}\end{equation*}
 
The corresponding relations in the strip chart are
\begin{equation*}\begin{split}
 \delta_d^{\SSS}(z) 
 = 4 \cth \frac{z}{2} - 2 \nu,\quad  
 \sigma_d^{\SSS}(z) 
 = - 2\sqrt{\kappa},\quad 
 \nu \in \mathbb{R},
\end{split}\end{equation*}
obtained with the help of
(\ref{Formula: tilde phi = phi - chi arg}). Then
\begin{equation*}
 \eta^{\SSS}(z) 
 = \eta^{\HH}(\tau_{\HH,\SSS}(z)) 
 + \frac{\kappa-4}{2\sqrt{\kappa}} \Im {\tau_{\HH,\SSS}}'(z),
\end{equation*}
where we used (\ref{Formula: chi = 2/k - k/2}) and the expression for 
$\tau_{\HH,\SSS}(z)=\tau_{\SSS,\HH}^{-1}(z)=\psi_{\HH}\circ\psi_{\SSS}^{-1}(z)$ 
that defines the strip chart 
(\ref{Formula: tau_H S}). 
Alternatively $\eta^{\SSS}(z)$ can be found as the solution to 
(\ref{Formula: j^+ = L eta^+}) 
in the strip chart
\begin{equation*}
 \eta^{\SSS}(z) 
 = \frac{-2 }{\sqrt{\kappa}} \arg \rm{sh} \frac{z}{2}
 - \frac{\nu}{2\sqrt{\kappa}} \Im z + C, 
\end{equation*}  
\begin{equation*}\begin{split}
 &{M_1}_t^{\SSS}(z) = ({G_t^{-1}}_* \eta)^{\SSS}(z) 
 =\\=&
 \frac{-2}{\sqrt{\kappa}} \arg \rm{sh} \frac{G_t^{\SSS}(z)}{2} 
 - \frac{\nu}{2\sqrt{\kappa}} \Im G_t^{\SSS}(z)
 +\frac{\kappa-4}{2\sqrt{\kappa}} \arg {G_t^{\SSS}}'(z) + C .
\end{split}\end{equation*}  
The expression for $\Gamma_D$ in the strip chart becomes
\begin{equation*}
 \Gamma_D^{\SSS}(z,w) 
 = \Gamma_D^{\HH}\left( \tau_{\HH,\SSS}(z),\tau_{\HH,\SSS}(w) \right) 
 = -\frac12 \log \frac{ \rm{sh}(\frac{z-w}{2}) \rm{sh}(\frac{\bar z-\bar w}{2})}
  { \rm{sh}(\frac{\bar z-w}{2}) \rm{sh}(\frac{z-\bar w}{2})}.
\end{equation*}

We remark that $\eta$ can be presented in a chart-independent form as a function of $\delta_d$ and $\sigma_d$ using
(\ref{Formula: eta = -mu log v - mu log v})
as
\begin{equation*}
 \eta = - \frac{2}{\sqrt{\kappa}} \arg \frac{\sigma_d^{\frac{\kappa}{4}}}
  {\sqrt{\sigma_d^2-\frac{\kappa}{4} \left(\delta_d-\frac{\nu}{\sqrt{\kappa}}\sigma_d\right)^2 }} + 
 \frac{\nu}{\sqrt{\kappa}} \Im \,\rm{arcth}\,
 \left(
  \frac{2}{\sqrt{\kappa}}\frac{\sigma_d}{\delta_d-\frac{\nu}{\sqrt{\kappa}}\sigma_d}
 \right) + C.
\end{equation*}
The expression under the square root vanishes only at the same points as $\sigma_d$.
As before, the choice of the branch is irrelevant because $\sigma$ can vanish at infinity only at the boundary, and $\eta$ is defined up to a constant $C$.

Now we work with the concrete form of the space $\Hc_s$  discussed in Section 
\ref{Section: Test function}. It is convenient to define it in the strip chart. 
We  choose the subspace ${C^{\infty}}^{\SSS}_0$  of $C^{\infty}$-smooth  functions 
with compact support
in the strip chart. 
The function $\phi$ defining the metric can be, for example, zero in the strip chart, $\phi^{\SSS}(z)\equiv 0$, which
guarantees that the integrals in
(\ref{Formula: dI int G^-1 eta = int dI G^-1 eta})
and
(\ref{Formula: dI int G^-1 Gamma = int dI G^-1 Gamma})
are well-defined with $\eta$ as above and $\Gamma=\Gamma_D$.

\subsection{Radial SLE with drift}
\label{Section: Radial SLE with drift} 

The radial SLE equation 
(\ref{Formula: G radial SLE in D}) 
is usually formulated in the unit disk chart. It can be defined with the vector fields 
(\ref{Formula: delta and sigma for radial SLE in D}), which admit the form (\ref{Formula: delta and sigma for radial SLE in H})
 in the half-plane chart.
By the same reasons as for the dipolar SLE, we can change  normalization and define
\begin{equation}
 \delta_r^{\HH}(z):= 2 \left( \frac{1}{z} + z \right) - \nu(1+z^2),\quad 
 \sigma_r^{\HH}(z):=-\sqrt{\kappa}(1+z^2),\quad
 \nu \in \mathbb{R}
 \label{Formula: delta nad sigam radial in H}
\end{equation}
that coincides with the expressions in Table \ref{Table: some simple cases}.

Let us give  here the expressions for
$\delta$, $\sigma$, $\Gamma_D$, $\eta$ and ${M_1}_t$ in three different charts: half-plane, 
logarithmic (see below for the details), and the unit disk using the same method as before. 
The calculations are  similar to the 
dipolar case. In fact, it is enough to  change some signs and 
replace the hyperbolic functions by trigonometric. 
In contrast to the dipolar case, $\eta$ is multiply defined. We discuss this difficulty at the end of this subsection. As for now, we just remark
that from the heuristic point of view this is not an essential problem. In any chart $\eta^{\psi}(z)$ just  changes its value only up to an irrelevant constant after the harmonic continuation about the fixed point of the radial equation. 

In the half-plane chart, we have
 
\begin{equation*}
 -\sqrt{\kappa}(1+z^2) \de_z {\eta^+}^{\HH}(z) 
 + \mu \left( -\sqrt{\kappa} (1+z^2) \right)' = 
 \frac{-i}{z} + i \alpha,
\end{equation*}
\begin{equation}
 {\eta}^{\HH}(z) = 
 \frac{-2}{\sqrt{\kappa}} \arg z  
 - \frac{(\kappa-6)}{2 \sqrt{\kappa}} \arg (1+z^2)
 - \frac{\nu}{\sqrt{\kappa}} \Im \rm{arctg}\, z + C,
 \label{Formula: phi - radial with drift in H}
\end{equation}
\begin{equation*}\begin{split}
 &{M_1}_t^{\HH}(z) = ({G_t^{-1}}_* \eta)^{\HH}(z) 
 =\\=&
 \frac{-2}{\sqrt{\kappa}} \arg G_t^{\HH}(z) 
 - \frac{(\kappa-6)}{2 \sqrt{\kappa}} \arg (1 + G_t^{\HH}(z)^2) 
 -\\-& 
 \frac{\nu}{\sqrt{\kappa}} \Im \rm{arctg}\, z
 +\frac{\kappa-4}{2\sqrt{\kappa}} \arg {G_t^{\HH}}'(z) + C 
\end{split}\end{equation*}
 analogously to 
\eqref{Formula: phi+ - chordal with drift}.

The unit disk chart is defined in 
(\ref{Formula: tau_H D}), and
\begin{equation*}
 \eta^{\D}(z) 
 = \frac{-2}{\sqrt{\kappa}} \arg (1-z)
 - \frac{\kappa-6}{2\sqrt{\kappa}} \arg z
 + \frac{\nu}{2\sqrt{\kappa}} \log |z| + C,
\end{equation*}
\begin{equation*}\begin{split}
 &{M_1}_t^{\D}(z) = ({G_t^{-1}}_* \eta)^{\D}(z)
 =\\=& 
 \frac{-2}{\sqrt{\kappa}} \arg (1-G_t^{\D}(z))
 - \frac{\kappa-6}{2\sqrt{\kappa}} \arg G_t^{\D}(z)
 + \frac{\nu}{2\sqrt{\kappa}} \log |G_t^{\D}(z)|
 + \frac{\kappa-4}{2\sqrt{\kappa}} \arg {G_t^{\D}}'(z) + C, 
\end{split}\end{equation*}
\begin{equation*}
 \Gamma_D^{\D}(z,w) 
 = \Gamma_D^{\D}\left( \tau_{\HH,\D}(z),\tau_{\HH,\D}(w) \right) 
 = -\frac12 \log \frac{ (z-w)(\bar z - \bar w) }
  { (\bar z-w)(z - \bar w) }.  	
\end{equation*}

The third chart is called \emph{logarithmic}, and it is defined by the transition map
\begin{equation}
 \tau_{\D,\LL}(z):=e^{iz}:\HH\map\D,\quad
 \tau_{\LL,\D}(z)=\tau_{\D,\LL}^{-1}(z) =-i\log \, z.
 \label{Formula: tau_D LL = ...}
\end{equation}
Therefore,
\begin{equation*}
 \tau_{\HH,\LL}(z) = \tau_{\HH,\D} \circ \tau_{\D,\LL}(z)= \rm{tg}\frac{z}{2}:\HH\map\HH,\quad
 \tau_{\LL,\HH}(z) = \tau_{\HH,\LL}^{-1}(z) = 2 \,\rm{arctg}\, z,
\end{equation*}
and $\psi^{\LL}$ is not a global chart map as we used before because there is a point (the origin in the unit-disc chart) which is mapped to  infinity. 
Besides, the function $\log$ is multivalued and the upper half-plane contains infinite number of identical copies of the  radial SLE slit 
($\tau_{\LL,\HH}(z+2\pi)=\tau_{\LL,\HH}(z)$). The advantage of this chart is that the automorphisms
$H_t[\sigma_r]^{\LL}$ induced by $\sigma_r$ 
(see \ref{Formula: d G = sigma G ds in a chart})
are  horizontal translations because $\sigma_r^{\LL}(z)$ is a real constant (see below). The corresponding relations for the radial SLE in the logarithmic chart can be easily obtained from the  dipolar SLE in the strip chart just by replacing the hyperbolic functions by their trigonometric analogs as
\begin{equation*}\begin{split}
  \delta_r^{\LL}(z) 
  = 4\,\rm{tg}\frac{2}{z} - 2 \nu,\quad  
  \sigma_r^{\LL}(z) 
  = - 2\sqrt{\kappa},\quad 
  \nu \in \mathbb{R}.
\end{split}\end{equation*}
\begin{equation*}
 \eta^{\LL}(z) 
 = \frac{-2 }{\sqrt{\kappa}} \arg \rm{sin} \frac{z}{2} 
 - \frac{\nu}{2\sqrt{\kappa}} \Im z + C, 
\end{equation*}  
\begin{equation*}\begin{split}
 &{M_1}_t^{\LL}(z) = ({G_t^{-1}}_* \eta)^{\LL}(z) 
 =\\=&
 \frac{-2}{\sqrt{\kappa}} \arg \sin \frac{G_t^{\LL}(z)}{2} 
 - \frac{\nu}{2\sqrt{\kappa}} \Im G_t^{\LL}(z)
 +\frac{\kappa-4}{2\sqrt{\kappa}} \arg {G_t^{\LL}}'(z) + C .
\end{split}\end{equation*}  
\begin{equation*}
 \Gamma_D^{\LL}(z,w) 
 = -\frac12 \log \frac{ \sin(\frac{z-w}{2}) \sin(\frac{\bar z-\bar w}{2})}
  { \sin(\frac{\bar z-w}{2}) \sin(\frac{z-\bar w}{2})},  	
\end{equation*}
This relations above coincide up to a constant with the analogous ones established in \cite{Kang2012a, Kang_preparation}.

We remark again that $\eta$ can be represented in a chart-independent form as a function of $\delta_r$ and $\sigma_r$ with the help of 
(\ref{Formula: eta = -mu log v - mu log v})
by the relation
\begin{equation*}
 \eta = - \frac{2}{\sqrt{\kappa}} \arg \frac{\sigma_r^{\frac{\kappa}{4}}}
  {\sqrt{\sigma_r^2+\frac{\kappa}{4} \left(\delta_r-\frac{\nu}{\sqrt{\kappa}}\sigma_r\right)^2 }} + 
 \frac{\nu}{\sqrt{\kappa}} \Im \,\rm{arcth}\,
 \left(
  \frac{2}{\sqrt{\kappa}}\frac{\sigma_r}{\delta_r-\frac{\nu}{\sqrt{\kappa}}\sigma_r}
 \right) + C.
\end{equation*}

In order to  define GFF for radial SLE carefully we need to generalize slightly the above approach. 
Let $b\in\Dc$ be a zero point  of $\delta_r$ and $\sigma_r$ simultaneously inside $\Dc$: 
$\delta_r^{\psi}(b)=\sigma_r^{\psi}(b)=0$ (for any $\psi$). We have 
$\psi^{\D}(b)=0$  in the unit disk chart,  $\psi^{\HH}(b)=i$ in the half-plane chart, and  $\psi^{\LL}(b)=\infty$ in the logarithmic chart.

Let $\hat \Dc_b$ be the universal cover of $\Dc\setminus\{b\}$. Then the logarithmic chart map 
$\psi^{\LL}:\hat \Dc \map \HH$ defines a global chart map of $\hat \Dc_b$.
The space $\Hc_s[\hat \Dc_b]$ in the logarithmic chart is defined as in Section
\ref{Section: Test function}
with 
$\phi^{\LL}\equiv 0$,
and we require in addition, that the support is bounded. The last condition guaranties the finiteness of functionals such as $\eta[f]$, $\Gamma_D[f_1,f_2]$, and the compatibility condition of $\Hc_s$, $\Gamma_D$, $\eta$ (as above), $\delta_r$ and $\sigma_r$ on $\hat \Dc_b$. 
The map 
$G_t:\Dc\setminus K_t \map \Dc$ 
is lifted to
$\hat G_t:\hat\Dc_b\setminus \hat K_t \map \hat\Dc_b$, 
where $\hat K_t\subset\hat\Dc_b$ 
is the corresponding union of the countable number of the copies of $K_t$.

Consideration of $\hat \Dc_b$ instead of $\Dc$
is possible thanks to a special property of radial L\"owner equation to have a fixed point $b\in\Dc$.
The  branch point $b$ is in fact eliminated from the domain of definition and 
the pre-pre-Schwarzian $\eta$ is well-defined on $\hat\Dc_b$.

\subsection{General remarks}

Here we are aimed at explaining  why all  three cases of $\eta$ above 
has the same form for $\kappa=6$ and $\nu=0$. Besides, we explain the relations between the chordal case and other cases considered in the next two subsections.

Let
$G_t$
be a
$(\delta,\sigma)$-SLE 
driven by 
$B_t$, and let
and 
$\tilde G_{\tilde t}$
be a
$(\tilde \delta,\tilde \sigma)$-SLE 
driven by
$\tilde B_{\tilde t}$ with the same
parameter $\kappa$.
Then there exists a stopping time $\tilde \tau>0$,
a family of random M\"obius automorphisms 
$M_{\tilde t}\, \colon\,  \Dc\map\Dc,~\tilde t\in [0,\tilde \tau)$,
and a random time reparametrization 
$\lambda:[0,\tilde \tau)\map[0,\tau)$
($\tau:=\lambda(\tilde \tau)$),
such that 
\begin{equation*}
 \tilde G_{\tilde t} = M_{\tilde t} \circ G_{\lambda_{\tilde t}},\quad \tilde t\in[0,\tilde \tau)
\end{equation*}
and
\begin{equation*}
 d\tilde B_{\tilde t} = 
 a_{\tilde t} d\tilde t + 
 \left(\dot \lambda(\tilde t)\right)^{-\frac12} dB_{\lambda(\tilde t)},\quad
 \tilde t \in [0,\tilde \tau)
\end{equation*}
for some continuous $a_{\tilde t}$.
In particular, this means that the laws of $\mathcal{K}_t$ and $\mathcal{\tilde K}_t$ induced by the 
$(\delta,\sigma)$-SLE 
and 
$(\tilde \delta,\tilde \sigma)$-SLE
correspondingly are absolutely continuous with respect to each other until some stopping time.
We proved this fact in \cite{Ivanov2014}. 
However, it is possible to show a bit more:
if $\nu=0$ for both $(\delta,\sigma)$-SLE and $(\tilde \delta,\tilde \sigma)$-SLE, then
the coefficient $a_{\tilde t}$
is proportional to $\kappa-6$.
Here the drift parameter $\nu$ is defined by
\begin{equation*}
 \nu :=\delta_{-1} + 3 \sigma_0,
\end{equation*}
see 
(\ref{Formula: allowed delta and sigma}).
This definition agrees with 
(\ref{Formula: delta and sigma chordal}),
(\ref{Formula: delta and sigma dipolar in H})
and
(\ref{Formula: delta nad sigam radial in H})
and is invariant with respect to
(\ref{Formula: M - transfrom}).
Since
$\dot \lambda_{\tilde t}^{\frac12} B_{\lambda(\tilde t)}$
agrees in law with 
$B_{\tilde t}$,
the random laws of
$\mathcal{K}_t$ and $\mathcal{\tilde K}_t$ are 
identical, not just absolutely continuous as above, at least until some stopping time. 

It can be observed that $\eta$ for 
the chordal
(\ref{Formula: eta - chordal with drift in H}), 
dipolar 
(\ref{Formula: eta - chordal with drift in H}),
and radial
(\ref{Formula: phi - radial with drift in H})
cases are identical for $\kappa=6$ and $\nu=0$.
This is a consequence of the above fact.
Special cases of chordal and radial SLEs were considered in 
\cite{Schramm2006}.

Besides, there are two special situations when $a_{\tilde t}$ is identically zero for all values of $\kappa>0$, not only for $\kappa=6$ as above. 
In order to study them, let us consider the chordal SLE $G_t$,
see (\ref{Formula: chordal SLE in H in Strat}),
and a differentiable time reparametrization $\lambda$, which possesses  property
(\ref{Formula: d lambda = a dt + b dB}).
Set
\begin{equation*}
 \tilde G_{\tilde t} := s_{c_{\tilde t}} \circ G_{\lambda_{\tilde t}},  
\end{equation*}
where $s_c:\Dc\map\Dc$ is the scaling flow 
($s_c^{\HH}(z)=e^{-c }z,~c\in\mathbb{R}$). 
In the half-plane chart we have
\begin{equation}
 \tilde G_{\tilde t}^{\HH}(z) = e^{-c_{\tilde t}} G_{\lambda_{\tilde t}}^{\HH}(z).
 \label{Formula: tilde G = e^c G}
\end{equation}
The Stratonivich differential of $\tilde G^{\HH}_{\tilde t}(z)$ is 
\begin{equation*}\begin{split}
 \dS \tilde G^{\HH}_{\tilde t}(z) =& 
 (\dS e^{-c_{\tilde t}}) G^{\HH}_{\lambda_{\tilde t} }(z) +
 e^{-c_{\tilde t}} \dS G^{\HH}_{\lambda_{\tilde t} }(z)  
 =\\=&
 (\dS e^{-c_{\tilde t}}) G^{\HH}_{\lambda_{\tilde t} }(z) +
 e^{-c_{\tilde t}} \dot \lambda_{\tilde t}
 \left( 
  \frac{2}{G^{\HH}_{\lambda_{\tilde t} }(z)}d\tilde t - \sqrt{\kappa} \dS B_{\lambda_{\tilde t}}
 \right)
\end{split}\end{equation*}
Due to 
(\ref{Folrmula: dB = lambda dB - 1/4 b/lambda dt}),
we have to assume that 
\begin{equation*}
 e^{-c_{\tilde t}} \dot \lambda_{\tilde t} \equiv \dot \lambda_{\tilde t}^{\frac12},\quad
\end{equation*}
in order to have an autonomous equation.
So
\begin{equation*}
 e^{-c_{\tilde t}} = \dot \lambda^{-\frac12}_{\tilde t},
\end{equation*}
and, consequently,
\begin{equation*}
 \dS e^{-c_{\tilde t}} =
 - \frac12 e^{-3 c_{\tilde t}} a_{\tilde t} d \tilde t
 - \frac12 e^{-3 c_{\tilde t}} b_{\tilde t} \dS \tilde B_{\tilde t},
 \end{equation*}
where we used 
(\ref{Formula: d lambda = a dt + b dB}).
Eventually, we conclude that
\begin{equation}\begin{split}
 \dS \tilde G^{\HH}_{\tilde t}(z) =&
 \left( 
  - \frac12 e^{-3 c_{\tilde t}} a_{\tilde t} d \tilde t
  - \frac12 e^{-3 c_{\tilde t}} b_{\tilde t} \dS \tilde B_{\tilde t}
 \right) 
  e^{ c_{\tilde t}} \tilde G^{\HH}_{\tilde t}(z) +
 \frac{2}{\tilde G^{\HH}_{\tilde t }(z)}d\tilde t 
 - \sqrt{\kappa} \dS \tilde B_{\tilde t} 
 + \frac14 \sqrt{\kappa} e^{-2c_{\tilde t}} b_{\tilde t} d \tilde t.
 \label{Formula: 3}   
\end{split}\end{equation}

In order to have time independent coefficients we  assume that $a_{\tilde t}$ and $b_{\tilde t}$ are proportional to $e^{2 c_{\tilde t}}$. Hence, define $\xi\in \mathbb{R}$ by
\begin{equation*}
 a_{\tilde t} = - 4 \xi e^{2c_{\tilde t}}.
\end{equation*}
Without lost of generality, we can assume that $b_{\tilde t}$ is of the following three possible forms
\begin{enumerate}[1.]
\item $b_{\tilde t}=0$,
\item $b_{\tilde t}=4\sqrt{\kappa} e^{2c_{\tilde t}} $,
\item $b_{\tilde t}=-4\sqrt{\kappa} e^{2c_{\tilde t}} $,
\end{enumerate}
because all other choices can be reduced to these three by 
(\ref{Formula: S - transform}).
The first case is considered in Section 
\ref{Section: Chordal SLE with fixed time reparametrization}. 
Other two cases are discussed in Section
\ref{Section: SLE with one fixed boundary point}.

\subsection{Chordal SLE with fixed time reparametrization.}
\label{Section: Chordal SLE with fixed time reparametrization}

Let $\xi \in(-\infty,+\infty) \setminus \{0\}$, and let $G_t$ be a chordal stochastic flow,
i.e., the chordal SLE (\ref{Formula: delta and sigma chordal}). 
Define 
\begin{equation*}
 \tilde G_{\tilde t}^{\HH}(z) = e^{2\xi \tilde t} G_{\lambda(\tilde t)}^{\HH}(z)  
\end{equation*}
in the half-plane chart and
assume that
\begin{equation*}\begin{split}
 &\lambda(\tilde t) := \frac{1 - e^{-4\xi \tilde t}}{4\xi};\\
 &\lambda:~ [0,+\infty)\map[0,(4\xi)^{-1}), \quad \xi > 0;\\
 &\lambda:~ [0,+\infty)\map[0,+\infty), \quad \xi < 0.
\end{split}\end{equation*}
This choice of $\lambda$ corresponds to 
$c_{\tilde t}=-2\xi \tilde t$, 
in the previous subsection. We remark, that the time reparametrization here is not random.

The flow $\tilde G_{\tilde t}$ satisfies the autonomous equation 
(\ref{Formula: Slit hol stoch flow Strat}) with 
\begin{equation}
 \delta^{\HH}(z) = \frac{2}{z} + 2\xi z,\quad 
 \sigma^{\HH}(z) = -\sqrt{\kappa},
 \label{Formula: delta and sigma for fixed time change case}
\end{equation}
which are the vector fields from the second string of Table \ref{Table: some simple cases}
and a special case of
(\ref{Formula: 3})
with $a_{\tilde t} = - 4 \xi e^{2c_{\tilde t}}$ and $b_{\tilde t}=0$.

There is a common zero of $\delta$ and $\sigma$ at infinity in the half-plane chart, 
so  infinity is a stable point
$\tilde G^{\HH}_{\tilde t}: \infty \map \infty$. But in contrast to the chordal case 
the coefficient at $z^{-1}$ in the Laurent series is not $1$ but 
$e^{2 \xi \tilde t}$.   
The vector field $\delta$ is of radial type if $\xi> 0$, and of dipolar type if $\xi< 0$.
It is remarkable that if $\xi<0$, then the equation induces exactly the same measure as 
the chordal stochastic flow but with a different time parametrization. 
If $\xi>0$ the measures also coincide when the chordal stochastic flow is stopped at the time 
$t=(4\xi)^{-1}$.   

By the reasons described above it is natural to expect that the GFF coupled with such kind of 
($\delta,\sigma$)-SLE is the same as in the chordal case, 
because it is supposed to induce the same random law of the flow lines. 
Indeed, $\sigma$ from 
(\ref{Formula: delta and sigma for fixed time change case})
coincides with that from the chordal case, hence, $\eta$ defined by 
(\ref{Formula: j = L eta}),
with 
$\alpha=0$ (see the table) also coincides with 
(\ref{Formula: eta - chordal with drift in H}) 
with $\nu=0$. Thus, the martingales are the same as in the chordal case.

\subsection{SLE with one fixed boundary point.}
\label{Section: SLE with one fixed boundary point} 

Let the vector fields $\delta$ and $\sigma$ be defined by the 5$^{\text{th}}$ and the 6$^{\text{th}}$ strings of Table 
\ref{Table: some simple cases}. 
There are two `mirror' cases. The ($\delta,\sigma$)-SLE  denoted here by 
$G_{t}$ 
is characterized by the stable point at $z=1$ (the 5$^{\text{th}}$ case) or
$z=-1$ (the 6$^{\text{th}}$ case)  in the half-plane chart. 
We will consider only the first (the 5$^{\text{th}}$ string) case, the second (the 6$^{\text{th}}$ string) is similar.

We will show below that this $(\delta,\sigma)$-SLE 
coincides with the chordal SLE up to a random time reparametrization for all values of $\kappa>0$. Let us apply a M\"obius transform 
$r_c\colon \Dc\map\Dc$ defined in 
\eqref{Formula: R - transform}
with $c=-1$
\begin{equation*}
 r_{-1}^{\HH}(z)=\frac{z}{1+z}.
\end{equation*}
In the half-plane chart, it maps the stable point $z=1$ to  infinity keeping the origin  and the normalization 
(\ref{Formula: delta and sigma normalized form}) unchanged.
It results in
\begin{equation*}\begin{split}
	&\tilde G_t := r_{-1} \circ G_t \circ r^{-1}_{-1}, \\
	&r_{-1}{}_* \delta^{\HH}(z) = \tilde \delta^{\HH}(z) = \frac{2}{z} + \kappa + 2\xi z,\\
	&r_{-1}{}_* \sigma^{\HH}(z) = \tilde \sigma^{\HH}(z) = -\sqrt{\kappa}(1+2z),
\end{split}\end{equation*}
and the equation for $\tilde G_t$ becomes
\begin{equation}\begin{split}
 &d^S \tilde G_t^{\HH}(z) = 
  \left(
   \frac{2}{\tilde G_t^{\HH}(z)} + \kappa + 2 \xi \tilde G_t^{\HH}(z)
  \right)dt -
  \sqrt{\kappa} \left(1 + 2 \tilde G_t^{\HH}(z) \right) d^S B_t, 
   \label{Formula: moved SLE in half-plane}	
\end{split}\end{equation}
which is a special case of 
(\ref{Formula: 3}) 
with 
$a_{\tilde t}=-4\xi \tilde t$
and
$b_{\tilde t}=4\sqrt{\kappa} e^{c_{\tilde t}}$.
In other words, the relation
(\ref{Formula: moved SLE in half-plane})
can be obtained from
(\ref{Formula: tilde G = e^c G})
with
$c_{\tilde t}=-2 \xi \tilde t +2\sqrt{\kappa} \tilde B_{\tilde t}$ under
the random time reparametrization
$\lambda_{\tilde t}= e^{4\xi \tilde t - 4\sqrt{\kappa} \tilde B_{\tilde t}}$.

It is remarkable that the subsurface $\tilde{\mathcal{I}}\subset \Dc$
defined in the half-plane chart as
\begin{equation*}
	\psi^{\HH} (\tilde{\mathcal{I}}) = \{ z\in\HH:\rm{Re}(z)>-\frac12\}
\end{equation*}
is invariant 
($G_t^{-1}(\tilde \Dc) \subset \tilde{\Dc}$)
if and only if 
$\xi\geq \kappa$. 
In order to see this, it is enough to calculate the real parts of 
\begin{equation*}\begin{split}
 &\tilde \delta^{\HH}(z) = \frac{2}{z} + \kappa + 2\xi z,\\
 &\tilde \sigma^{\HH}(z) = -\sqrt{\kappa}(1+2z),
\end{split}\end{equation*}
which are actually the horizontal components of the vector fields
at the boundary of 
$\psi^{\HH}(\tilde{\mathcal{I}})$ in $\HH$, 
$\{z\in\HH:~\rm{Re}(z)=-\frac12\}$,
\begin{equation*}\begin{split}
 \Re\left(\tilde \delta^{\HH}\left(-\frac12 + i h\right)\right) =& 
 \Re\left(\frac{2}{-\frac12 + i h} + \kappa + 2\xi\left(-\frac12 + i h\right)\right)
 =\\=&
 -\frac{1}{h^2+\frac14} + \kappa - \xi,\\
 \Re \left(\tilde \sigma^{\HH}\left(-\frac12 + i h\right) \right)=& 
 \Re \left(\sqrt{\kappa}\left(1+2\left(-\frac12 + i h\right)\right)\right)=0,\quad
 h>0.
\end{split}\end{equation*}
The first number is negative for all values of $h$ if and only if $\xi\geq \kappa$.

We remark that  the $r_{-1}$-transform has the invariant subsurface $\mathcal{I}:=r_{-1}(\tilde{\mathcal{I}}) \subset \Dc$ for the $(\delta,\sigma)$-SLE above, which is an upper half of the unit disk
\begin{equation*}
	\psi^{\HH}(\mathcal{I}) =\{ z\in\HH\colon |z|<1 \}
\end{equation*}

Similarly to the previous subsection it is reasonable to expect that the GFF coupled with this 
($\tilde\delta,\tilde \sigma$)-SLE is the same as in the chordal case, 
because it is supposed to induce the same random law of the flow lines. 
Indeed, the solution to 
(\ref{Formula: j^+ = L eta^+}),
with $\sigma$ and $\alpha$ as in the 5$^{\text{th}}$ string of the table, is
\begin{equation*}
 {\eta^+}^{\HH}(z) 
 = \frac{i}{\sqrt{\kappa}} \log z  
 + i\frac{\kappa-6}{2\sqrt{\kappa}} \arg (1-z) + C^+.
\end{equation*}
Thus,
\begin{equation*}
 \eta^{\HH}(z) 
 = \frac{-2}{\sqrt{\kappa}} \arg z 
 - \frac{\kappa-6}{\sqrt{\kappa}} \arg (1-z) + C.
\end{equation*}
After the $r_{-1}$-transform for $\tilde \delta$ and $\tilde \sigma$, we have
\begin{equation*}
 \tilde \eta^{\HH}(z) 
 = \frac{-2}{\sqrt{\kappa}} \arg z + C. 
 \end{equation*}
The last relation coincides with
(\ref{Formula: eta - chordal with drift in H})
with $\nu=0$. We remind that $\Gamma_0$ is invariant under M\"obius transforms, in particular, under $r_{-1}$.

\section{Coupling of GFF with Dirichlet-Neumann boundary conditions}
\label{Section: Coupling to GFF with Dirichlet-Neumann boundary condition}

We assume  in this chapter  that $\Gamma=\Gamma_{DN}$, see Example
\ref{Example: Gamma: Combined Dirichlet-Neumann boundary conditions}, 
which becomes 
\begin{equation}
 \Gamma_{DN}^{\mathbb{S}}(z,w)= -\frac12\log \frac{
  \rm{th}\frac{z-w}{4} \rm{th}\frac{\bar z-\bar w}{4}}
  {\rm{th}\frac{\bar z-w}{4} \rm{th}\frac{z -\bar w}{4}},\quad
  z,w\in \SSS:=\{z~:~ 0<\rm{Im}z<\pi \}.
 \label{Formula: G_DN^S = ...}
\end{equation} 
in the strip chart (\ref{Formula: tau_H S}).
It is exactly Green's function used in \cite{Kanga} (it is also a special case of \cite{Izyurov2010}).

The function $\Gamma_{DN}^{\SSS}$ satisfies the  boundary conditions
\begin{equation*}
 \left.\Gamma_{DN}^{\SSS}(x,w)\right|_{x\in \mathbb{R}}=0,\quad
 \left. \de_y \Gamma_{DN}^{\SSS}(x+iy,w)\right|_{x \in \mathbb{R},~y=\pi}=0,
\end{equation*}
the symmetry property, and
\begin{equation*}
 \Lc_{\sigma} \Gamma_{DN}(z,w)=0.
\end{equation*}

The coupling of GFF with this $\Gamma$ to the dipolar SLE  is geometrically motivated. We also require that both zeros of $\delta$ and $\sigma$ are at the same boundary points where $\Gamma_{DN}$ changes the boundary conditions from Dirichlet to Neumann.
In the strip chart these points are $\pm \infty$.

\begin{proposition}
Let the vector fields $\delta$ and $\sigma$ be as in
(\ref{Formula: delta and sigma dipolar in H}),
let $\Gamma=\Gamma_{DN}$, and let $\eta$ be a pre-pre-Schwarzian. Then the coupling is possible only for $\kappa=4$ and $\nu=0$. 
\label{Proposition: D-N coupling}
\end{proposition} 
 
\begin{proof}
We use Lemma 
\ref{Lemma: eta structure} 
in the strip chart. 
From 
(\ref{Formula: G_d^S = ...})
we obtain that
\begin{equation*}
 {\Gamma^{++}_{DN}}^{\SSS}(z,w) = -\frac12 \log \th \frac{z-w}{4},\quad
 {\Gamma^{+-}_{DN}}^{\SSS}(z,\bar w) = -\frac12 \log \th \frac{z-\bar w}{4},
\end{equation*}
and relations
(\ref{Formula: L Gamma^++ = 0, L Gamma^+- = 0})
hold. From
(\ref{Formula: L Gamma^++ + L sigma^+ L sigma^+ = e + e})
we find that
\begin{equation}\label{xxx}
 {j^+}^{\SSS}(z) = \frac{-i}{\sh\frac{z}{2}} + i \alpha,\quad \alpha\in\mathbb{R}.
\end{equation}
Substituting \eqref{xxx} in 
(\ref{Formula: delta/sigma j + mu [sigma,delta]/sigma + 1/2 L j = ibeta})
gives
\begin{equation*}
 -i \frac{ 
  \left( \beta \sqrt{\kappa} - \alpha \nu \right) \sh^2\frac{z}{2} 
  + \nu \sh\frac{z}{2}
  - 2 i \sqrt{\kappa} \mu
  \ch\frac{z}{2} \left( 4 \sh\frac{z}{2} \alpha + (\kappa -4) \right)
  }{2 \sqrt{\kappa} \sh^2\frac{z}{2} }
 \equiv 0,
\end{equation*}
which is possible only for $\kappa=4$, $\nu=0$, $\beta=0$ and $\mu=0$,
where the latter agrees with (\ref{Formula: mu = i pm (-4 pm k)/4/k}).
\end{proof}

From
(\ref{Formula: j^+ = L eta^+})
we obtain that
\begin{equation*}
 {\eta^+}^{\SSS}(z) = \frac{i}{2} \log \th \frac{z}{4}+C^+,
\end{equation*}
and
\begin{equation*}
 \eta^{\SSS}(z) = - 2\arg \,\rm{cth}\, \frac{z}{4} + C.
\end{equation*}

We also present here the relations in the half-plane chart 
\begin{equation*}
 \eta^{\HH}(z) 
 = - 2\arg \frac{z}{1+\sqrt{1-z^2}} + C,  
\end{equation*}
\begin{equation*}
 \Gamma_{DN}^{\HH}(z)
 = -\frac12 \log \frac{(z-w)(\bar z-\bar w)
  (1-\bar z w+\sqrt{1-\bar z^2}\sqrt{1-w^2})(1- z \bar w+\sqrt{1- z^2}\sqrt{1-\bar w^2})}
  {(\bar z-w)(z-\bar w)
  (1-zw+\sqrt{1-z^2}\sqrt{1-w^2})(1-\bar z \bar w+\sqrt{1-\bar z^2}\sqrt{1-\bar w^2})}.
\end{equation*}

The pre-pre-Schwarzian $\eta$ is scalar in this case, and in its chart-independent form is
\begin{equation*}
 \eta = - \arg 
 \frac{\frac{\sqrt{\kappa}}{2} (\delta - \frac{\nu}{\kappa}\sigma) + 
 \sqrt{\frac{\kappa}{4} (\delta - \frac{\nu}{\kappa}\sigma) ^2 - \sigma^2} }{\sigma}.
\end{equation*}

\section{Coupling of twisted GFF}
\label{Section: Coupling to twisted GFF}

This model is  similar to the previous one.  As it will be shown below,  it is enough at the algebraic level to replace formally  all hyperbolic functions in the dipolar case in the strip chart $\SSS$ by the corresponding trigonometric functions in order to obtain the relations for the radial SLE in the logarithmic chart $\LL$. But at the analytic level, we have to consider the correlation functions which are doubly defined on $\Dc$ and change their sign after the analytic continuation about the center point. 
This construction was considered before as we were informed by  Num-Gyu Kang, \cite{Kang_preparation}.

We have to generalize slightly the general approach similarly to Section 
\ref{Section: Radial SLE with drift}
  considering the double cover   
$\Dc_b^{\pm}$ instead of the infinitely ramified cover of $\Dc\setminus\{b\}$.
Let us define the space $\Hc_s[\Dc_b^{\pm}]$
of test functions $f\colon \Dc_b^{\pm}\map\mathbb{R}$
as in  Section \ref{Section: Test function} 
with $\phi^{\LL}(z)\equiv 0$ and
with an extra condition
$f(z_1)=-f(z_2)$, where $z_1$ and $z_2$ are two points of $\Dc^{\pm}_b$ corresponding to the same point of $\Dc\setminus\{b\}$. Thus, in the logarithmic chart, we have
\begin{equation}
 f^{\LL}(z)=f^{\LL}(z+4\pi k) = - f^{\LL}(z+ 2\pi k),\quad k\in \mathbb{Z},\quad z\in\HH. 
 \label{Formula: f = -f  = f}
\end{equation}
Such  functions are $4\pi$-periodic and $2\pi$-antiperiodic. In particular,
$f^{\LL}$ is not of compact support, but we require in addition that 
\begin{equation*}
	\sup\Im(\{z\in \HH\colon f^{\LL}(z)\neq 0\}) < \infty
\end{equation*}
in order to maintain compatibility.
In some sense, the `value' of $\Phi_{tv}$ changes its sign after  horizontal translation by $\pi$.

The {\it twisted Gaussian free field} $\Phi_{\text{tw}}$ is defined similarly to the usual one but taking values in 
$\Dc^{\pm}_b{}'$. 

In this section, we define $\Gamma$ by
\begin{equation}
 \Gamma_{\text{tw}}^{\LL}(z,w)= -\frac12\log \frac{
  \rm{tg}\frac{z-w}{4} \rm{tg}\frac{\bar z-\bar w}{4}}
  {\rm{tg}\frac{\bar z-w}{4} \rm{tg}\frac{z -\bar w}{4}},\quad
  z,w\in \HH.
 \label{Formula: G_d^S = ...}
\end{equation}
in the logarithmic chart. Observe that 
\begin{equation*}
 \Gamma_{\text{tw}}^{\LL}(z,w) = \Gamma_{\text{tw}}^{\LL}(z+4\pi k,w) =
  -\Gamma_{\text{tw}}^{\LL}(z+2\pi k,w),\quad k\in \mathbb{Z}.
\end{equation*}

In the unit disk chart the covariance $\Gamma_{\text{tw}}^{\D}$ admits  the form
\begin{equation*}
 \Gamma^{\D}_{\text{tw}}(z,w) = -\frac12\log \frac
  {(\sqrt{z}-\sqrt{w})(\sqrt{\bar z}-\sqrt{\bar w})(\sqrt{z}+\sqrt{\bar w})(\sqrt{\bar z}-\sqrt{w})}
  {(\sqrt{z}+\sqrt{w})(\sqrt{\bar z}+\sqrt{\bar w})(\sqrt{z}-\sqrt{\bar w})(\sqrt{\bar z}+\sqrt{w})}, 
\end{equation*}
or in the half-plane chart, 
\begin{equation*}
 \Gamma_{\text{tw}}^{\HH}(z)
 = -\frac12 \log \frac{(z-w)(\bar z-\bar w)
  (1+\bar z w+\sqrt{1+\bar z^2}\sqrt{1+w^2})(1+ z \bar w+\sqrt{1+ z^2}\sqrt{1+\bar w^2})}
  {(\bar z-w)(z-\bar w)
  (1+zw+\sqrt{1+z^2}\sqrt{1+w^2})(1+\bar z \bar w+\sqrt{1+\bar z^2}\sqrt{1+\bar w^2})}.
\end{equation*}
It is doubly defined because of the square root, and the analytic continuation about the center changes its sign. 

The covariance $\Gamma_{\text{tw}}^{\LL}$ satisfies  the Dirichlet boundary conditions and tends to zero as one of  the variables tends to the center point $b$ (or $\infty$ in the $\LL$ chart) 
\begin{equation*}
 \left.\Gamma_{\text{tw}}^{\LL}(x,w)\right|_{x\in \mathbb{R}}=0,\quad
 \lim\limits_{y\map +\infty} \Gamma_{\text{tw}}^{\LL}(x+iy,w) = 0,\quad x\in \mathbb{R}, \quad w\in \HH;
\end{equation*}
\begin{equation*}
 \left.\Gamma_{\text{tw}}^{\D}(z,w)\right|_{|z| = 1 }=0,\quad
 \lim\limits_{z\map 0} \Gamma_{\text{tw}}^{\D}(z,w) = 0, \quad w\in \D.
\end{equation*}

The $\sigma$-symmetry property
\begin{equation*}
 \Lc_{\sigma_r} \Gamma_{\text{tw}}(z,w)=0
\end{equation*}
holds.

As we will see below, $\eta$ also possesses  property
(\ref{Formula: f = -f  = f}). Thus, the construction of the level (flow) lines can be performed for  both layers simultaneously and the lines will be identical. In particular, this means that the line can turn around the central point and appears in the second layer but can not intersect itself. This agrees with the property of the SLE slit which evoids self-intersections.

Similarly to the dipolar case in the previous section the following proposition can be proved.
\begin{proposition}
Let the vector fields $\delta$ and $\sigma$ be as in
(\ref{Formula: delta nad sigam radial in H}), let
$\Gamma=\Gamma_{\text{tw}}$, and let $\eta$ be a pre-pre-Schwarzian. Then the  coupling is possible only
 for $\kappa=4$ and $\nu=0$. 
\end{proposition} 
 
The {\it proof} in the logarithmic chart actually repeats  the proof of Proposition 
\ref{Proposition: D-N coupling}.

We give here the expressions for $\eta$ in the logarithmic, unit-disk and half-plane charts:
\begin{equation*}
 \eta^{\LL}(z) = - 2\arg \tg \frac{4}{z} + C.
\end{equation*}
\begin{equation*}
 {\eta}^{\D}(z) 
 = -2 \arg \frac{1-\sqrt{z}}{1+\sqrt{z}} + C 
 = 4\Im \,\rm{arctgh}\sqrt{z} + C.
\end{equation*}  
\begin{equation*}
 \eta^{\HH}(z) 
 = - 2\arg \frac{z}{1+\sqrt{1+z^2}} + C.
\end{equation*}
From this relation it is clear that $\eta$ is  antiperiodic. 

The pre-pre-Schwarzian $\eta$ is scalar in this case and its chart-independent form becomes
\begin{equation*}
 \eta = - \frac{2}{\sqrt{\kappa}} \arg 
 \frac{\frac{\sqrt{\kappa}}{2} (\delta - \frac{\nu}{\kappa}\sigma) + 
 \sqrt{\frac{\kappa}{4} (\delta - \frac{\nu}{\kappa}\sigma) ^2 + \sigma^2} }{\sigma}.
\end{equation*}

We defined  the linear functional $\Phi_{\text{tw}}$ on the space of antiperiodic functions before, however, such functional can be also defined on the space of  functions with bounded support in the logarithmic chart. Thus, we can use the same space $\Hc_s$ as in Section
\ref{Section: Radial SLE with drift}.

\section*{Perspectives}

\begin{enumerate}[1.]

\item
The coupling with the reverse $(\delta,\sigma)$-SLE can also be established using the same classification as in Table \ref{Table: some simple cases}.

\item
We did not prove in this paper but our experience shows that we listed all possible ways of coupling of GFF  with $(\delta,\sigma)$-SLE if we assume that $\Gamma$ transforms as a scalar, 
see (\ref{Formula: G B(z,w) = B(G(z),G(w))}), and that $\eta$ is a pre-pre-Schwarzian. It would be useful to prove this.

\item 
The pre-pre-Schwarzian rule 
(\ref{Formula: G eta(z) = eta(G(z)) + mu log G'(z) + ...})
is motivated by the local geometry of SLE curves \cite{Sheffield2010}. 
In principle, one can consider alternative rules. 
Moreover, the scalar behaviour of $\Gamma$
can also be relaxed because the harmonic part $H^{\psi}(z,w)$ can transform in many ways. 
Such more general coupling is intrinsic and can be thought of as a generalization of the coupling 
in Sections 
\ref{Section: Coupling to GFF with Dirichlet-Neumann boundary condition}
and
\ref{Section: Coupling to twisted GFF}
for arbitrary $\kappa$.  

\item
We considered only the simplest case of one Gaussian free field. It would be interesting to examine tuples of $\Phi_i,~i=1,2,\dotso n$ which transform into non-trivial combinations
$\tilde \Phi_i = G_i[\Phi_1,\Phi_2,\dotso \Phi_n]$ under  conformal transforms $G$.

\item
The Bochner-Minols Theorem
\ref{Theorem: Bochner-Minols}
suggests to consider not only free fields, but for example, some polynomial combinations 
in the exponential of
(\ref{Formula: GFF chracteristic function 2}).
In particular, the quartic functional corresponds to conformal field theories related 
to 2-to-2 scattering of particles in dimension two. 
\end{enumerate}

\medskip

\noindent
{\bf Acknowledgement.}  The authors gratefully acknowledge many useful and inspiring conversations with Nam-Gyu Kang and Georgy Ivanov.

\appendix

\section{Technical remarks}

In this appendix section, we prove some technical propositions needed in the proof of Theorem 
\ref{Theorem: The coupling theorem}

Consider an It\^o process $\{X_t\}_{t\in[0,+\infty)}$ such that
\begin{equation*}
	\dI X_t = a_t dt + b_t \dI B_t,\quad t\in[0,+\infty),
\end{equation*}
for some continuous processes 
$\{a_t\}_{t\in[0,+\infty)}$
and
$\{b_t\}_{t\in[0,+\infty)}$.
We denote by 
$\{X_{t\wedge T}\}_{t\in[0,+\infty)}$ 
the stopped process by a stopping time $T$. It satisfies the following SDE:
\begin{equation*}
	\dI X_{t\wedge T} = \theta(T-t) a_t dt + \theta(T-t) b_t \dI B_t,\quad
	t\in[0,+\infty),
\end{equation*}
where
\begin{equation*}
	\theta(t) :=
	\begin{cases} 
		0 	&\mbox{if } t\leq 0 \\
		1		&\mbox{if } t>0.
	\end{cases}
\end{equation*}
If $\{X_t\}_{t\in[0,+\infty)}$ is a local martingale ($a_t=0$,
$t\in[0,+\infty)$), then 
$\{X_{t\wedge T}\}_{t\in[0,+\infty)}$ 
is also a local martingale. 

We consider below the stopped processes 
$Y(G_{t\wedge T}) \}_{t\in[0,+\infty)}$ 
instead of 
$\{Y(G_t)\}_{t\in[0,+\infty)}$. 
for some functions 
$Y:\mathscr{G}\map \mathbb{R}$,
and the corresponding It\^o SDE. In order to make the relations
less cluttered, we usually drop the terms `$...\wedge T[f]$' and $\theta(T-t)$.
However, in the places where it is essential to remember  them, e.g., the
proof of Theorem 
\ref{Theorem: The coupling theorem},
we specify the stopping times explicitly.

Define the \emph{diffusion operator}
\index{diffusion operator} 
\index{$\Ac$}
\begin{equation}
  \Ac: = \Lc_{\delta} + \frac12 \Lc_{\sigma}^2.
  \label{Formula: A = L + 1/2 L^2}
\end{equation}
and consider how a regular pre-pre-Schwarzian $\eta$ changes under the random
evolution $G_t$. 
We also define the stopping time $T(x)$, $x\in\bar\Dc$ analogously
to 
\eqref{Formula: T[f] = ...}
using a neighborhood $U(x)$ of single point $x\in\Dc$. 
The functions  
${G^{-1}_t}_* \eta^{\psi}(z)$ 
and 
${G^{-1}_t}_* \Gamma^{\psi}(z,w)$ 
are defined by 
\eqref{Formula: G eta(z) = eta(G(z)) + mu log G'(z) + ...}
and
\eqref{Formula: G B(z,w) = B(G(z),G(w))} 
until the stopping times $T(z)$ and $\min(T(z),T(w))$ respectively.

\begin{proposition} 
Let  $\{G_t\}_{t\in[0,+\infty)}$ be a ($\delta,\sigma$)-SLE.
\begin{enumerate}[1.]
\item 
Let $\eta$ be a regular pre-pre-Schwarzian such that the Lie dereivatives
$\Lc_{\sigma} \eta$, $\Lc_{\delta} \eta$,
and
$\Lc_{\sigma}^2 \eta$ are well-defined. Then
\begin{equation}
  \dI {G^{-1}_{t}}_* \eta^{\psi}(z) = 
  {G_{t}^{-1}}_* \left( \Ac \eta^{\psi}(z) ~dt + \Lc_{\sigma}
  \eta^{\psi}(z) \dI B_t \right).
  \label{Formula: d G^-1 eta  = G A^+ eta dt + G L eta dB}
\end{equation}
\item
Let $\Gamma$ be a scalar 
bilinear functional 
(\eqref{Formula: G B(z,w) = B(G(z),G(w))} holds) 
such that the Lie derivatives
$\Lc_{\sigma} \Gamma$, 
$\Lc_{\delta} \Gamma$,
and
$\Lc_{\sigma}^2 \Gamma$ 
are well-defined.
Then
\begin{equation}
  \dI {G_{t}^{-1}}_* \Gamma^{\psi}(z,w) = 
  {G_{t}^{-1}}_* \left( \Ac \Gamma^{\psi}(z,w) ~dt 
  	+ \Lc_{\sigma} \Gamma^{\psi}(z,w) \dI B_t \right).
  \label{Formula: d G Gamma  = G A Gamma dt + G L Gamma dB}
\end{equation}
\end{enumerate}
\label{Proposition: dI G eta = G A eta dt + G L eta dB}
\end{proposition}

This can be proved by the direct calculation but we show a more
preferable way, which is valid not only for pre-pre-Scwarzians but, for
instance, for vector fields, and even more generally, for assignments
whose transformation rules contain an arbitrary finite number of derivatives at
a finite number of points. To this end let us prove the following lemma.
\begin{lemma}
\label{lemma2}
Let $X^i(t)$ ($i=1,2,\dotso,n$) be a finite collection of stochastic processes
defined by the following system of equations in the Stratonovich form
\begin{equation}\begin{split}
	d^{\rm{S}} X_t^i = \alpha^i(X_t) dt + \beta^i(X_t) d^{\rm{S}}B_t,
	\label{Formula: dX = a X dt + b X dB}
\end{split}\end{equation}
for some fixed functions $\alpha,\beta\colon\, \mathbb{R}^n\map\mathbb{R}^n$.
Let us define $Y_s^i$, $Z_s^i$ as the solution to the initial value problem
\begin{equation}\begin{split}
&\dot Y_s^i = \alpha^i(Y_s), \quad Y_0^i=0, \\
&\dot Z_s^i = \beta^i(Z_s), \quad Z_0^i=0,
\label{Formula: Y_s = ..., Z_s = ...}
\end{split}\end{equation}
in some neighbourhood of $s=0$. Let also
$F\colon\mathbb{R}^n\map\C$
be a twice-differentiable function.
Then,  It\^{o}'s differential of $F(X_t)$ can be written in the following form
\begin{equation}
d^{\rm{It\^{o}}}F(X_t) =
\left.
\frac{\de}{\de s} F(X_t+Y_s) dt +
\frac{\de}{\de s} F(X_t+Z_s) d^{\rm{It\^{o}}}B_t +
\frac12\frac{\de^2}{\de s^2} F(X_t+Z_s) dt
\right|_{s=0}.
\label{Formula: It\^{o} derivative lemma}
\end{equation}
\label{Lemma: It\^{o} derivative lemma}
\end{lemma}

\begin{proof}
The direct calculation of the right-hand side of
(\ref{Formula: It\^{o} derivative lemma})
gives
\begin{equation*}\begin{split}
  &
  F'_i(X_t) \left(\alpha^i(X_t) + \frac12 {\beta'_j}^i(X_t)\beta^j(X_t) \right) dt +
  F'_i(X_t) \beta^i(X_t) d^{\rm{It\^{o}}}B_t +
  \frac12 F''_{ij}(X_t) \beta^i(X_t)\beta^j(X_t) dt,
  \end{split}\end{equation*}
which is indeed  It\^{o}'s differential of $F(X_t)$. We employed summation
over repeated indices and used the notation 
$F'_i(X):=\frac{\de}{\de X^i}F(X)$.
\end{proof}

\medskip

\noindent
{\it Proof of  Proposition \ref{Proposition: dI G eta = G A eta dt + G L eta dB}.}
We use the lemma above.
Let $n=4$, and let  us define a vector valued linear map
$\{ \cdot \}$ for an analytic function $x(z)$ as
\begin{equation}
 \{ x(z) \} : = \{ \Re x(z),\Im x(z),\Re x'(z),\Im x'(z) \}.
 \label{Formula: [x] :=  [Re,Im,Re',Im']}
\end{equation}
For example,
\begin{equation*}
 X_t := \{ G_t^{\psi}(z) \} =
  \{ \Re G_t^{\psi}(z), \Im G_t^{\psi}(z), \Re {G_t^{\psi}}'(z), \Im {G_t^{\psi}}'(z) \}.
\end{equation*}
From
\eqref{Formula: Slit hol stoch flow Strat}
we have
\begin{equation*}
 \alpha(X_t) = \{ \delta^{\psi}(G_t^{\psi}(z)) \},\quad
 \beta(X_t) = \{ \sigma^{\psi}(G_t^{\psi}(z)) \}.
\end{equation*}
Let also
\begin{equation*}\begin{split}
	F(X_t) =& F(\{{G_{t}^{\psi}}(z) \}):=
	{G_{t}^{-1}}_* \eta^{\psi}(z) 
	=\\=& 
	\eta^{\psi}(G_{t}^{\psi}(z)) 
	+ \mu \log {G_{t}^{\psi}}'(z) 
	+ \mu^* \log \overline{{G_{t}^{\psi}}'(z)}.
\end{split}\end{equation*}
Then
\begin{equation*}
 Y_s = \{ H_s[\delta]^{\psi}(z)-z \},\quad
 Z_s = \{ H_s[\sigma]^{\psi}(z)-z \}
\end{equation*}
due to
(\ref{Formula: d H = sigma H ds}),
(\ref{Formula: Slit hol stoch flow Strat}),
(\ref{Formula: dX = a X dt + b X dB}), and
(\ref{Formula: Y_s = ..., Z_s = ...}).

Now we can use Lemma~\ref{lemma2} in order to obtain (\ref{Formula: d G^-1 eta  = G A^+ eta dt + G L eta dB}) for $t=0$:
\begin{equation}\begin{split}
 &\left. \dI {G_{t}^{-1}}_* \eta^{\psi}(z) \right|_{t=0} =
 \left. \dI F[X_t] \right|_{t=0} 
 =\\=&
 \text{ (right-hand side of (\ref{Formula: It\^{o} derivative lemma}) with $t=0$
 ) }.
 \label{Formula: d G eta |_t=0 = ...}
\end{split}\end{equation}
But
\begin{equation*}\begin{split}
 &\left. \frac{\de}{\de s} F(X_t+Y_s) \right|_{s=0,t=0} =
 \left. \frac{\de}{\de s} F( \{z+ H_s[\delta]^{\psi}(z) - z \}) \right|_{s=0}
 =\\=&
 \left. \frac{\de}{\de s} F( \{H_s[\delta]^{\psi}(z)\}) \right|_{s=0}
 =\left. \frac{\de}{\de s} \{H_s[\delta]^{-1}_* \eta^{\psi}(z) \right|_{s=0}
 = \Lc_{\delta} \eta^{\psi}(z).
\end{split}\end{equation*}
A similar observation for other terms in
(\ref{Formula: d G eta |_t=0 = ...})
implies that
\begin{equation*}\begin{split}
 \left. \dI {G_t^{-1}}_* \eta^{\psi}(z) \right|_{t=0} =&
 \Lc_{\delta} \eta^{\psi}(z) dt
 + \Lc_{\sigma} \eta^{\psi}(z) \dI B_t
 + \frac12 \Lc_{\sigma}^2 \eta^{\psi}(z) dt
 =\\=&
 \Ac \eta^{\psi}(z) dt + \Lc_{\sigma} \eta^{\psi}(z) \dI B_t.
\end{split}\end{equation*}
For $t>0$ we  conclude that
\begin{equation*}\begin{split}
 &\dI {G_t^{-1}}_* \eta^{\psi}(z) =
 \dI \left( \tilde G_{t-t_0} \circ G_{t_0} \right)^{-1}_*  \eta^{\psi}(z) =
 \dI \left. {G_{t_0}^{-1}}_* { \tilde G_{t-t_0}^{-1}\,  }_* \eta^{\psi}(z) \right|_{t_0=t}
 =\\=&
 \left. {G_{t}^{-1}}_* \dI  { \tilde G_{s}^{-1}\,  }_* \eta^{\psi}(z) \right|_{s=0} =
 {G_{t}^{-1}}_* \left( \Ac \eta^{\psi}(z) dt + \Lc_{\sigma} \eta^{\psi}(z) \dI B_t \right).
\end{split}\end{equation*}

The proof of 
\ref{Formula: d G Gamma  = G A Gamma dt + G L Gamma dB}
is analogous. The only difference is that we do not have
the pre-pre-Schwarzian terms with the derivatives but there are two points $z$
and $w$. We can assume
\begin{equation*}
  \{ x \} : = \{ \Re x(z),\Im x(z),\Re x(w),\Im x(w) \}
\end{equation*}
instead of
(\ref{Formula: [x] :=  [Re,Im,Re',Im']})
and the remaining part of the proof is the same.
\quad\qed
\medskip

We will obtain below the It\^o differential of
${G_t^{-1}}_*\eta[f]$ 
and 
${G_t^{-1}}_* \Gamma [f,g]$ 
for 
($\delta,\sigma$)-SLE
$\{G_t\}_{t\in[0,+\infty)}$
and
$f,g\in\Hc$.  
To this end we need the It\^o formula for nonlinear functionals over $\Hc$.
For linear functionals on the Schwartz space this has been shown in
\cite{Krylov2009}.
However, the authors are not aware of similar results for nonlinear
functionals.
The following propositions are special cases required for this paper. They are consequences of the proposition above, the classical It\^o 
formula, and the stochastic Fubini theorem.

\begin{proposition} 
Under the conditions of Proposition
\ref{Proposition: dI G eta = G A eta dt + G L eta dB}
the following holds:
\begin{enumerate}[1.]
\item
The It\^{o} differential is
interchangeable with the integration over $\Dc$. Namely,
\begin{equation}\begin{split}
  &\dI \int\limits_{\psi(\supp f)} {G^{-1}_{t}}_*
  \eta^{\psi}(z) f^{\psi}(z) l(dz) 
  =\\=&
  \int\limits_{\psi(\supp f)} {G^{-1}_{t}}_* \Ac
  \eta^{\psi}(z) f^{\psi}(z) l(dz)\, dt 
  +\\+& 
  \int\limits_{\psi(\supp f)} 
  	{G^{-1}_{t}}_* \Lc_{\sigma} \eta^{\psi}(z) f^{\psi}(z)
  	 l(dz) \,\dI B_t.
  \label{Formula: dI int G^-1 eta = int dI G^-1 eta}
\end{split}\end{equation}
An equivalent shorter formulation is
\begin{equation}\begin{split}
  \dI {G_{t }^{-1}}_* \eta[f]
  = {G^{-1}_t}_* \Ac \eta[f] dt +
  {G^{-1}_t}_* \Lc_{\sigma} \eta[f] \dI B_t
  \label{Formula: dI G^-1 eta[f] = G^-1 A eta[f] dt + G^-1 L eta[f] dB}.
\end{split}\end{equation}
\item 
The It\^{o} differential is interchangeable with the double integration
over $\Dc$, namely,
\begin{equation}\begin{split}
  &\dI \int\limits_{\psi(\supp f)} 
  \int\limits_{\psi(\supp f)} {G^{-1}_{t}}_* \Gamma (x,y)
  f^{\psi}(x)f^{\psi}(y)  l(dx) l(dy) =\\=&
  \int\limits_{\psi(\supp f)} \int\limits_{\psi(\supp f)}
  {G^{-1}_{t}}_* \Ac \Gamma (x,y)
  f^{\psi}(x) f^{\psi}(y) l(dx) l(dy)\, dt
  +\\+&
  \int\limits_{\psi(\supp f)} \int\limits_{\psi(\supp f)}
  {G^{-1}_{t}}_* \Lc_{\sigma} \Gamma(x,y)
  f^{\psi}(x) f^{\psi}(y) l(dx) l(dy)\, \dI B_t.
  \label{Formula: dI int G^-1 Gamma = int dI G^-1 Gamma}
\end{split}\end{equation}
 An equivalent  shorter formulation is
 \begin{equation*}\begin{split}
  \dI {G_t^{-1}}_* \Gamma[f,g]
  =
  {G^{-1}_t}_* \Ac \Gamma[f,g]\, dt +
  {G^{-1}_t}_* \Lc_{\sigma} \Gamma[f,g]\, \dI B_t.
 \end{split}\end{equation*}
\end{enumerate}
\label{Proposition: dI int eta = int dI eta}
\end{proposition}

\begin{proof}
The relation
(\ref{Formula: dI int G^-1 eta = int dI G^-1 eta})
in the integral form becomes
\begin{equation*}\begin{split}
  &\int\limits_{\psi(\supp f)} 
  	{G^{-1}_t}_* \eta^{\psi} (z) f^{\psi}(z) l(dz)
  =\eta[f]  
  +\\+&
  \int\limits_0^{t}
   \int\limits_{\psi(\supp f)} 
   {G^{-1}_{\tau}}_* \Ac \eta^{\psi} (z)  f^{\psi}(z) l(dz) d{\tau} +
  \int\limits_0^{t}
  \int\limits_{\psi(\supp f)} 
  	{G^{-1}_{\tau}}_* \Lc_{\sigma} \eta^{\psi}(z) f^{\psi}(z) l(dz) \dI
  B_{\tau}
\end{split}\end{equation*}
The order of the It\^{o} and the Lebesgue integrals can be changed using the
stochastic Fubini theorem, see, for example \cite{Protter2004a}.
It is enough now to use
(\ref{Formula: d G^-1 eta  = G A^+ eta dt + G L eta dB})  to obtain
(\ref{Formula: dI int G^-1 eta = int dI G^-1 eta}).

The proof of 
\ref{Formula: dI G^-1 eta[f] = G^-1 A eta[f] dt + G^-1 L eta[f] dB}
is analogous.
\end{proof}

\begin{proposition}
\label{Proposition: G chi[f] is an It\^{o} process}
Let
\begin{equation*}
 \hat\phi[f] = \exp \left( W[f] \right),\quad
 W[f] := \frac12 \Gamma[f,f] + \eta[f].
\end{equation*}
Then
${G_t^{-1}}_* \hat\phi[f]$  is an It\^{o} process defined by the integral
\begin{equation}\begin{split}
 &{G_t^{-1}}_* \hat\phi[f] 
 =\\=&
 \int\limits_0^t
 \exp \left( {G_{\tau}^{-1}}_* W[f] \right)
 \left(
  {G_{\tau}^{-1}}_* \Ac W[f] d\tau +
  {G_{\tau}^{-1}}_* \Lc_{\sigma} W[f] \dI B_{\tau} +
  \frac12 \left( {G_{\tau}^{-1}}_* \Lc_{\sigma} W[f] \right)^2 d\tau
 \right).
 \label{Formula: G chi = chi + int ...}
\end{split}\end{equation}
\end{proposition}

\begin{proof}
The stochastic process
${G_t^{-1}}_* W^{\psi}[f]$
has the integral form
\begin{equation*}\begin{split}
 &{G_t^{-1}}_* W^{\psi}[f] = 
 \frac12 {G_t^{-1}}_* \Gamma^{\psi}[f,f] + {G_t^{-1}}_* \eta^{\psi}[f]
 =\\=&
 \int\limits_{\psi(\supp f)} \int\limits_{\psi(\supp f)}
  {G_t^{-1}}_* \Gamma^{\psi}(z,w) f^{\psi}(z) f^{\psi}(w) l(dz) l(dw)
 +
 \int\limits_{\psi(\supp f)} {G_t^{-1}}_* \eta^{\psi}(z) f^{\psi}(z) l(dz).
\end{split}\end{equation*}
due to  Proposition. 
\ref{Proposition: dI int eta = int dI eta}. 
In  terms of the It\^{o} differentials it becomes
\begin{equation*}
 \dI {G_t^{-1}}_* W^{\psi}[f] =
 {G_{t}^{-1}}_* \Ac W^{\psi}[f] dt +
 {G_{t}^{-1}}_* \Lc_{\sigma} W^{\psi}[f] \dI B_{t}.
\end{equation*}
In order to obtain the exponential function we can just use It\^{o}'s lemma
\begin{equation*}\begin{split}
 &\dI {G_t^{-1}}_* \exp \left( W^{\psi}[f] \right) =
 \dI \exp \left( {G_t^{-1}}_* W^{\psi}[f] \right)
 =\\=&
 \exp \left( {G_{t}^{-1}}_* W^{\psi}[f] \right)
 \left(
  {G_{t}^{-1}}_* \Ac W^{\psi}[f] dt
  + {G_{t}^{-1}}_* \Lc_{\sigma} W^{\psi}[f] \dI B_t
  + \frac12 \left( {G_{t}^{-1}}_* \Lc_{\sigma} W^{\psi}[f] \right)^2 dt
 \right).
\end{split}\end{equation*}
\end{proof}

\section{Some formulas from stochastic calculus}
\label{Appendix: Some relations from stochastic calculus}

We refer to 
\cite{Gardiner1982}
\cite{Oksendal2003}, and
\cite{Protter2004a}
for the definitions and properties of the It\^{o} and Stratonovich calculus and
use the following relation between the  It\^{o} and Stratonovich integrals
\begin{equation*}
 \int\limits_{0}^T F(x_t,t) \dS B_t =
 \int\limits_{0}^T F(x_t,t) \dI B_t 
 + \frac12 \int\limits_{0}^T b_t \de_1 F(x_t,t) dt.
\end{equation*}
The latter item   can also  be expressed in terms of the covariance
\begin{equation}
 \int\limits_{0}^T b_t \de_1 F(x_t,t) dt = \langle F(x_T) ,B_t \rangle.
\end{equation}

In order to obtain 
(\ref{Formula: Slit hol stoch flow It\^{o}})
from
(\ref{Formula: Slit hol stoch flow Strat}), let us
assume 
\begin{equation}
 x_t:=G_t(z),\quad
 b_t:=\sigma(G_t(z)),\quad
 F(x_t,t):=\sigma(x_t)=\sigma(G_t(z)).
\end{equation}
Then
\begin{equation}
 \de_1F(x_t,t) = \sigma'(G_t(z)),
\end{equation}
and
\begin{equation}
 \int\limits_0^T \sigma(G_t(z)) \dS B_t =
 \int\limits_0^T \sigma(G_t(z)) \dI B_t +
 \frac12 \int\limits_0^T \sigma(G_t(z))\sigma'(G_t(z)) dt.
\end{equation}
It is enough now to add $\int\limits_0^T \delta(G_t(z)) dt$ to both parts to obtain the right-hand sides of the integral forms of 
(\ref{Formula: Slit hol stoch flow It\^{o}})
and
(\ref{Formula: Slit hol stoch flow Strat}).

We also use in this paper that 
\begin{equation*}
 \tilde B_{\tilde T} := 
 \int\limits_0^{\tilde T} \dot \lambda^{\frac12}_{\tilde t} \dI B_{\lambda_{\tilde t}}
 = \int\limits_0^{\lambda_{\tilde T}} \dot \lambda^{\frac12}_{\lambda^{-1}_t} \dI B_{t}
\end{equation*}
has the same law as $B_{\tilde T}$ for any monotone and continuously differentiable function 
$\lambda\colon [0,\tilde T]\map[0,T]$.
In differential form this relation becomes
\begin{equation}
 \dI \tilde B_{\tilde t} = \dot \lambda^{\frac12}_{\tilde t} \dI B_{\lambda_{\tilde t}}.
 \label{Formula: dB = lambda dB}
\end{equation}
We need to reformulate relation 
(\ref{Formula: dB = lambda dB})
in the Stratonovich form.
Let now $\lambda$ satisfy
\begin{equation}
 \dS \dot \lambda_{\tilde t} 
 = a_{\tilde t} d \tilde t + b_{\tilde t} \dS \tilde B_{\tilde t}.
 \label{Formula: d lambda = a dt + b dB}
\end{equation}
\begin{equation*}\begin{split}
 \int\limits_0^{\tilde T}  \dS \tilde B_{\tilde t}
 =&\int\limits_0^{\tilde T}  \dI \tilde B_{\tilde t}
 = \int\limits_0^{\lambda_{\tilde T}} \dot \lambda^{\frac12}_{\lambda^{-1}_t} \dI B_{t}
 = \int\limits_0^{\lambda_{\tilde T}} \dot \lambda^{\frac12}_{\lambda^{-1}_t} \dS B_{t}
 - \frac12
 \langle \dot \lambda^{\frac12}_{\tilde T}, B_{\lambda_{\tilde T}}  \rangle
 =\\=& 
 \int\limits_0^{\tilde T} \dot \lambda^{\frac12}_{\tilde t} \dS B_{\lambda_{\tilde t}}
 - \frac12
 \langle 
  \dot \lambda^{\frac12}_{\tilde T}, 
  \int\limits_0^{\tilde T} \lambda^{-\frac12}_{\tilde t} d\tilde B_{\tilde t}
 \rangle
 =
 \int\limits_0^{\tilde T} \dot \lambda^{\frac12}_{\tilde t} \dS B_{\lambda_{\tilde t}}
 - \frac12 \int\limits_0^{\tilde T} 
 \frac12 \dot \lambda^{-\frac12}_{\tilde t} b_{\tilde t} 
 \dot \lambda^{-\frac12}_{\tilde t} d \tilde t
 =\\=&
  \int\limits_0^{\tilde T} \dot \lambda^{\frac12}_{\tilde t} \dS B_{\lambda_{\tilde t}}
 - \frac14 \int\limits_0^{\tilde T} 
 \frac{b_{\tilde t}}{\dot \lambda_{\tilde t} } d \tilde t.
\end{split}\end{equation*}
We conclude that
\begin{equation}
 \dS \tilde B_{\tilde t} = 
 \dot \lambda_{\tilde t}^{\frac12} \dS B_{\lambda_{\tilde t}} - 
 \frac14 \frac{b_{\tilde t}}{\dot \lambda_{\tilde t} } d \tilde t.
 \label{Folrmula: dB = lambda dB - 1/4 b/lambda dt}
\end{equation}

\end{document}